\documentclass[12pt,letterpaper]{article}

\usepackage{fullpage,amsfonts,amsmath,amsthm,amssymb,mathrsfs,graphicx,enumitem}

\usepackage{verbatim,url}

\usepackage[margin=1in]{geometry}
\makeatletter
\g@addto@macro\normalsize{%
  \setlength\abovedisplayskip{3pt}
  \setlength\belowdisplayskip{4pt}
  \setlength\abovedisplayshortskip{3pt}
  \setlength\belowdisplayshortskip{4pt}
}
\makeatother

\makeatletter
\renewcommand*{\section}{%
\@startsection {section}{1}{\z@}%
  {-1.8ex \@plus -0.7ex \@minus -.1ex}%
  {1.2ex \@plus.1ex}%
  {\normalfont\Large\bfseries}%
}

\makeatletter
\def\thm@space@setup{%
  \thm@preskip=0.1cm
  \thm@postskip=\thm@preskip 
}
\makeatother

 \newcommand{\lab}[1]{\label{#1}}                

 \usepackage[usenames,dvipsnames]{color}

\newcommand{\remove}[1]{}
\newcommand{\jt}[1]{{#1}}

\newcommand\eqn[1]{(\ref{#1})}

\newcommand{\be}{\begin{equation}}
\newcommand{\bel}[1]{\begin{equation}\lab{#1}\ }
\newcommand{\ee}{\end{equation}}
\newcommand{\bea}{\begin{eqnarray}}
\newcommand{\eea}{\end{eqnarray}}
\newcommand{\bean}{\begin{eqnarray*}}
\newcommand{\eean}{\end{eqnarray*}}

\newtheorem{thm}{Theorem}

\newtheorem{lemma}[thm]{Lemma}
\newtheorem{definition}[thm]{Definition}

\def\proof{\noindent{\bf Proof.\ }  }
\def\qed{~~\vrule height8pt width4pt depth0pt}

\def\heavy{{\mathcal H}}
\def\light{{\mathcal L}}

\def\ex{{\mathbb E}}
\def\pr{{\mathbb P}}
 \def\P{{\cal P}}
\def\C{{\cal C}}
\def\MM{{\bf M}}
\def\heavyaccept{{\Phi}_0}
\def\heavyacceptone{{\Phi}_2}
\def\accept{{\mathcal A}_0}
\def\heavyindices{{\mathcal I}}
\def\heavyindicespair{{\mathcal J}}
\def\UB{{\overline m}}
\def\LB{{\underline m}}
\def\LBS{\widehat{{\underline m}}}
\def\b{\widehat{b}}
\def\B{\widehat{B}}
\def\multiplicity{m}
\def\state{{\mathcal S}}
\def\class{{\mathcal C}}
\def\imax{{i_1}}
\def\Z{{\mathcal Z}}
\def\lea{{\le_a}}
\def\a{{\mathcal A}}
\def\O{{\mathcal O}}
\def\R{{\mathcal R}}
  \def\pairstars{{\mathfrak z}}

\def\calf{{\mathcal F}}
\def\Simple{{\mathcal S}}
\def\PL{{PLD}}

\newcommand{\plib}{{\em plib}}

\def\eps{\epsilon}

\def\ss{\smallskip}
\def\ms{\medskip}
\def\non{\nonumber}
\def\no{\noindent}

\date{}

\title{Uniform generation of random graphs with power-law degree sequences}
\author{
Pu Gao\thanks{Research supported by NSERC and ARC DP160100835.}\\
School of Mathematics\\
 Monash University\\
jane.gao@monash.edu
\and
Nicholas Wormald\thanks{Research supported by  ARC DP160100835.}
\\
School of Mathematics\\
Monash University\\
nick.wormald@monash.edu }
\begin{document}
\maketitle

\begin{abstract}

We give a linear-time algorithm that approximately uniformly generates  a  random  simple  graph  with a power-law degree sequence whose exponent is at least  2.8811. While sampling graphs with power-law degree sequence of exponent at least 3 is fairly easy, and many samplers work efficiently in this case, the problem becomes dramatically more difficult when the exponent drops below 3; ours is the first provably practicable sampler for this case. We also show that with  an appropriate  rejection scheme, our algorithm can be tuned into an exact uniform sampler. The running time of the exact sampler is $O(n^{2.107})$   with high probability, and  $O(n^{4.081})$ in expectation.
\end{abstract}

\section{Introduction}
\lab{s:intro}
 We consider the following problem. Given a sequence of \jt{nonnegative} integers $\bf d$ with even sum, how can we generate uniformly at random a graph with degree sequence $\bf d$? Motivation to do this can come from testing algorithms, to test null hypotheses in statistics (see Blitzstein and Diaconis~\cite{BD} for example), or to simulate networks.
If the degree sequence comes from a `real-world' network then it will often follow  a power law (defined precisely below), and often the `exponent' of this power law is between 2 and 3. We give a linear-time algorithm for generating such graphs approximately uniformly at random  in the sense that the total variation distance between the output a distribution from the uniform one is $o(1)$ as the number of vertices goes to infinity. There has been no such approximate sampler before with good guaranteed performance, i.e.\ small error in distribution  and  low time complexity,   for  power-law degree sequences with $\gamma$ below 3. We will also give an exact uniform sampler, \jt{for $\gamma$ less than but sufficiently close to 3}, which has a reasonably low time complexity: $n^{2.107}$ with high probability and $n^{4.081}$ in expectation.

Research on uniform generation of graphs with prescribed degrees has a long history. Tinhofer~\cite{T} studied the case of bounded regular  degrees and described an algorithm without bounding how far the output is from the uniform distribution. Soon afterwards, a uniform sampler using a simple rejection scheme (described in Section~\ref{sec:general}) arose as an immediate consequence of several enumeration methods~\cite{BC,B}. This algorithm works efficiently only when the maximum degree is $O(\sqrt{\log n})$, where $n$ is the number of vertices. A major breakthrough that significantly relaxed the constraint on the maximum degree was by McKay and Wormald~\cite{MWgen}. Their algorithm efficiently generates random graphs whose degree sequence satisfies $\Delta=O(m^{1/4})$, where $\Delta$ denotes the maximum degree and $m$ is the number of edges in the graph. The expected running time of  this  algorithm  is  $O(\Delta^4 n^2)$. Very recently, the authors of the current paper developed a new algorithm which successfully generates random $d$-regular graphs for $d=o(\sqrt{n})$, with expected running time $O(nd^3)$. This new algorithm  is related to  the McKay-Wormald algorithm  but contains  significant  new features which make it possible to cope with larger degrees  without essentially increasing the runtime. 

While the uniform generation of graphs whose degrees are regular, or close to regular, has already been  found  challenging, generating graphs whose degrees are very far from being regular  seems much more difficult. However, a desire  to generate  such  graphs  materialised in recent years. Because of the important roles of the  internet and social networks in modern society, much attention has been paid to  graphs with real-world network properties. One of the most prominent  traits of many real-world networks is that their degree distribution follows the so-called power law, usually with parameter $\gamma$ between 2 and 3 (i.e.\ the  number of vertices with degree $i$ is  roughly proportional to  $i^{-\gamma}$).  Graphs with such degree distributions are sparse but have vertices with very large degrees. There is more than one definition of a power law degree sequence examined in the literature, and we consider the common one resulting when the degrees are independently distributed with a power law with parameter $\gamma$. For this, the maximum degree is roughly $n^{1/(\gamma-1)}$.

Generating random graphs with power-law degrees for $\gamma>3$ is quite easy using the above-mentioned simple rejection scheme or the McKay-Wormald algorithm. However, a radical change occurs when the power-law parameter $\gamma$ drops from above 3 to below. A notable difference is that the second moment of the degree distribution changes from finite to infinite. As a result, the running time of  the simple rejection scheme changes from linear to super-polynomial time. The running time required for the McKay-Wormald algorithm also becomes super-polynomial when $\gamma<3$. Another major difference is the appearance of vertices with degree well above $\sqrt m $. To date, no algorithm has been proved to \jt{uniformly} generate such graphs in time that is useful in practice, with any reasonable performance guarantee on the distribution.

Some known approximation algorithms have polynomial running times without being provably practical. Jerrum and Sinclair~\cite{JS} gave an approximate sampler using the Markov Chain Monte Carlo method for generating graphs with so-called P-stable  degree sequences. Sufficient conditions for P-stability and non-P-stability are given by Jerrum, McKay and Sinclair~\cite{JMS}, but none of these conditions are met by power-law degree sequences with $2<\gamma<3$. On the other hand, the authors found enumeration formulae for power-law degree sequences~\cite{GW} that imply P-stability when $\gamma>1+\sqrt 3$, and one can show using a related but simpler argument that P-stability holds  for all $\gamma>2$. However  the degree of the polynomial bound for the running time of this algorithm is too high to make it suitable for any practical use with a provable approximation bound. Recently Greenhill~\cite{G4} used a different Markov chain, an extension of the one used for the regular case in~\cite{CDG}, to approximately generated non-regular graphs. The mixing time of that chain is bounded by $\Delta^{14}m^{10}$, and the algorithm only works for degree sequences that are not too far from   regular. In particular, it is not applicable to power-law degree sequences of the type we consider that have  $\gamma$  below 3. In~\cite{BD} it is explained how to use sequential importance sampling to estimate probabilities in the uniform random graph with given degrees, using statistics from non-uniform samplers. However, there are no useful theoretical bounds on the number of samples required to achieve desirable accuracy of such estimates. As stated in~\cite[Section 8]{BD}, ``Obtaining adequate bounds on the variance for importance sampling algorithms is an open research problem in most cases of interest.'' Using a uniform sampler completely nullifies this problem.

In this paper, we present a linear-time approximate sampler \PL*\ for graphs with power-law degree sequences, whose output  has a distribution that differs from uniform  by $o(1)$ in total variation distance.  This approximate sampler, like others such as in~\cite{SW,KV,BKS,Zhao}, has   weaker properties  than the ones from  MCMC-based algorithms. Its approximation error  depends on $n$ and  cannot be improved by running the algorithm longer. However, it has the big  advantage of a  linear running time.
None of  the previously proposed asymptotic approximate samplers~\cite{SW,KV,BKS,Zhao}  is effective in the non-regular case  when $\Delta$ is above $m^{1/3}$, whereas the type of power-law degree sequences we consider can handle $\Delta$   well above $m^{1/2}$. So  \PL*\ greatly extends the family of degree sequences for which we can generate random graphs asymptotically uniformly and fast.

Using an appropriate rejection scheme, we tune \PL* into an exact uniform sampler \PL\ which,
for    $\gamma>2.8811$, has a running time that is $O(n^{4.081})$ in expectation, and $O(n^{2.107})$  asymptotically almost surely (a.a.s.).  
Most of the running time  of \PL\ is spent  computing certain probabilities in the rejection scheme.
The reason that the expected   running time is much greater  than our bound on the likely running time  is  that our algorithm has a component that computes  a certain graph function. This particular component is unlikely to be needed on any given run, but it is so time-costly that, even so, it contributes immensely to the expected running time.

We next formally define  power law sequences. There have been different versions in the literature. The definition by Chung and Lu~\cite{CL} requires the number of vertices with degree $i$  to be  $O(i^{-\gamma} n)$, uniformly for all $i$. Hence the maximum degree is $O(n^{1/\gamma})$,  which is well below $\sqrt{n}$.  
 However, it has been pointed out~\cite{N} that this definition leads to misleading results, as in many real networks there are vertices with significantly higher degrees. A more realistic model~\cite{N,V} is to consider degrees composed of independent power-law variables. In this model the maximum degree is of order $n^{1/(\gamma-1)}$, which matches the statistics in the classical preferential attachment model by Barab\'{a}si and Albert~\cite{BA} for modelling complex networks.  In this paper, we will use a more relaxed definition~\cite{GW} which includes the latter version of power-law sequences.  
Let $X$ be a random variable with   power law distribution with parameter $\gamma>1$, i.e.\ $\pr(X=i) = c i^{-\gamma}$, where $c^{-1}=\sum_{i=1}^{\infty} i^{-\gamma}$, then for $\gamma>1$ we have $\pr(X\ge x)  =O(i^{1-\gamma})$. 
As vertices with degree  0 can be ignored in a generation algorithm, we  assume that the minimum degree is at least 1. 

 \begin{definition}\lab{def:powerlaw}
A sequence ${\bf d}=(d_1,\ldots,d_n)$ is power-law distribution-bounded (\plib\ for short) with parameter $\gamma>1$  if the minimum component in ${\bf d}$ is at least $1$ and there is a constant $K>0$ independent of $n$ such that the number of components  that are at least $i$ is at most $Kni^{1-\gamma}$ for all $i\ge 1$. 
\end{definition}

In \plib\ sequences, the maximum element can reach as high as   $n^{1/(\gamma-1)}$, making both enumeration and generation of graphs with such degree sequences more challenging than the Chung-Lu version mentioned above.  It is easy to see that if ${\bf d}$ is composed of independent power-law variables with exponent $\gamma$ then {\bf d} is a \plib\ sequence with parameter $\gamma'$ for any $\gamma'<\gamma$.


  Our main result for the approximate sampler  \PL*, defined in Sections~\ref{sec:general} and~\ref{sec:PLstar}, is as  follows.

  \begin{thm}\lab{thm:approximate}
Assume ${\bf d}=(d_1,\ldots,d_n)$ is a \plib\ sequence with parameter $\gamma>21/10+\sqrt{61}/10\approx 2.881024968$ and $\sum_{1\le i\le n} d_i$ is even. Then 
\PL*\   runs in time $O(n)$ in expectation  and generates a graph with degree sequence ${\bf d}$  whose distribution differs from uniform by $o(1)$ in total variation distance.  \end{thm}

A general description of our uniform sampler, which we call \PL,  will be given in Section~\ref{sec:general},  and  its formal definition  in Sections~\ref{sec:heavy}  and~\ref{sec:double}.  Our result for the exact uniform sampler is the following.

\begin{thm}\lab{thm:main}
Assume ${\bf d}=(d_1,\ldots,d_n)$ is a \plib\ sequence with parameter $\gamma>21/10+\sqrt{61}/10\approx 2.881024968$ and $\sum_{1\le i\le n} d_i$ is even. Then 
\PL\  uniformly generates  a random simple graph  with degree sequence ${\bf d}$, with  running time $O(n^{4.081})$ in expectation, and $O(n^{2.107})$ with high probability.
\end{thm}                  
  
  We remark that the same statements are immediately implied for the Chung-Lu version of power law sequences (which we called power-law density-bounded in~\cite{GW}), however for that version one can easily modify our methods to reduce the bound on $\gamma$ somewhat.  This is because the maximum degree in the Chung-Lu version is much smaller (well below $\sqrt{n}$), and moreover, all multiple edges are likely to have bounded multiplicity. Hence, the major difficult issues we face in this paper do not appear in the Chung-Lu version of power law sequences. 
  
\remove{
\no {\bf  Remarks }. 
\begin{description}
\item{(a)} $21/10+\sqrt{61}/10\approx 2.881024968$.
\item{(b)} The upper bound $\gamma<3$ is \nk{included} for convenience \nk{---} the simple rejection scheme or the McKay-Wormald algorithm in~\cite{MWgen} will work for $\gamma>3$ and our \nk{method}\nm{It's a bit   weird to say that a {\em result} easily extends to $\gamma=3$ ... in what sense?   I think the point is the method works.} in this paper extends easily to the case $\gamma=3$.
\end{description}
  }
  
  We remark that our argument easily gives slightly improved  time complexity bounds for any particular $\gamma$ in the range given in Theorem~\ref{thm:main}; to simplify the presentation we have taken a uniform bound on a complicated function of $\gamma$.

  We describe the general framework of both \PL\ and \PL* in Section~\ref{sec:general}. The formal definition of \PL\ is given throughout Sections~\ref{sec:heavy} and~\ref{sec:double},  together with key lemmas used for bounding the running time. We focus on the design of the algorithm including its definition, and the specification of various parameters involved in the algorithm definition, leaving  many of the straightforward proofs to  the Appendix, particularly those similar to the material in~\cite{GWreg}. We prove Theorems~\ref{thm:main} and~\ref{thm:approximate} in Sections~\ref{sec:proofMain} and~\ref{sec:approximate}.

\section{General description}\lab{sec:general}

As in~\cite{MWgen} and~\cite{GWreg}, we will use the pairing model. Here each vertex $i$ is represented as a bin containing exactly $d_i$ points. Let $\Phi$ be the set of perfect matchings  of  these $\sum_{1\le i\le n} d_i$ points. Each element in $\Phi$ is called a {\em pairing}, and  each set of two matched points  is called a {\em pair}. Given a pairing $P$,  
  let $G(P)$ denote the graph obtained from $P$ by contracting each bin into a vertex and represent each pair in $P$ as an edge. Typically, $G(P)$ is a multigraph, since it can have edges in parallel (multiple edges) joining the same vertices. An edge that is not multiple we call a {\em single} edge. The pairs of $P$ are often referred to as  edges, as in $G(P)$. A loop is an edge with both end points contained in the same vertex. The graph is simple if it has no loops or multiple edges.
  
   The simple rejection scheme repeatedly generates a random $P\in\Phi$ until $G(P)$ is simple.  Let ${\mathcal G}_{{\bf d}}$ denote the set of simple graphs with degree sequence {\bf d}.
An easy calculation shows that every simple graph in ${\mathcal G}_{{\bf d}}$ corresponds to exactly $\prod_{i=1}^n d_i!$ pairings in $\Phi$. Hence, the output of the rejection scheme is uniformly distributed over ${\mathcal G}_{{\bf d}}$. The problem is that, when the degrees become very non-regular, it takes too many iterations for the rejection scheme to find a simple $G(P)$. With power-law degree sequences as in Theorems~\ref{thm:approximate} and~\ref{thm:main}, the rejection scheme is doomed  to run in super-polynomial time.

The McKay-Wormald algorithm~\jt{\cite{MWgen}} does not reject $P$ if $G(P)$ is a multigraph. Instead, it uses switching operations (e.g.\ see Figures~\ref{f:loop} and~\ref{f:triple}) to switch away the multiple edges. These switchings  slightly distort  the distribution of the pairings away from uniform, and so some rejection scheme is used to correct the distribution. The algorithm in~\cite{GWreg} included major rearrangements in set-up and analysis in order to incorporate   new features that reduce the probability of having a rejection, and thereby extended the family of degree sequences manageable. 
Both~\cite{MWgen} and~\cite{GWreg} only deal with multiplicities at most three. However, it is easy to see that for a \plib\ degree sequence with parameter $\gamma<3$, there are multiedges of multiplicity as high as some power of $n$.   Previous exact generation algorithms have never reached the density at which  unbounded multiplicity occurs in the pairing model, and new considerations for the design of \PL* and \PL\ stem from this. One may think that switchings for low multiplicities such as in Figures~\ref{f:loop} and~\ref{f:triple} work for any multiplicity. However, these switchings would cause difficulties for high multiplicities because (a) the error arising in the analysis is too hard to control, and (b) incorporating unbounded multiplicities into the counting scheme would result in  super-polynomial running time.

One new feature in the present paper is that we treat vertices of large degree differently. We will specify a set $\heavy$ of such vertices, which we call {\em heavy}.  Other vertices are called {\em light}.  An edge in $P$ is {\em heavy} if its both end vertices are heavy. In particular, a  loop is   heavy if it is at a heavy vertex. 
Non-heavy edges are called {\em light}.
The algorithm \PL* contains two stages. In the first stage, it uses some switchings, defined in Section~\ref{sec:defheavy}, to turn every heavy multiple edge $ij$ into either a non-edge, or a single edge, and  to switch away all heavy loops. 
It can be shown that with a non-vanishing probability, the final pairing is in the following set
 \be\lab{def:a0}
\a_0=\left\{P\subseteq \Phi:\ G_{[\heavy]}(P)\ \mbox{is simple},\ 
\begin{array}{l}
  L(P)\le B_L;\  D(P)\le B_D;\  T(P)\le B_T;\\
 G(P)\ \mbox{contains no other types of multiple edges}
\end{array}\right\},
\ee
where \jt{$G_{[\heavy]}(P)$ denotes the subgraph of $G(P)$ induced by $\heavy$,} $L(P)$, $D(P)$ and $T(P)$ denote the number of light simple loop, light double edges, and light triple edges in $P$, and $B_L$, $B_D$ and $B_T$ are prescribed parameters specified in~\eqn{imax} in Section~\ref{sec:double}. \PL* restarts until the final pairing $P$ is in $\a_0$, and this is the first stage of \PL*.  We will show that the distribution of the output of the first stage of \PL* is almost uniform over $\a_0$. Algorithm \PL\ is \PL* accompanied with a rejection scheme so that the output of Stage 1 is uniformly distributed over $\a_0$. After Stage 1, \PL* enters the second stage, in which it repeatedly removes light loops, and then light triple edges using random ``valid'' switchings in Figures~\ref{f:loop} and~\ref{f:triple}, which were also used in~\cite{GWreg}. Valid switchings mean that the performance of the switching does not destroy more than one multiple edge (or loop), or \jt{create} new multiple edges (or loops). 
\begin{figure}[!tbp]
  \centering
  \begin{minipage}[b]{0.4\textwidth}
    \includegraphics[width=7cm]{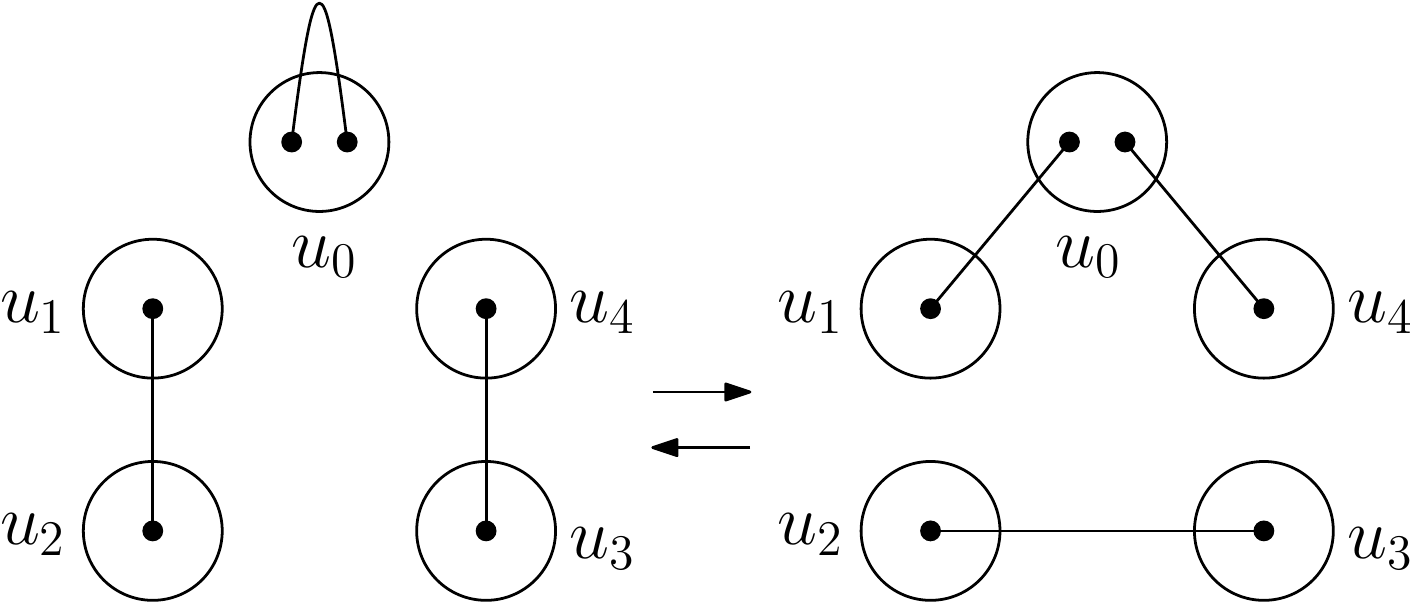}
    \caption{\it   Switching for a light loop} 
    \lab{f:loop}
  \end{minipage}
  \hfill
  \begin{minipage}[b]{0.45\textwidth}
    \includegraphics[width=5cm]{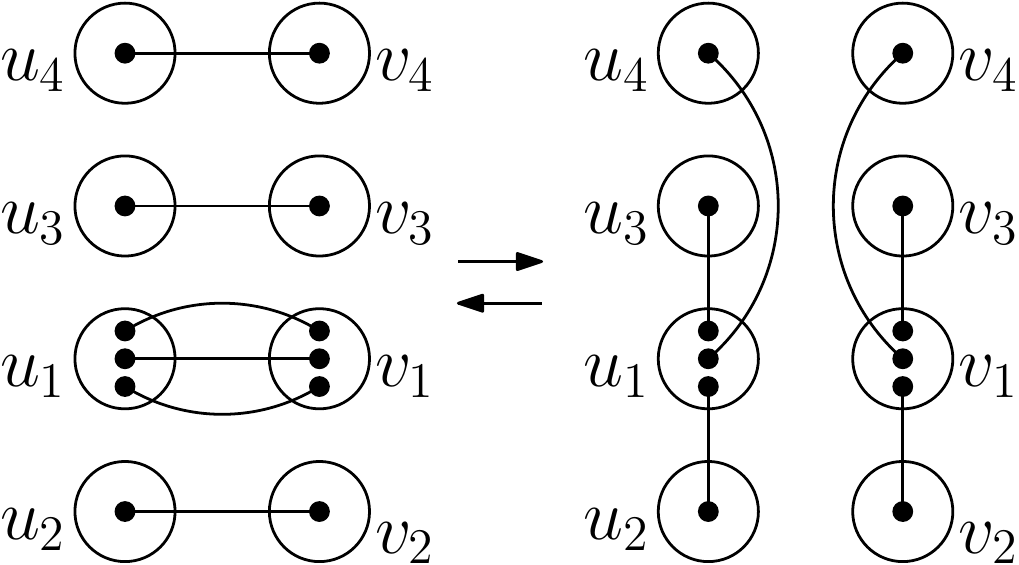}
     \caption{\it  Switching for a light triple edge} 
\lab{f:triple}
  \end{minipage}
\end{figure}
To remove the double edges, we will use parameters $\rho_{III}(i)$, specified in Section~\ref{sec:rho}. In each iteration, if the current pairing $P$ contains $i$ double edges, then with probability $\rho_{III}(i)$, it performs a random valid switching as in Figure~\ref{f:typeIII}, and with probability $1-\rho_{III}(i)$, it performs a random valid switching as in Figure~\ref{f:typeI}. Note that the switching in Figure~\ref{f:typeI} is the typical type (called type I) used to switch away a double edge. The switching in Figure~\ref{f:typeIII} is a new type (called type III) to remedy the distribution distortion caused by  merely performing type I switchings.
\begin{figure}[!tbp]
  \centering
  \begin{minipage}[b]{0.4\textwidth}
    \includegraphics[width=8.5cm]{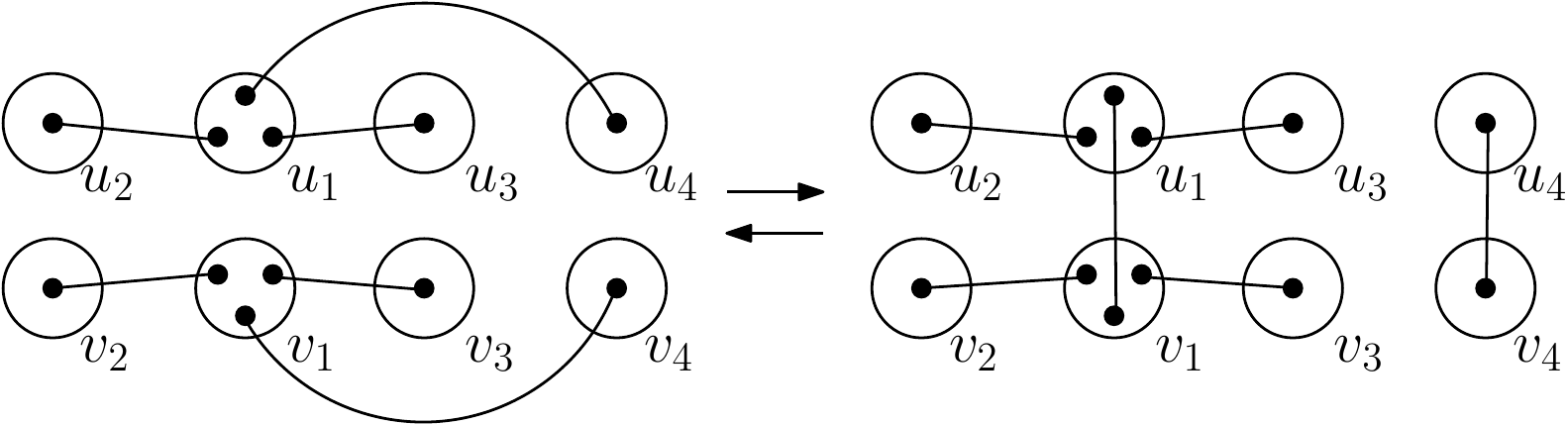}
    \caption{\it  Type III} 
    \lab{f:typeIII}
  \end{minipage}
  \hfill
  \begin{minipage}[b]{0.4\textwidth}
    \includegraphics[width=7cm]{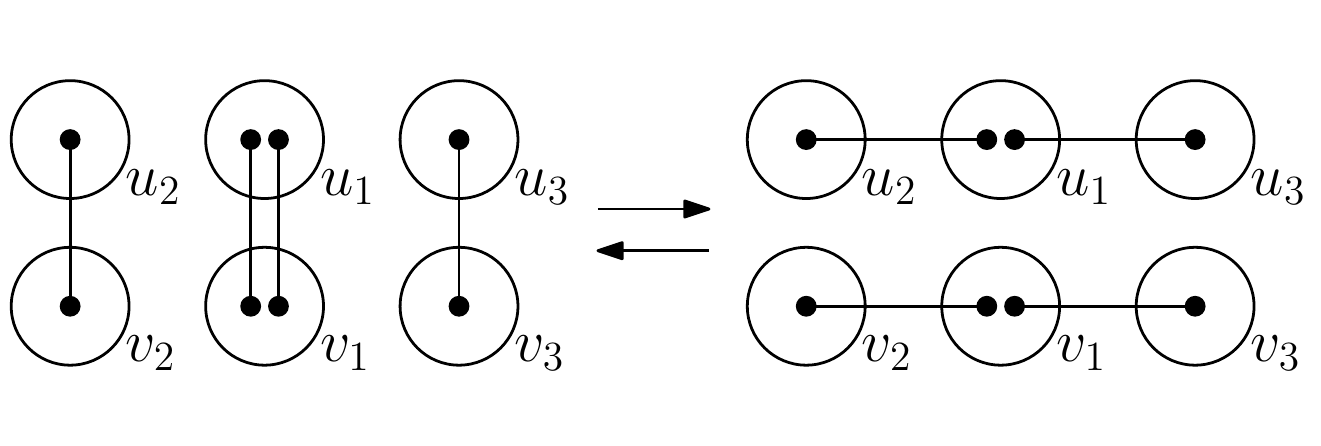}
         \caption{\it  Type I} 
\lab{f:typeI}
  \end{minipage}
\end{figure}

We will prove:
\vspace{-0.5em}
\begin{enumerate}[itemsep=-0.2em]
\item[(a)] all parameters $\rho_{III}(i)$ can be computed in time $o(n)$;
\item[(b)] the number of iterations in the second stage is $o(n)$ in expectation and with high probability;
\item[(c)] starting from a uniformly random $P\in \a_0$, 
the second stage of \PL* outputs an asymptotically uniform random graph with target degree sequence ${\bf d}$.
\end{enumerate}

Even though the type III switching is new, and some other new considerations will be brought in to cope with various complications, the main ideas involved in the second stage of \PL*\ \jt{are} similar to \jt{those} in~\cite{GWreg}. Indeed, we will tune \PL* into \PL\ with introduction of more types of switchings, and with various sorts of rejections. The output of \PL\ is a uniformly random graph with the target degree sequence (see Lemma~\ref{lem:uniformity}).  We can show that the probability that \PL\ ever performs a rejection or other types of switchings is $o(1)$. This allows us to couple \PL\ and \PL* and deduce that the output distributions of \PL* differs from \PL, which is uniform, by $o(1)$, and thus derive our main theorem. The full details of \PL\ and \PL* for the second stage will be presented in Section~\ref{sec:double}. 

Our major innovative ideas lie in the first stage of \PL\ and \PL*.  We will later specify a set of pairings $\heavyaccept\subseteq \Phi$ (in Section~\ref{sec:Phi0}) such that with a non-vanishing probability  a random pairing in $\Phi$ is in $\heavyaccept$. Roughly speaking, $\Phi_0$ excludes pairings \jt{in which} some vertex is incident with too many heavy multiple edges or heavy loops. The algorithm starts by repeatedly generating a random pairing $P\in\Phi$ until $P\in\heavyaccept$. Then, \PL* enters two phases (both in Stage one) sequentially where  heavy multiple edges and then heavy loops are eliminated using switchings. 


\remove{
The algorithm \PL\ contains five phases. Initially, \PL\ generates a uniformly random pairing $P_0$. If $P_0$ satisfies some required conditions, then the algorithm enters each phase sequentially. In each phase, one type of multiple edges are switched away using switching operations. Rejection can occur with a small probability in every step of the algorithm and when that happens the algorithm restarts. If no rejection occurs by the last step of Phase 5, the graph corresponding to the final  pairing is outputted. The framework of \PL\ is shown in Figure~\ref{f:frame}. We will prove that the output is uniformly distributed in ${\mathcal G}_{{\bf d}}$.
\begin{figure}[htb]

 \vspace{-0cm}
 \hbox{
 \centerline{\includegraphics[width=6cm]{frame}}}
 \caption{\it   Overall scheme for \PL} 

\lab{f:frame}

\end{figure}
}

Since our analysis is tailored for the power-law degree sequence, we start by stating some properties of power-law degree sequences.
 In the rest of this paper, we will assume that    ${\bf d}$ is  such that $\sum_i d_i$ is even and Definition~\ref{def:powerlaw} is satisfied for some fixed $K>0$ and   $\gamma>2.5$, that is
\bel{plib}
\mbox{$\displaystyle \sum_{ j\ge i} N_j \le Kni^{1-\gamma}$,\quad for all $1\le i\le\Delta$}
\ee
where $N_j=|t:d_t=j|$. 
    A stronger lower bound on $\gamma$ will be imposed later. We will also assume without loss of generality in the proof that $\gamma<3$, since a \plib\ sequence with parameter $\gamma\ge 3$ is also a \plib\ sequence with parameter $\gamma'$ for every $\gamma'<\gamma$. 


 Recall that $\Delta$ denotes the maximum degree. Without loss of generality we may assume that $\Delta=d_1\ge \cdots \ge d_n\ge 1$.
 For every $1\le i\le n$, the number of vertices with degree at least $d_i$ is at least $i$, so~\eqn{plib} implies that
$
i\le K n(d_i)^{1-\gamma}$.
This gives 
\be
d_i\le (Kn/i)^{1/(\gamma-1)}\quad  \mbox{for all}\ 1\le i\le n. \lab{di}
\ee
For a positive integer $k$, define   $M_k=\sum_{1\le i\le n} [d_i]_k$ where $[x]_k$ denotes $\prod_{i=0}^{k-1}(x-i)$.
It follows immediately from~\eqn{di}  using $\gamma>2$ that 
\bea
\Delta&\le& (Kn)^{1/(\gamma-1)}\non\\ 
M_1&=&\Theta(n)\lab{Delta}\\
M_k&=& O(\Delta^k)=  O(n^{k/(\gamma-1)})\ \mbox{for every fixed integer $k\ge 2$}. \non
\eea


Let $0<\delta<1/2$ be a constant to be specified later. 
Define 
\bea
& &  h=n^{1-\delta(\gamma-1)}, \lab{h}\\
& &\heavy=\{i: 1\le i\le h\},\quad \light=[n]\setminus \heavy. \lab{heavy}
\eea
Immediately from~\eqn{di}, 
\bel{dh}
d_h\le K^{1/(\gamma-1)} n^{\delta},
\ee
and so
every vertex in $\light$, the set of light vertices, has degree $O(n^{\delta})$. 

Finally, for  $k\ge 1$  define $H_k=\sum_{v\in \heavy} [d_v]_k$ and  $L_k=M_k-H_k$. 
\ms

\section{First stage of \PL\ and \PL*: heavy multiple edge reductions}\lab{sec:heavy}

Recall that \PL\  (\PL*) first generates a uniformly random pairing $P_0\in \Phi_0$, where $\Phi_0\subseteq \Phi$ is the set of pairings satisfying some initial conditions. Our next step is to specify $\Phi_0$. 

 Suppose firstly that $M_2 < M_1$. In this case, we define $\Phi_0$ to be the set of simple pairings in $\Phi$. Janson~\cite{Janson} showed that under this condition, there is a  probability bounded away from 0, uniformly for all $n$, that $P_0\in\Phi_0 $. So if $P_0$ is not simple, \PL\  (\PL*) rejects it, and otherwise, the remaining phases are defined to pass $P_0$ directly to the output. In this case, \PL\  (\PL*) is equivalent to the simple rejection scheme. Henceforth in the paper, we assume
\bel{M2}
M_2 \ge M_1.
\ee

\subsection{Specifying $\heavyaccept$}
\lab{sec:Phi0}

Given $i<j$, let $W_{i,j}(P)$ denote the number of points in vertex $i$ that belong to heavy loops or heavy multiple edges  in $P$ with one end in $i$ and the other end {\em not} in $j$. Let $W_{i}(P)$ denote the number of pairs that are in a heavy non-loop multiple edge, with one end in $i$.  For any $i\le j$, let $m_{i,j}=m_{i,j}(P)$ denote the number of pairs between $i$ and $j$ in $P$.   
Define $\eta=\sqrt{M_2^2H_1/M_1^3}$ and
 
\be
\heavyaccept=\left\{P\in\Phi:\ 
\begin{array}{ll}
m_{i,j}I_{m_{i,j}\ge 2}W_{i,j}\le \eta d_i &  \forall 1\le i,j\le h \mbox{ with } i\ne j;\\
m_{i,i} W_i\le \eta d_i &\forall 1\le i\le h;\\
\sum_{1\le i<j\le h} m_{i,j}I_{m_{i,j}\ge 2} \, \le\,  4M_2^2/M_1^2; &\\
\sum_{1\le i\le h} m_{i,i}\, \le\,  4M_2/M_1.& 
\end{array}
\right\}\lab{Wij}
\ee
\ss

 The following lemma ensures that the probability of an initial rejection  is at most $1/4+o(1)$.
\begin{lemma}\lab{lem:Phi0}  Assume that $\eta=o(1)$, i.e.\ $M_2^2H_1/M_1^3=o(1)$. If   $P$ is a random pairing in $\Phi$ then the probability that $P\in\heavyaccept$ is at least $1/4+o(1)$.
\end{lemma}

  Next we  define the switchings used in Stage 1.

\subsection{Heavy switchings}
\lab{sec:defheavy}

 Given a pairing $P$, let $m(P,i,j)$ denote the number of pairs between $i$ and $j$ in $P$, and set $\multiplicity(P,i) = m(P,i,i)$, the number of loops at $i$.

We first define the switchings we are going to use to switch away heavy multiple edges between a fixed pair of heavy vertices $(i,j)$, $i<j$.  The switching that we are going to define will change the multiplicity of $ij$ from some $m\ge 1$ to $0$ (i.e.\ no pair between $i$ and $j$). We will only use this switching directly if $m(P,i,j)\ge 2$, and we will also use an inverse version with a certain probability  to change the multiplicity from 0 to $1$. The inverse switching is needed for the present problem because  $i$ and $j$ may be quite likely to have an edge between them.  

\begin{definition}\lab{def:heavymultiple}
 (Heavy $m$-way switching at $(i,j)$ on $P$.) 

 This is defined for an ordered pair $(i,j)$ of vertices   with $\multiplicity(P,i,j)=m\ge 1$. 
Label the endpoints of the $m$ pairs between $i$ and $j$ as $2g-1$ and $2g$, $1\le g\le m$, where points $1,3,\ldots,2m-1$ are contained in vertex $i$.
Pick $m$ distinct light pairs $x_1,\ldots,x_m$. Label the endpoints of
$x_g$ as $2m+2g-1$ and $2m+2g$.  The switching operation replaces pairs $\{2g-1,2g\}$ and $\{2m+2g-1,2m+2g\}$ by $\{2g-1,2m+2g-1\}$ and $\{2g,2m+2g\}$. This is a valid switching if  no new heavy multiple edges or heavy loops are created. See Figure~\ref{f:heavy1}.

If $S$ is a heavy multiple edge switching creating $P'$ from $P$, then the inverse of $S$ is an inverse heavy multiple edge switching creating $P$ from $P'$. 	
\end{definition}


\begin{figure}[!tbp]
  \centering
  \begin{minipage}[b]{0.45\textwidth}
 \hbox{\centerline{\includegraphics[width=7cm]{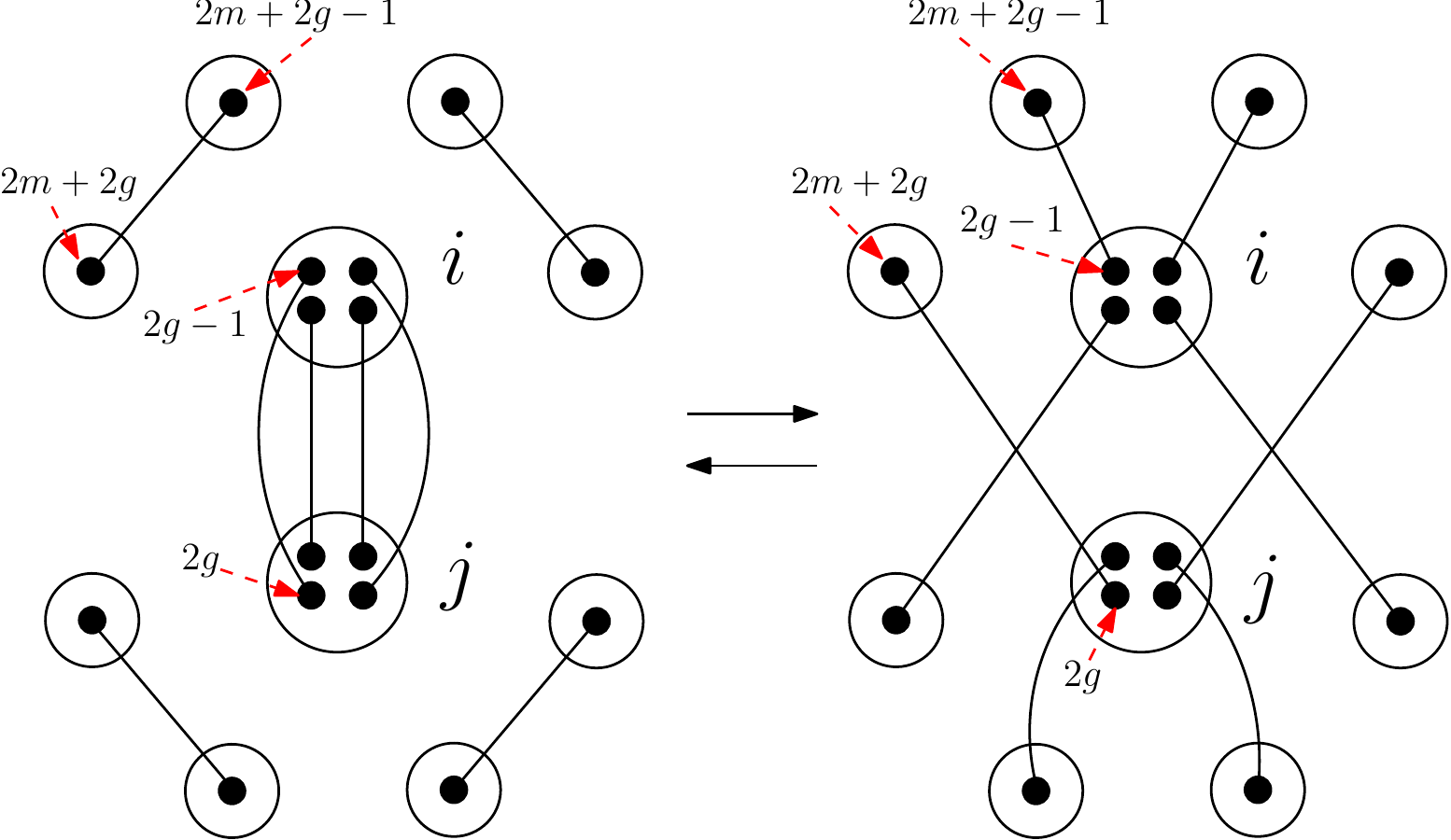}}}
\caption{ Switching for   heavy multiple edge }
\lab{f:heavy1}
  \end{minipage}
  \hfill
  \begin{minipage}[b]{0.45\textwidth}
\hbox{\centerline{\includegraphics[width=7cm]{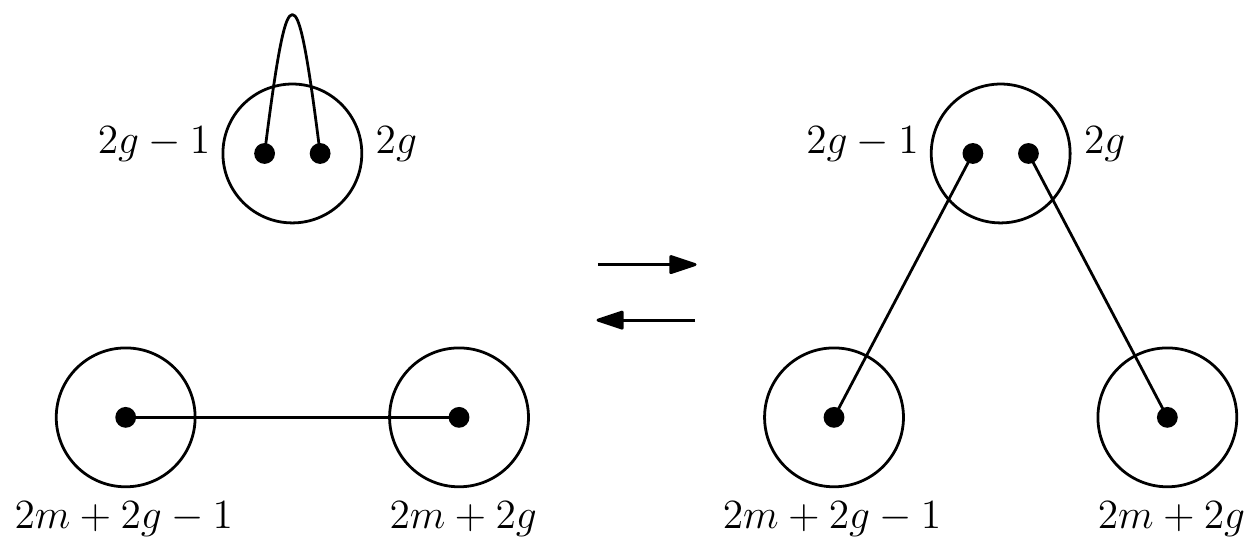}}}
\caption{ Switching for   heavy multiple loop }
\lab{f:heavy2}
  \end{minipage}
\end{figure}
 

\begin{definition}
(Heavy $m$-way loop switching at $i$.)

 This is defined if $\multiplicity(P,i)=m\ge 1$.  
Label the endpoints of the $m$ loops at $i$ as $2g-1$ and $2g$, $1\le g\le m$. Pick $m$ distinct light pairs in $P$,  labelling the endpoints of the $g$-th pair  $2m+2g-1$ and $2m+2g$. The switching replaces pairs
$\{2g-1,2g\}$ and $\{2m+2g-1,2m+2g\}$ by $\{2g-1,2m+2g-1\}$ and $\{2g,2m+2g\}$. This switching is valid if no new heavy multiple edges or heavy loops are simultaneously created.  See Figure~\ref{f:heavy2}. 
\end{definition}




Note from the above definitions that  heavy multiple edge  and heavy loop switchings are permitted to create  or destroy  light multiple edges or light loops. This was not the case for the switchings introduced in~\cite{GW} for enumeration purposes, and is another feature of our arguments not present in~\cite{MWgen} and~\cite{GWreg}. A heavy multiple edge switching on the pair $(i,j)$ still does not affect other heavy multiple edges.

\subsection{\PL: Phases 1 and 2}
Now we are ready to define the first two phases of \PL. We define \PL* in Section~\ref{sec:PLstar}.

\subsubsection{Phase 1: heavy multiple edge reduction}
~\lab{sec:heavy-uniform}

 Given a pairing $P$, we define  an $h\times h$ symmetric matrix $\MM(P)$   as follows. For each $1\le i\le j\le h$, if  $ij$ is not a multiple edge or loop  in $G(P)$, then the $ij$ and $ji$ entries in $\MM(P)$ are set as $\clubsuit$; otherwise, these entries are set as the multiplicity of the edge $ij$. Therefore, in this matrix, all the diagonal entries take values in $\{\clubsuit\}\cup \mathbb{N}_{\ge 1}$ and all off-diagonal entries take values in $\{\clubsuit\}\cup \mathbb{N}_{\ge 2}$ where $\mathbb{N}_{\ge k}$ denotes the set of integers that are at least $k$. We call $\MM(P)$ the {\em signature} of $P$. Given such a matrix $\MM$, we let $\class(\MM)$ denote the set of pairings that have $\MM$ as their signature. Phase 1 of \PL\ uses the  heavy $m$-way switchings of  Definition~\ref{def:heavymultiple} to switch away heavy multiple edges.  Each switching step converts a pairing uniformly distributed in $\class(\MM)$ for some $\MM$   to a pairing uniformly distributed in $\class(\MM')$ where $\MM'$ is obtained from $\MM$ by setting a symmetric pair of its off-diagonal entries to $\clubsuit$. By the end of Phase 1, if no rejection occurs, the resulting pairing has a signature with all off-diagonal entries being $\clubsuit$,  and hence has no heavy multiple edges.

Recall $\heavyaccept$ in~\eqn{Wij}. When Phase 1 starts, $P_0$ is uniformly distributed in $\heavyaccept$. Let $\MM_0=\MM(P_0)$.

Recall that $\heavy$ denotes the set of heavy vertices. Let $\heavyindicespair$ be a  list  of all pairs $(i,j)$ with $i<j$ and both $i,j\in\heavy$; i.e.\ $\heavyindicespair$ is an enumeration of all pairs of heavy vertices. This defines an ordering of elements in $\heavyindicespair$: $(i',j')\preceq (i,j)$ if $(i',j')$ appears before $(i,j)$ in  the list  $\heavyindicespair$. For each $k\ge 1$, let $\MM_k$ denote the matrix obtained from $\MM_0$ by setting the  $ij$ and $ji$   entries of  $\MM_0$ to $\clubsuit$,  for each $ij$ among the first $k$ entries in the list
$\heavyindicespair$.

Let $k\ge 1$  and suppose $(i,j)$ be the $k$-th element of $\heavyindicespair$.
 For any non-negative integer $m$, let $\C(\MM_k,i,j,m)$ denote the set of pairings with $ij$ having multiplicity $m$, and all other edges between heavy vertices satisfying the signature $\MM_k$ (ignoring its $ij$ entry). 
Given $m\ge 1$ and a pairing  $P\in \C(\MM_k,i,j,m)$,  
let $f_{i,j}(P)$ denote the number of  heavy $m$-way switchings at $(i,j)$ where $m$ is the multiplicity of $ij$ in $P$. Given $m\ge 1$ and $P'\in \C(\MM_k,i,j,0)$, let $b_{i,j}(P',m)$ denote the number of inverse heavy $m$-way switchings applicable to $P'$; these convert  $P'$ to some $P\in \C(\MM_k,i,j,m)$.  

Note that $W_{i,j}(P)$ is determined by the signature of $P$; hence we may write $W_{i,j}(\MM(P))$ for $W_{i,j}(P)$. For any $m\ge 1$, define
\bea
\UB^f_{i,j}(\MM_k,m)&=&m!M_1^m;\lab{boundsmultiplefu}\\
\LB^f_{i,j}(\MM_k,m)&=&m!(M_1-H_1-2m)^m;\lab{boundsmultiplefl}\\
\UB^b_{i,j}(\MM_k,m)&=&[d_i-W_{i,j}]_m[d_j-W_{j,i}]_m;\lab{boundsmultiplebu}\\
\LB^b_{i,j}(\MM_k,m)&=&[d_i-W_{i,j}]_m[d_j-W_{j,i}]_m-mh^2[d_i-W_{i,j}]_{m-1}[d_j-W_{j,i}]_{m-1},\lab{boundsmultiplebl}
\eea
where $W_{i,j}=W_{i,j}(\MM_k)$ and $W_{j,i}=W_{j,i}(\MM_k)$.

By simple counting arguments (see Lemma~\ref{lem:bounds}) we have that for every $m\ge 1$ and for every $P\in \C(\MM_k,i,j,m)$ and $P'\in \C(\MM_k,i,j,0)$,
\[
\LB^f_{i,j}(\MM_{k},m)\le f_{i,j}(P)\le \UB^f_{i,j}(\MM_{k},m), \quad \LB^b_{i,j}(\MM_{k},m)\le b_{i,j}(P',m)\le \UB^b_{i,j}(\MM_{k},m).
\]

Phase 1 of \PL\ is  defined inductively as follows. For each $1\le k\le |\heavyindicespair|$,  let $P_{k}$ denote the pairing obtained after the $k$-th step of Phase 1. To define this step, if $\multiplicity(P_{k-1},i,j)\le 1$, then $P_k=P_{k-1}$; otherwise, let $m=\multiplicity(P_{k-1},i,j)$ and do the following sub-steps:
\begin{enumerate}[itemsep=-0.2em]
\item[(i)] Choose a random heavy  $m$-way switching at $(i,j)$ on  $P$. Let $P'$ be the pairing created by $S$; 
\item[(ii)] Perform an f-rejection with probability $1-f_{i,j}(P)/\UB^f_{i,j}(\MM_{k},m)$; then perform a b-rejection with probability $1-\LB^b_{i,j}(\MM_{k},m)/b_{i,j}(P',m)$;
\item[(iii)] If no f or b-rejection took place, set $P_k^*=P'$;
\item[(iv)] Choose a random {\em inverse} heavy  1-way  switching $S'$ at $(i,j)$ on $P'$; let $P''$ denote the pairing created by $S'$;
\item[(v)] With probability $1/(1+\UB^b_{i,j}(\MM_k,1)/\LB^f_{i,j}(\MM_k,1))$, set $P_k=P_k^*=P'$; with the remaining probability, perform an f-rejection (w.r.t.\ $S'$) with probability $1-b_{i,j}(P')/\UB^b_{i,j}(\MM_k,1)$ and perform a b-rejection with probability $1-\LB^{f}_{i,j}(\MM_k,1)/f_{i,j}(P'')$. If no f- or b-rejection occurred, set $P_k=P''$.    
\end{enumerate}

\begin{lemma}\lab{lem:heavy-rej}
The probability of an $f$-rejection or $b$-rejection during Phase 1 is $o(1)$, if $\delta>(3-\gamma)/(\gamma-2)$, $\delta>1/(2\gamma-3)$, $\delta<1/2$ and $1-\delta>1/(\gamma-1)$.
\end{lemma}

\subsubsection{Phase 2: heavy loop reduction}

The heavy loop reduction is similar to the heavy multiple edge reduction but simpler. Let $\heavyindices$ be an ordering of elements in $\heavy$. At step $k$ of Phase 2, for $k\ge 1$, the algorithm switches away all loops at the $k$-th vertex in $\heavyindices$ using the heavy loop switching, if there are any.  In each step, sub-steps (i) and (ii) are the same as in heavy multiple edge reduction, except that $f_{ij}(P)$ and $b_{ij}(P',m)$ are replaced by $f_i(P)$ and $b_i(P',m)$, and $\UB^f_{ij}(\MM_k,m)$ and $\LB^b_{ij}(\MM_k,m)$ are replaced by $\UB^f_{i}(\MM_k,m)$ and $\LB^b_{i}(\MM_k,m)$, defined as below: 
\bea
\UB_i^f(\MM_k,m)&=& 2^m m! M_1^{m};\lab{loopswtUBf}\\
\LB_i^f(\MM_k,m)&=& 2^m m! (M_1-H_1-2m)^{m}; \\
\UB_i^b(\MM_k,m)&=& [d_i-W_i]_{2m}; \\
\LB_i^b(\MM_k,m)&=& [d_i-W_i]_{2m}-mh^2[d_i-W_i]_{2m-2}.\lab{loopswtLBb}
\eea
Moreover, $P_k$ is set \jt{to} $P'$ as long as no f or b-rejection happend. There are no sub-steps (iii)--(v).

It is easy to show that for any $P\in  \C(\MM_k,i,i,m)$,
$\LB_i^f(\MM_k,m)\le f_{i}(P)\le \UB_i^f(\MM_k,m)$; and for any $P\in \C(\MM_k)$,
$\LB_i^b(\MM_k,m)\le b_{i}(P,m)\le \UB_i^b(\MM_k,m)$.


 \begin{lemma}\lab{lem:heavyloopcond}
If $2/(\gamma-1)<1+\delta(\gamma-2)$ and $\delta(\gamma-1)>2/3$, then the probability that a rejection happens during Phase 2 is $o(1)$.
\end{lemma}

\subsubsection{Uniformity}
Recall
\[
\heavyacceptone=\{P\subseteq \heavyaccept:\ G_{[\heavy]}(P)\ \mbox{is simple}\}.
\]
We have the following lemma confirming the uniformity of the output of the first two phases.
\begin{lemma}\lab{lem:afterheavy}
Let $P_0$ be a random pairing in $\heavyaccept$ and $P'$ be the output   after the first two phases, assuming that no rejection occurred. Then $P'$ has the uniform distribution on $\Phi_2$.
\end{lemma}

\subsection{Stage 1 of \PL*}
\lab{sec:PLstar}
The first stage of \PL* is just \PL\ without rejection. I.e. Phase 1 of \PL* consists of substeps (i), (iv) and (v) without f-rejection or b-rejection. Phase 2 of \PL* repeatedly switches away heavy loops using random heavy loop switchings.

\section{Proof of Theorem~\ref{thm:main}}\lab{sec:proofMain}

\remove{



\bel{delta}
f_1(\gamma)<\delta<f_2(\gamma), \quad 5/2<\gamma<3
\ee
where
\[
f_1(\gamma)=\max\left\{\frac{4/(\gamma-1)-2}{\gamma-2},\frac{3-\gamma}{\gamma-2},\frac{1}{2\gamma-3}\right\},
\]
\bean
f_2(\gamma)&=&\min\Big\{\frac12,1-\frac{1}{\gamma-1},\frac{2}{7-\gamma}, \frac{2-2/(\gamma-1)-(2\gamma-3)/(\gamma-1)^2}{3-\gamma},\frac{2-3/(\gamma-1)}{4-\gamma} \Big\}
\eean 
It is easy to see that \eqn{delta} is feasible whenever
\[
21/10+\sqrt{61}/10<\gamma<3,
\]
where $21/10+\sqrt{61}/10\approx 2.881024968$.

For $\gamma$ in that range, we may choose any
\bel{deltaRange}
\frac{1}{2\gamma-3}<\delta<\frac{2-3/(\gamma-1)}{4-\gamma}
\ee
for the algorithm. To optimise the running time, which we discuss next, we will choose $\delta$ just slightly above $1/(2\gamma-3)$. See the Appendix for a detailed reduction of all required constraints to~\eqn{delta}.
\ss
}

\no {\bf Uniformity}

By Lemma~\ref{lem:afterheavy}, \PL\ 
generates a uniformly random pairing $P\in\heavyacceptone$ after the first stage. In Section~\ref{sec:double} we show that the second stage of \PL\ outputs a uniformly random graphs with a plib degree sequence with parameter $\gamma$; see Lemma~\ref{lem:uniformity}.\ss




\no {\bf Running time}

The initial generation of $P\in\Phi$ takes $O(n)$ time, as there are only $O(n)$ points in the pairing model. 

If $(\gamma,\delta)$ satisfies all conditions in Lemmas~\ref{lem:Phi0}, ~\ref{lem:heavy-rej} and~\ref{lem:heavyloopcond} then \PL\ only restarts $O(1)$ times in expectation in Stage 1. In Section~\ref{sec:delta} we show that if $\gamma>21/10+\sqrt{61}/10\approx 2.881024968$ then there exists $\delta$ satisfying 
\bel{deltaRange}
\frac{1}{2\gamma-3}<\delta<\min\Big\{1-\frac{1}{\gamma-1}, \frac{2-2/(\gamma-1)-(2\gamma-3)/(\gamma-1)^2}{3-\gamma} \Big\},
\ee
and that with any such choice of $\delta$, the probability of any rejection occurring in the second stage is $o(1)$.   It only remains to bound the time complexity of \PL\ assuming no rejection occurs.

We first bound the running time in the first two phases. Each step of the algorithm involves computing $f_{i,j}(P)$ and $b_{i,j}(P,m)$ in Phase 1, and $f_i(P)$ and $b_i(P,m)$ in Phase 2.  In the next lemma we bound the computation time for such functions. The proof is in the Appendix. 
\begin{lemma}\lab{lem:heavy-compute}
\begin{enumerate}
\item[(a)] The running time for computing $f_{i,j}(P)$ and $f_i(P)$ in Phases 1 and 2, given $P$, is $O(H_1+|\heavy|m(P,i,j))$ and $O(H_1+|\heavy|m(P,i,i))$ respectively.
\item[(b)] 
The running time for computing $b_{i,j}(P,m)$ and $b_i(P,m)$ in Phases 1 and 2, given $P$, is $O(\Delta+m)$.
\end{enumerate}
\end{lemma}

Next, we bound the time complexity of Phases 1 and 2. 
\begin{lemma}\lab{lem:timeheavy}
The expected running time for the first two phases is $o(n)$.
\end{lemma}
\proof By the definition of $\Phi_0$ in~\eqn{Wij}, 
\[
\sum_{i<j\le h} \ex(m_{i,j}I_{m_{i,j}\ge 2}) = O(M_2^2/M_1^2)
\]
for every $P\in\Phi_0$. Let $P_0\in\Phi_0$ be the pairing obtained before entering Phase 1. Recall that $\MM(P_0)$ is the signature of $P_0$ and let $m_{i,j}$ denote the $ij$ entry of $\MM(P_0)$ for $(i,j)\in\heavyindicespair$.  Note that the heavy multiple edge switchings are applied only for $(i,j)\in \heavyindicespair$ such that $m_{i,j}\ge 2$. By Lemma~\ref{lem:heavy-compute}, for each such $(i,j)$, the running time is $O(m_{i,j})$ for sub-step (i), $O(H_1+|\heavy|m_{i,j})$ 
 for (ii), $O(1)$ for  (iii) and (iv), and 
$O(H_1+|\heavy|)$ 
 for (v). Hence, the total time required for switching away each heavy multiple edge is $O(H_1+|\heavy|m_{i,j})$ and hence, the time complexity for Phase 1 is
\[
\sum_{1\le i<j\le h} O((H_1+|\heavy|m_{i,j}) I_{m_{i,j}\ge 2}) =O(H_1)\sum_{1\le i<j\le h} \ex(m_{i,j}I_{m_{i,j}\ge 2})= O(H_1 M_2^2/M_1^2),
 \] 
as $|\heavy|=O(H_1)$. The above is bounded by $o(n)$ for any $\gamma$ in the range of Theorem~\ref{thm:main} and any choice of $\delta$ satisfying~\eqn{deltaRange}. Similarly, the time complexity for Phase 2 is $o(n)$.\qed
\ss

 The analysis for the running time of Phases 3--5 is similar to that of~\cite{GWreg}. In Section~\ref{sec:rho} we show that parameters $\rho_{\tau}(i)$ such as $\rho_{III}(i)$ can be computed inductively in at most $B_D=o(n)$ steps. Each switching step involves counting the number of certain structures (involving a bounded number of edges) in a pairing (e.g.\ the number of 6-cycles). Using direct brute-force counting methods yields the time complexity claimed in Theorem~\ref{thm:main}. The details are given in the Appendix under the heading ``Running time for Phases 3--5".


\remove{
Next, we bound the time complexity of Phases 3 and 4. Again, each step of the algorithm involves information on $f_{\tau}(P)$, $\b_{\tau}(P,\pairstars)$ and $b(P)$. However, it is easy to see that it is never necessary to compute $f_{\tau}(P)$, due to the way the $\UB_{\tau}(i)$ is defined. For instance, $\UB_{I}(i)=4iM_1^2$ which is the total number of ways to pick a double edge and label its end points, and pick two pairs,  with repetition allowed, and label their end points. If such a choice does not yield a valid switching then perform an f-rejection. It is easy to see that the f-rejection is perform with the correct probability, as defined in~\eqn{frejDef}. Thus, it suffices to compute $\b_{\tau}(P,\pairstars)$ and $b(P)$ only. 

We first discuss the computation of $b(P)$. By~\eqn{Zstar}, $b(P)=M_2L_2-Z^*(P)$. Note that $Z^*(P)$ counts pairs of $2$-paths involving at most $5$ vertices, or pairs of $2$-paths looking like structures in Figure~\ref{f:H} but involving double edges. As~\cite{MWgen}, $Z^*(P)$ can be computed by counting the number of 3-paths joining a pair of points. The number of such 3-paths can be computed within time $M_2\Delta$ as there are at most so many 3-paths in $P$. Then, we can maintain these data with proper data structure, by updating their values each time when a switching is performed. The time cost for updating the data is only $O(1)$ in every step. So the total computation time for $b(P)$ is bounded by $M_2\Delta+O(B_D)=O(M_2\Delta)$. 

Last we discuss the computation of  $\b_{\tau}(P,\pairstars)$ for $\tau\in\{III,IV,V,VI,VII\}$. Each such function counts the number of ways to choose a structure corresponding to one of those in Figure~\ref{f:H}, plus several other pairs with restrictions that between certain pairs of vertices
 there should be no edge. For instance, for 
$\tau=III$, $\b_{III}(P,\pairstars)$ counts the number of choices for pairs representing $H_1$ in Figure~\ref{f:H} and an extra pair that is not too close to the $H_1$. Clearly these structures can be efficiently counted by maintaining certain data using proper data structures, as was done for computing $b(P)$. However, for simplicity, we just use the bound from doing a brute-force search. The computation time required is always bounded by the total number of structures to be counted that appear in $P$. Hence, $\b_{\tau}(P,\pairstars)$ can be counted within time $M_3n^{2\delta}M_1$ for $\tau=III$; $M_2\Delta^2 n^{\delta} M_1^3$ for $\tau=IV$; $M_3n^{2\delta} M_1^4$ for $\tau=V$; $M_2\Delta^2 n^{\delta} M_1^6$ for $\tau=VI$ and $M_3n^{2\delta} M_1^7$ for $\tau=VII$. Let $\rho^*_{\tau}$ denote $\max_{0\le i\le \imax}\rho_{\tau}(i)$. Then the expected total cost for computing $\b_{\tau}(P,\pairstars)$ during the algorithm is bounded by
\[
B_D(M_3n^{2\delta}M_1\rho^*_{III}+M_2\Delta^2 n^{\delta} M_1^3 \rho^*_{IV}+ M_3n^{2\delta} M_1^4 \rho^*_{V}+M_2\Delta^2 n^{\delta} M_1^6 \rho^*_{VI} +M_3n^{2\delta} M_1^7 \rho^*_{VII}).
\]
With the value of $\rho^*_{\tau}$ in Lemma~\ref{lem:solution}, it is easy to see that the last term $B_DM_3n^{2\delta}M_1^7\rho^*_{VII}$ dominates the sum.
For the range of $\gamma$ in Theorem~\ref{thm:main} and by choosing sufficiently small $\delta$ satisfying~\eqn{deltaRange}, the above is bounded by $n^{7.58}$.

Moreover, since $B_D\rho^*_{\tau}=o(1)$ for $\tau\in\{IV,V,VI,VII\}$, the probability that these types of switchings are ever performed in \PL\ is $o(1)$, since a.a.s.\ Phase 5 lasts only $O(B_D)$ iterations by Lemma~\ref{lem:steps}. Therefore, a.a.s.\ the  the computation time of  $\b_{\tau}(P,\pairstars)$ for $\tau\in\{III,IV,V,VI,VII\}$ is bounded by
\[
B_D(M_3n^{2\delta}M_1\rho^*_{III})=O(n^{3.32}).
\]

It is easy to bound the computation time in Phases 3 and 4, using the brute-force searching, by $n^{3.1}$. Hence, we proved the time complexity claimed in Theorem~\ref{thm:main}. \qed

}


\section{Proof of Theorem~\ref{thm:approximate} }
\lab{sec:approximate}

The approximate sampler \PL*\ is simply \PL\ without rejections, except for rejections when checking $P\in \heavyaccept$ and $P\in\a_0$.  As a result, there is no need to compute the number of certain structures, as done in \PL. 
 \ms

\no {\em Proof of Theorem~\ref{thm:approximate}.}
We have already discussed above that it takes \PL, and thus \PL*\ as well, $O(n)$ time to find $P\in\a_0$ before entering the last three phases. The computation of $\rho_{III}(i)$ takes at most $B_D=o(n)$ units of time. It is easy then to see then that each iteration of \PL*\ in these three phases takes only $O(1)$ time, as it only involves switching several pairs, and it is shown in Section~\ref{sec:double} that the total number of iterations in these three phases is $O(B_L+B_D+B_T)=o(n)$. Hence,  \PL*\ is a linear time algorithm  in expectation. It only remains to show that the output of \PL*\ is close to uniform.

Clearly \PL\ and \PL*\ can be coupled so that they produce the same output if no rejection occurs, and no types other than $\tau=I,III$ are chosen in \PL\ in the last three phases. By Lemma~\ref{lem:uniformity}, 
any rejection occurs in Phase 3--5 with probability $o(1)$. We also prove (see~\eqn{xi} and Lemma~\ref{lem:solution}) that the probability of ever performing a switching of type other than $I$ and $III$ is $o(1)$ (the probability is $O(\xi B_D)$ which is guaranteed $o(1)$ by Lemma~\ref{lem:typerej}). Thus, the probability that \PL\ and \PL* have different output is $o(1)$. Let $A$ be an arbitrary subset of ${\mathcal G}_{{\bf d}}$. The probability that \PL\ outputs a graph in $A$ whereas \PL*\ does not is $o(1)$, as that happens only if they have different outputs. Hence,
\[
\pr_{PLD}(A)-\pr_{PLD*}(A)=o(1)\quad \mbox{for all $A\subseteq {\mathcal G}_{{\bf d}}$}.
\]
That is, the distribution of the output of \PL*\ differs  by  $o(1)$ in total variation distance from   the distribution of the output of \PL, which is uniform. \qed

\section{Second stage: phases for light switchings}
\lab{sec:double}

The second stage of \PL* has been described already in Section~\ref{sec:general}. In this section, we tune it into \PL\ with a rejection scheme, as well as introducing more types of switchings.

\jt{Recall that $M_k=\sum_{1\le i\le n}[d_i]_k$, $H_k=\sum_{i\in\heavy}[d_i]_k$ and $L_k=M_k-H_k$.} Define
\bel{imax}
B_L=\frac{4L_2}{M_1}; \quad B_D=\frac{4L_2M_2}{M_1^2} \quad B_T=\frac{2L_3M_3}{M_1^3}.
\ee
and
recall from~\eqn{def:a0} that
 \[
\a_0=\left\{P\in\heavyacceptone:\ 
\begin{array}{l}
  L(P)\le B_L;\  D(P)\le B_D;\  T(P)\le B_T;\\
 G(P)\ \mbox{contains no other types of multiple edges}
\end{array}\right\}.
\]
Roughly speaking, the only possible multiple edges in a pairing in $\accept$ are loops,  double edges and triple edges, and there are not too many of them.
Lemma~\ref{lem:afterheavy} guarantees that the resulting pairing $P$ after the first two phases is uniformly distributed in $\heavyacceptone$, and the following lemma guarantees that \PL\ restarts $O(1)$ times in expectation until entering the second stage, i.e.\ the last three phases. 
The proof of the lemma is presented in Appendix.
\begin{lemma}\lab{lem:A0rej} 
Assume $L_4=o(M_1^2)$ and $L_4M_4=o(M_1^4)$.
With probability at least $1/2$, a random $P\in \heavyacceptone$ is in $\a_0$.
\end{lemma}

The second stage consists of Phases 3, 4 and 5 during which light multiple edges (loops, triple edges and double edges) are  eliminated  in turn. These phases are similar to those in~\cite{GWreg} for regular graphs, except that here we use more types of switchings to boost several types of structures due to the non-regular degree sequence. 
  In each phase, the number of a certain type of light multiple edges will be reduced until there are no multiple edges of that type. There will be  f and b-rejections designed so that after each phase, the output is uniform \jt{conditioned on the event that} the number of multiple edges of a certain type is zero. In each phase, the set of pairings is partitioned into classes $\cup_{i=0}^{\imax} \state_i$, where $\state_i$ is the set of pairings with exactly $i$ multiple edges of the type that is treated in the phase; $\imax\in\{B_L,B_D,B_T\}$ is a pre-determined integer whose validity is guaranteed by the definition of $\accept$. Let $\a_i$ be the output after each phase for $i=3,4,5$. In Phase 3 (and 4), the number of loops (triple edges) will be reduced by one in each step. Hence, $\a_4$ will be the set of pairings in $\accept$ containing no loops or triple edges. The analysis for Phases 3 and 4 is quite simple,   almost the same as that in~\cite{MWgen} and~\cite{GWreg}. We postpone this until Sections~\ref{sec:loop}.  In Phase 5 the number of light double edges is reduced.   For this, we will use the boosting technique introduced in~\cite{GWreg}, but we use more booster switchings to cope with non-regular degrees.
We briefly describe the framework of Phase 5 here, pointing out the major differences compared with~\cite{GWreg}, with the detailed design, definitions and analysis given in the rest of this section. Phase 5 consists of a sequence of switching steps, where each step typically eliminates one double edge using a type I switching operation. This is the same operation used in~\cite{MWgen}. However, in some occasional switching steps, some other types of switchings are performed. The algorithm first determines certain parameters $\rho_{\tau}(i)$, $0\le i\le B_D$. In each switching step, \PL\ chooses a switching type $\tau$, e.g. $\tau=I$, from a given set of switching types, with probability $\rho_{\tau}(i)$, where $i$ is the number of double edges in the  current pairing. After choosing type $\tau$, \PL\ performs a random switching of  type $\tau$. Rejections can happen in each switching step with small probabilities, and these probabilities are carefully designed to maintain the   property that  the expected number of times a pairing is reached, throughout the course of the algorithm, depends only on the number of double edges the pairing has.  (This feature was introduced in~\cite{GWreg} and permits the use of switchings that occasionally increase the number of double edges, and is a marked departure from the method in~\cite{MWgen}.) As a consequence, 
the output after Phase 5 has uniform distribution. It follows that the output of \PL\ is uniform over simple graphs with   degree sequence {\bf d}. 

 The parameters $\rho_{\tau}(i)$ are designed to guarantee the uniform distribution of the output of \PL, and meanwhile to have a small probability of any rejection during Phase 5.  They are determined as the solution of a system of equations for which we give an efficient ($o(n)$ time) computation scheme. The analysis is a little different from~\cite{GWreg} due to the different types of switchings being used. Switchings in~\cite{GWreg} are classified into different classes (class A and class B), whereas the switchings in the present paper can all be classified as class A under that paradigm. However, the analysis of b-rejections is  more complicated than in~\cite{GWreg}, and as a convenience for the argument we need to introduce  ``pre-states" of a pairing, and ``pre-b-rejections". These concepts will be discussed in detail in Section~\ref{sec:outline}. 

As in~\cite{GWreg}, under various conditions on $M_k$, $H_k$ and $L_k$ for several small values of $k$, we bound the probability of any rejection during Phases 3--5 by $o(1)$ (see Lemmas~\ref{lem:xi}, \ref{lem:typerej}, \ref{lem:drejf}, \ref{lem:pre-b-rej}, \ref{lem:drejb}, \ref{lem:lrej}, and~\ref{lem:trej}). These conditions set constraints on $(\gamma,\delta)$. Eventually we will show that if $\gamma>21/10+\sqrt{61}/10\approx 2.881024968$ then there exists $\delta>0$ satisfying all conditions required in these lemmas; see Lemma~\ref{lem:uniformity}.

As mentioned before, we postpone the discussions of Phases 3 and 4 till Section~\ref{sec:loop} and we start our discussions of Phase 5 by assuming 
 that the algorithm enters Phase 5 with a uniform random pairing $P\in \a_4$, where $\a_4$ is the set of pairings in $\accept$ containing no loops or triple edges. By the definition of $\accept$, the number of double edges in $P$ is at most $\imax=B_D$, and they are all light. In Phase 5, $\state_i$ \jt{denotes} the set of pairings in $\a_4$ containing exactly $i$ double edges. Then, $\state_0$ is the set of pairings in $\Phi$ that produce simple graphs. 

\subsection{Phase 5: double edge reduction}



We will use five different types of switchings in the phase for the double edge reduction. We introduce one of them now, and the rest after some discussion. The most commonly used type is I, which was shown in less detail in Figure~\ref{f:typeI}. We formally define it as follows. \ss

\no {\em Type I switching}. Take a double edge $z$; note that one end of $z$ must be a light vertex. Label the end points of the two pairs in $z$ by $\{1,2\}$ and $\{3,4\}$ so that points 1 and 3 are in the same vertex and this vertex is light. Choose another two pairs $x$ and $y$ and label their ends by $\{5,6\}$ and $\{7,8\}$ respectively. Label the vertices that containing points $1\le i\le 8$ as shown in Figure~\ref{f:I} (same as Figure~\ref{f:typeI} but with points and vertices labelled as in the definition). If (a) $u_i$, $v_i$, $1\le i\le 3$ are six distinct vertices; (b) both $x$ and $y$ are single edges; (c) there are no edges between $u_1$ and $u_i$ for $i\in\{2,3\}$ and no edges between
$v_1$ and $v_i$ for $i\in\{2,3\}$, then replace the four pairs $\{2i-1,2i\}$, $1\le i\le 4$ by $\{1,5\}$, $\{3,7\}$, $\{2,6\}$ and $\{4,8\}$ as shown in Figure~\ref{f:I}.
\begin{figure}[htb]

 \hbox{\centerline{\includegraphics[width=7cm]{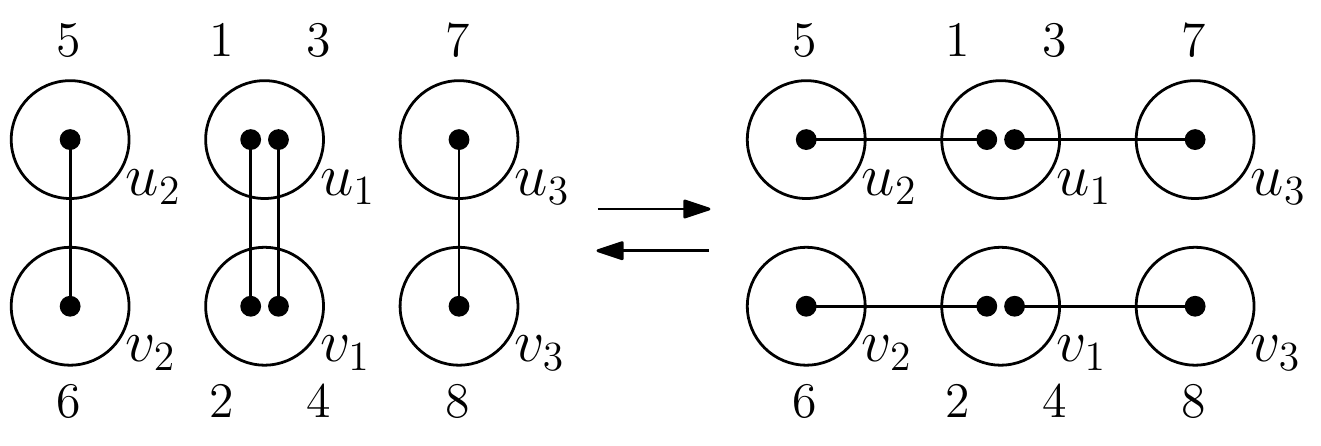}}}
\caption{type I}
\lab{f:I}

\end{figure}

The inverse of the above operation is called an inverse type I switching. Hence, to perform an inverse switching, we will choose a light 2-star and label the points in the two pairs by $\{5,1\}$ and $\{3,7\}$ so that 1 and 3 are in the same light vertex; choose another 2-star (not necessarily light) and label the points in the two pairs by $\{6,2\}$ and $\{4,8\}$ so that 2 and 4 are in the same vertex. Label all involved \jt{vertices} as shown in Figure~\ref{f:I}. If (a) all the six vertices are distinct; (b) both 2-stars chosen are simple; (c) there is no edge between $u_i$ and $v_i$ for all $1\le i\le 3$, then replace all four pairs $\{1,5\}$, $\{3,7\}$, $\{2,6\}$ and $\{4,8\}$ by  $\{2i-1,2i\}$, $1\le i\le 4$.

Note that from the definition of the type I switching, there can be multiple switchings of type I that \jt{switch} a pairing $P$ to another pairing $P'$, caused by different ways of labelling the points in the switched pairs.

 When a type I switching is performed, the number of double edges will decrease by one. Given a pairing $P$, let $\calf_1(P)$ denote the set of mappings from the numbers 5, 1, 3, 7  to the points of $P$ such that $\{5,1\}$ and $\{3,7\}$ are pairs of $P$ and   1 and 3  are in the same vertex, which is light. We call $(\{5,1\},\{3,7\})$ a {\em  light ordered 2-star}. Define $\calf_2(P)$ similarly from the numbers 6, 2, 4, 8 without the restriction to light vertices. Let $\calf(P)=\calf_1(P)\times \calf_2(P)$. The elements of $\calf(P)$ are called {\em doublets}.
As $|\calf_1(P)|=L_2$ and $|\calf_2(P)|=M_2$, we have $|\calf(P)|=M_2L_2$, which is independent of $P$. If a type I switching switches a pairing $P'$ to a new pairing $P$, then it creates a new element in $\calf(P)$ (but not all elements in $\calf(P)$ are created this way). 

Note the labels on the points in each type I switching induce, in an obvious way, a unique doublet in the resulting pairing. Now, for a given pairing $P$, define $\Z_0(P)$ to be the set of doubles in $\calf(P)$ that can be created in this way. 
Then, the number of ways that $P$ can be created via a type I switching is $|\Z_0(P)|$. 
Since $\Z_0(P)\subseteq \calf(P)$, we immediately have $|\Z_0(P)|\le M_2L_2$. For a power-law distribution-bounded degree sequence with $\gamma<3$, $|\Z_0(P)|$ varies a lot among pairings $P\in \state_i$. As we discussed before, a big variation of this number will cause a big probability of a b-rejection. It turns out that five main structures cause the undesired variation of this number: $H_i$, $1\le i\le 5$ in Figure~\ref{f:H}. In order to reduce the probability of a b-rejection, we boost pairings $P$ such that $G(P)$ does not contain sufficient number of copies of $H_0$, or in other words, $G(P)$ contains many copies of $H_i$ for some $1\le i\le 5$.  These force the use of switchings of type III, IV, V, VI and VII respectively (we leave out type II as there is a type of switching called type II used for the $d$-regular case which we do not use in this paper; but it is possible that using this type, together with strengthening other parts can lead to a weaker constraint on $\gamma$). These switchings will be defined in Section~\ref{sec:switchings} and are shown in  Figures~\ref{f:typeIII}  and~\ref{f:IV} -- \ref{f:VII}. Let $\Lambda=\{I,III,IV,V,VI,VII\}$ be the set of switching types to be used.  All types of switchings, except for type I, will switch a pairing to another pairing containing the same number of double edges. In other words, a type $\tau$ ($\tau\neq I$) switching does not change the number of double edges in a pairing.  

The basic method in this part of the algorithm and analysis is similar to that in~\cite{GWreg}  but there are new ingredients, which include the use of pre-states. We briefly describe how and why they are used.

Given a pairing $P$, let $\Z(P)$ be the subset of $\calf(P)$ that includes just those doublets induced (in the way described above) by applying a switching  of any type  to some other pairing to create $P$. Then, each switching (of any type) converting some $P'$ to $P$ induces  a doublet   $\pairstars\in\Z(P)$. However, except for type I, other type of switchings create extra pairs along with $\pairstars$. For instance, a type III switching creates one extra pair whereas a type IV switching 
creates four extra pairs (see Figures~\ref{f:typeIII} and~\ref{f:IV}). These extra pairs modify (slightly) the expected number of times that a given $\pairstars\in \Z(P)$ is created.
 In order to account for this, we will associate with each pairing $P$ a set of pre-states. These pre-states are basically a copy of $P$ plus a designated $\pairstars\in\Z(P)$, together with several pairs, depending on how $P$ is reached via a switching. Then, going from a pre-state to $P$, we may perform a pre-b-rejection. This rejection trims off the fluctuation caused by the extra pairs created together with $\pairstars$, ensuring that every $\pairstars$ \jt{is} created with the same expected number of times. 

For every pairing $P\in \state_i$, we want to make sure that the expected number of times that $P$ is created in the phase is independent of $P$; we do so by equalising the expected number of times that each $\pairstars\in\Z(P)$ is created, after the pre-b-rejection; denote this quantity by $q_i$ (i.e.\ this quantity is independent of $P$). Then the expected number of times that $P$ is created, before the b-rejection, is $q_i|\Z(P)|$. Then, the probability of a b-rejection depends only on the variation of $|\Z(P)|$ over all $P\in \state_i$.  The reason to use type III,IV,V,VI and VII aside from type I is to ensure that the variation of $|\Z(P)|$ is sufficiently small in a class $\state_i$. Let $\Z^*(P)=\calf(P)\setminus \Z(P)$. We will prove that 
$|\Z^*(P)|$ is sufficiently small compared with $|\Z(P)|$.

Next, we partition $\Z(P)$ to $\cup_{i=0}^5 \Z(P)$ according to the different restrictions put on to $\pairstars\in \Z(P)$ when different types of switchings are applied. 
Let $\Z_i(P)$ ($1\le i\le 5$) denote the set of $\pairstars\in\calf(P)$, as shown in Figure~\ref{f:I}, 
 satisfying
\begin{enumerate}
\item[(a)] all these four pairs \{5,1\},\{3,7\},\{6,2\},\{4,8\} are contained only in single edges; $v_1$ is a light vertex; and all six vertices $u_i$ and $v_i$, $1\le i\le 3$, are distinct;
\item[(b)] 
for each $1\le i\le 5$, $\pairstars$ corresponds to a copy of $H_i$ in Figure~\ref{f:H} in $G(P)$.
\end{enumerate}

\begin{figure}[htb]

 \hbox{\centerline{\includegraphics[width=12cm]{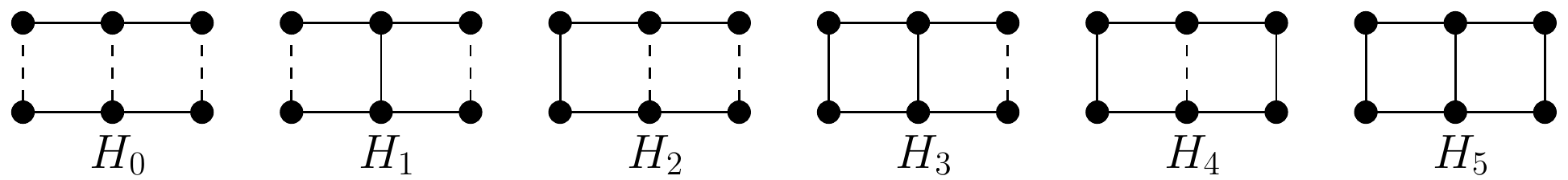}}}
\caption{A pair of $2$-paths}
\lab{f:H}

\end{figure}

Let $Z_i(P)=|\Z_i(P)|$ and $Z^*(P)=|\Z^*(P)|$.
Let $\state_{\tau}^+(P)$ \jt{denote} the number of type $\tau$ switchings that converts a pairing to $P$. Then, $\state_{I}^+(P)=Z_0(P)$. 
Clearly, 
\be
\sum_{i=0}^5Z_i(P)+Z^*(P)=|\calf(P)|=M_2L_2.\lab{Zstar}
\ee
We will show later that the variation of $\sum_{i=0}^5Z_i(P)$ caused by $Z^*(P)$ is sufficiently small for the range of $\gamma$ we consider. 
\remove{
\jcom{It seems that H-graph with the middle edge double also needs to be boosted; the overall probability of performing that switching is $\Delta^2n^{2\delta}/n^2=o(1)$;
  We do not need to boost 5-path with the middle edge double (use the bound on the number of 2-paths from a single vertex). We do not need to boost structures such that one of the 2-paths contains a double edge. 
  
  If we do not use b-class, f-rejection probability for type I is 
  
\[
\frac{\# 2-paths}{M_1}=n^{(2\gamma-3)/(\gamma-1)^2-1}.
\]
  Overall probability is $o(1)$ if $\gamma\ge 2.8$.
  
After raising $\gamma$ to $2.89$ we do not need to boost structures involving 5-vertices only, i.e.\ we do not need switching type V.

  We may need to boost the 6-cycle and the ladder.
  
  For 6-cycle, the probability of performing such a switching is
  $O(\Delta^8/n^6)$; overall probability $o(1)$.
  
  For ladder-minus one edge, the probability of performing such a switching is
  $O(\Delta^5 n^{\delta}/n^4)$; overall probability $o(1)$.
  
  For ladder, the probability of performing such a switching is
  $O(\Delta^9 n^{\delta}/n^7)$; overall probability $o(1)$.
  
  }
}

Note that each type $\tau\in \Lambda$ corresponds to a unique integer $t$ so that the created element $\pairstars\in\Z(P)$, \jt{formed} after a type $\tau$ switching creating a pairing $P$, is in $\Z_{t}(P)$. We define
\[
t=t(\tau) \ \mbox{to be this unique integer.}
\]
 For instance, $t(I)=0$, and by looking at Figures~\ref{f:typeIII},~\ref{f:IV} and~\ref{f:H}, we have $t(III)=1$ and $t(IV)=2$.

Before defining all types of switchings, we first give a general description of the whole approach.

\subsubsection{A brief outline}
\lab{sec:outline}
Note that a pairing can be reached more than once in Phase 5. Let $\sigma(P)$ denote the expected number of times that pairing $P$  
is reached. Our goal is to design the algorithm in such a way   that $\sigma(P)=\sigma(i)$ for every $P\in\state_i$ and for every $0\le i\le \imax$; this ensures that the output of Phase 5, if no rejection happens, is uniform over $\state_0$. \ms

\no {\bf Choosing a switching type}.\ss

 As in~\cite{GWreg}, in Phase 5, the algorithm starts with a uniform distribution on $\a_4$. In each step, the algorithm starts from a pairing $P\in\state_i$ obtained from the previous step, and chooses a switching type according to a pre-determined ``type distribution''. This distribution only depends on $i$, i.e.\ the class where $P$ is in, and is independent of $P$. We will use $\rho_{\tau}(i)$
to denote the probability that type $\tau$ switching is chosen for a pairing in $\state_i$. 
We require $\sum_{\tau}\rho_{\tau}(i)\le 1$ for every $i$. If this inequality is strict, then with the remain probability, the algorithm terminates. Such a rejection is called a t-rejection. \ms

 \no {\bf An f-rejection}.\ss
  
         If a type $\tau$ is chosen, then the algorithm chooses a random type $\tau$ switching $S$ that can be performed on $P\in\state_i$. An f-rejection may happen in order to equalise the expected number of times a particular type $\tau$ switching is performed, for every $P\in\state_i$. Let $f_{\tau}(P)$ denote the number of type $\tau$ switchings that can be performed on $P$, then, we will perform an f-rejection with probability
          \be\lab{frejDef}
          1-\frac{f_{\tau}(P)}{\UB_{\tau}(i)},
          \ee
          where $\UB_{\tau}(i)$ will be specified later as an upper bound for $\max_{P\in\state_i} f_{\tau}(P)$.
          Therefore, every switching $S$ is chosen and is not f-rejected with probability $\rho_{\tau}(i)/\UB_{\tau}(i)$. Hence, the expected number of times that $S$ is performed during Phase 5 is
          \[
          \sigma(i)\frac{\rho_{\tau}(i)}{\UB_{\tau}(i)},
          \]
          assuming that each pairing in $\state_i$ is reached with $\sigma(i)$ expected number of times.
 \ms
 
\no {\bf Pre-states and pre-b-rejection}.\ss         

We first give a formal definition of pre-states of a pairing $P'\in\state_j$. 
Given pairing $P'\in\state_j$, a switching type $\tau$ and an arbitrary $\pairstars\in \Z_{t(\tau)}(P')$, a pre-state of $P'$, corresponding to $\tau$ and $\pairstars$ is a copy of $P'$ in which $\pairstars$ is designated together with a set of pairs $X$ such that an inverse type $\tau$ switching can be performed on $X$ and $\pairstars$. Let $\B_{\tau}(P,\pairstars)$ denote the set of pre-states of $P$ corresponding to switching type $\tau$ and $\pairstars\in \Z_{t(\tau)}(P)$ and let $\b_{\tau}(P,\pairstars)=|\B_{\tau}(P,\pairstars)|$.  Define $\b_{\tau}(P)=\min_{\pairstars\in \Z_{t(\tau)}}\b_{\tau}(P,\pairstars)$. We will define later $\LBS_{\tau}(i)$ as a lower bound for $\min_{P\in\state_i}\b_{\tau}(P)$. In particular,  we will show that for every $0\le i\le \imax-1$, $\b_{I}(P,\pairstars)=1$ for any $P\in\state_i$ and for any $\pairstars\in \Z_{t(I)}(P)=\Z_0(P)$ (see~\eqn{LBS1}), which immediately implies that
\be
\LBS_{I}(i)=1, \quad \mbox{for all}\ 0\le i\le \imax-1. \lab{LBS1}
\ee

Now, let $P'$ be the pairing that $S$ converts $P$ into and let $\pairstars$ be the new pair of 2-stars created by $S$. Then, $\pairstars\in\Z_{t(\tau)}(P')$. Assume $P'\in\state_j$. Note that $S$ designates a pre-state of $P'$. We will perform a pre-b-rejection before entering $P'$, to ensure that every $\pairstars\in \Z_{t(\tau)}(P')$ are created with equal expected number of times, for every $P'\in\state_j$.

A pre-b-rejection will be performed in the following way. Note that our algorithm has \jt{been} designed so that given $\tau$, $P\in\state_i$ and $\pairstars\in \Z_{t(\tau)}(P)$, all pre-states of $P$ corresponding to $\tau$ and $\pairstars$ will be reached with equal expected number of times, which is
\[
y_i:=\sigma(i)\frac{\rho_{\tau}(i)}{\UB_{\tau}(i)},
\]
 which is
   independent of $P$, $P'$ and $\pairstars$. Then, a pre-b-rejection will be performed  with probability
\[
1-\frac{\LBS_{\tau}(j)}{\b_{\tau}(P',\pairstars)}.
\]
Now, for any $\pairstars\in \pairstars\in \Z_{t(\tau)}(P')$, the expected number of times that $\pairstars$ is created is
\be
\sum_{s\in \B_{\tau}(P',\pairstars)} y_i \frac{\LBS_{\tau}(j)}{\b_{\tau}(P',\pairstars)} = y_i\LBS_{\tau}(j)=\sigma(i)\frac{\rho_{\tau}(i)}{\UB_{\tau}(i)}\LBS_{\tau}(j),\lab{qi0}
\ee
which is independent of $\pairstars$. This ensures that every $\pairstars\in \Z_{t(\tau)}(P')$ are created with equal expected number of times, as desired. In fact, we will equalise this quantity for every switching type $\tau$ as well. This will be done by choosing $\rho_{\tau}(i)$ properly. We will later let $q(j)$ denote this equalised quantity, which depends only on $j$, i.e.\ the class $P'$ is in. \ms

\no {\bf A b-rejection}\ss

    If $S$ is not f-rejected or pre-b-rejected, then a b-rejection will be performed with probability
    \[
    1-\frac{\LB(j)}{b(P')},
    \]
    where $b(P')=|\Z(P')|$ and 
$\LB(j)$ will be a specified lower bound for $\min_{P'\in\state_j}|\Z(P')|$.    
    If no b-rejection occurred, then $P$ is switched to $P'$ and that completes a switching step. 
    The expected number of times that
      $P'$ is reached via a switching is
    \be
    q(j) \LB(j),\lab{sigma0}
    \ee
    as the expected number of times a $\pairstars\in\Z(P')$ is created is equalised to $q_j$, as we mentioned. This confirms that the expected number of times that $P'$ is reached is dependent only on the state in which $P'$ lies.\ms

\no {\bf The Markov chain}.\ss        

  This defines a Markov chain run on $\a_4$. Each state in the Markov chain is a pairing in $\a_4$ together with $\{\O,\R\}$; $\O$ is the state representing an output and $\R$ is the state representing a rejection.  A switching step from $P$ to $P'$ is a transition from $P$ to $P'$ in the chain. If a rejection (of any kind) happens in a step, then the chain enters state $\R$ and the algorithm terminates without an output. If $P\in \state_0$, then we interpret a type I switching on $P$ as outputting $P$. So if the chain reaches a pairing in $\state_0$, then it enters $\O$ in the next step with probability $\rho_I(0)$.

 \hbox{\centerline{\includegraphics[width=5cm]{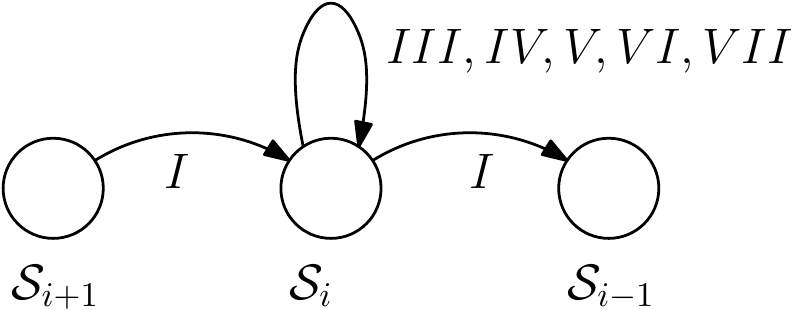}}}

Recall the following quantities that we are going to specify later:
\bean
\UB_{\tau}(i) && \mbox{an upper bound for}\ \max_{P\in \state_i} f_{\tau}(P);\\
\LBS_{\tau}(i) && \mbox{a lower bound for}\ \min_{P\in \state_i} \b_{\tau}(P)=\min_{P\in \state_i}\min_{\pairstars\in \Z_{i(\tau)}}\b_{\tau}(P,\pairstars);\\
\LB(i) && \mbox{a lower bound for}\ \min_{P\in \state_i} b(P).
\eean
\remove{
\subsection{Transition probabilities}
Given $P\in\state_i$ and a switching type $\tau$, the probability that type $\tau$ is chosen is $\rho_{\tau}(i)$. Assume $\tau$ is chosen. Then, an $f$-rejection occurs with probability $1-f_{\tau}(P)/\UB_{\tau}(i)$. Assume no $f$-rejection occurs, then each of the $f_{\tau}(P)$ switchings are chosen with probability $1/\UB_{\tau}(i)$. Assume $S$ is one of these switching that converts $P$ to another pairing $P'\in\state_j$. The number of pre-states of $P'$ corresponding to switching type $\tau$ is $\b_{\tau}(P')$. A pre-b-rejection is performed with probability $1-\LBS_{\tau}(j)/\b_{\tau}(P')$. Hence, given $P$, the probability that $S$ is chosen and is not f-rejected or pre-b-rejected is
\be
\frac{\rho_{\tau}(i)}{\UB_{\tau}(i)} \frac{\LBS_{\tau}(j)}{\b_{\tau}(P')}.\lab{after-pre-b-rej} 
\ee
Now, the probability that $P'$ is not b-rejected is $b(P')/\LB(j)$. Hence, the probability that $S$ is performed is
\[
\frac{\rho_{\tau}(i)}{\UB_{\tau}(i)} \frac{\LBS_{\tau}(j)}{\b_{\tau}(P')}\frac{b(P')}{\LB(j)}.
\]
Given $P\in\state_i$ and $P'\in\state_j$ and $\tau\in \Lambda$: the set of all switching types, let $\state_{\tau}(P,P')$ denote the set of type $\tau$ switchings that switches $P$ to $P'$. Then,
\[
q_{P,P'}=\sum_{\tau}\sum_{S\in \state_{\tau}(P,P')} \frac{\rho_{\tau}(i)}{\UB_{\tau}(i)} \frac{\LBS_{\tau}(j)}{\b_{\tau}(P')}\frac{b(P')}{\LB(j)}.
\]
}
In the next section, we will deduce a system involving $\rho_{\tau}(i)$ and $\sigma(i)$ using the above quantities, and use the solution of that system to specify $\rho_{\tau}(i)$, the type distribution function to be used in the algorithm.

\subsubsection{The system with $\rho_{\tau}(i)$ and $\sigma(i)$}

Note that $P_t$ is the pairing obtained after the $t$-th step of Phase 5.
Initially, each pairing in $\a_4$ can be $P_0$ with probability $1/|\a_4|$.

Note that while performing switchings, a pairing can be visited more than once. Recall that $\sigma(P)$ denotes the expected number of times that a given pairing $P$ is visited. We wish to ensure that $\sigma(P)=\sigma(i)$ for every $P\in\state_i$. 

Given a pairing $P\in\state_i$,  we want to equalise the expected number of times that any pair of 2-stars in $\Z(P)$ is created, after the pre-b-rejection. Recall that $q(i)$ denotes this quantity (see the discussion below~\eqn{qi0}). We first consider switchings of type I.  For such a switching $S$ converting $P'$ to $P$, $P'$ must be in $\state_{i+1}$. By~\eqn{qi0} (with $i$ in~\eqn{qi0} being $i+1$ and $j$ in~\eqn{qi0} being $i$ and $\tau=I$) and~\eqn{LBS1}, 
\be
q(i)=\sigma(i+1)\frac{\rho_{I}(i+1)}{\UB_{I}(i+1)}\LBS_{I}(i)=\sigma(i+1)\frac{\rho_{I}(i+1)}{\UB_{I}(i+1)}.	\lab{qi}
\ee
For any switching type $\tau\neq I$, a switching of type $\tau$ will switch a pairing to another pairing in the same class. So, for each $\tau\in\Lambda\setminus \{I\}$, let $S$ be a switching of type $\tau$ that converts a pairing $P'\in\state_i$ to $P$.
By~\eqn{qi0},
\[
q(i)=\sigma(i)\frac{\rho_{\tau}(i)}{\UB_{\tau}(i)}\LBS_{\tau}(i).
\]
Combining this with~\eqn{qi}, we must have
\[
\rho_{\tau}(i)=\frac{q(i)\UB_{\tau}(i)}{\sigma(i) \LBS_{\tau}(i)}=\frac{\sigma(i+1)}{\sigma(i)}\frac{\rho_{I}(i+1)}{\UB_{I}(i+1)} \frac{\UB_{\tau}(i)}{ \LBS_{\tau}(i)},
\]
for each $\tau\in \Lambda\setminus\{I\}$.


Finally, by~\eqn{sigma0} and~\eqn{qi},
\[
\sigma(i)=\LB(i)q(i)+1/|\a_4|=\sigma(i+1)\frac{\rho_{I}(i+1)}{\UB_{I}(i+1)}+1/|\a_4|,
\]
where $1/|\a_4|$ is the probability that $P_0=P$.
The boundary condition is $q(\imax)=0$.

Letting $x_i=\sigma(i)|\a_4|$, we have
\remove{
\bea
x_{\imax}&=&x_{\imax}\rho_I(\imax) \frac{\LB_B(\imax)}{\UB_{I}(\imax)}+1;\lab{initial}\\
x_0&=&x_{1}\rho_I(1)\frac{\LB_A(0)}{\UB_I(1)}+1;\\
x_i&=&x_{i+1}\rho_I(i+1)\frac{\LB_A(i)}{\UB_I(i+1)}+x_i\rho_I(i)\frac{\LB_B(i)}{\UB_{I}(i)}+1;\\
\rho_{\tau}(i)&=&\rho_I(i+1)\frac{x_{i+1}}{x_i}\frac{\UB_{\tau}(i)}{\UB_I(i+1) \LBS_{A,\tau}(i)}\quad \tau\in\{III,IV,V\};\\
\rho_{II}(i)&=&\rho_I(i+1)\frac{x_{i+1}}{x_i} \frac{\UB_{II}(i)}{\UB_{I}(i+1)} ;\\
\rho_{\tau}(i)&\ge &0; \rho_{\tau}(\imax)=0\ (\tau\neq I); \rho_{\tau}(0)=0; \sum_{\tau}\rho_{\tau}(i)\le 1.\lab{atmost1}
\eea
}
\bea
x_{\imax}&=&1;\lab{initial}\\
x_i&=&x_{i+1}\rho_I(i+1)\frac{\LB(i)}{\UB_I(i+1)}+1\quad \forall 0\le i\le \imax-1;\lab{xrec}\\
\rho_{\tau}(i)&=&\rho_I(i+1)\frac{x_{i+1}}{x_i}\frac{\UB_{\tau}(i)}{\UB_I(i+1) \LBS_{\tau}(i)}\quad \forall\tau\in\Lambda\setminus\{I\}, 1\le i\le \imax-1;\lab{rhorec}\\
\rho_{\tau}(i)&\ge &0;\quad \rho_{\tau}(\imax)=0\ (\tau\neq I);\quad  \sum_{\tau}\rho_{\tau}(i)\le 1.\lab{atmost1}
\eea

It follows immediately that
\be
\frac{x_{i+1}}{x_i}\le \frac{1}{\rho_I(i+1)} \frac{\UB_I(i+1)}{\LB(i)}.\lab{xratio}
\ee

We will show that there is a solution $\rho^*_{\tau}(i)$ and $x^*_i$ to the above system with $\rho^*_{I}(i)$ close to 1 for every $i$ and $\sum_{\tau\in\Lambda}\rho^*_{\tau}(i)=1-o(\imax)$.  Then we will set the type probability $\rho_{\tau}(i)$ in the algorithm to be $\rho^*_{\tau}(i)$ for every $\tau$ and $i$. A lemma in~\cite{GWreg} shows that with this setting, the expected number of times that a pairing in $\state_i$ is reached is $x^*_i/|\a_4|$, independent of $P\in \state_i$. Therefore, for each $P\in\state_0$, the probability that $P$ is outputted is equal to $(x^*_0/|\a_4|)\rho_I(0)$, and thus 
\begin{equation}
\mbox{the output of \PL\ is uniformly distributed in $\state_0$. }\lab{eq:uniform}
\end{equation}
Before proving the existence of such a solution, we need to define all types of \jt{switchings} and then \jt{specify} quantities $\UB_{\tau}(i)$, $\LB(i)$ and $\LBS_{\tau}(i)$ that have appeared in the system.

\subsubsection{Switchings}
\lab{sec:switchings}

In the definition of a switching (of each type), we will choose a set of pairs and label their end points in a certain way; then we replace this set of pairs by another set, without changing the given degree sequence. 
The type I switching has been formally defined in detail. The definition of other types is clear by the illustrations in Figures~\ref{f:typeIII}, and~\ref{f:IV}~--~\ref{f:VII}. All pairs involved in the switchings and the inverse switchings are required to be contained only in single edges.
\remove{
\no {\bf Type I switching}. Take a double edge $z$; note that one end of $z$ must be a light vertex. Label the end points of the two pairs in $z$ by $\{1,2\}$ and $\{3,4\}$ so that points 1 and 3 are in the same vertex and this vertex is light. Choose another two pairs $x$ and $y$ and label their ends by $\{5,6\}$ and $\{7,8\}$ respectively. Label the vertices that containing points $1\le i\le 8$ as shown in Figure~\ref{f:I}. If (a) $u_i$, $v_i$, $1\le i\le 3$ are six distinct vertices; (b) both $x$ and $y$ are single edges; (c) there are no edges between $u_1$ and $u_i$ for $i\in\{2,3\}$ and no edges between
$v_1$ and $v_i$ for $i\in\{2,3\}$, then replace the four pairs $\{2i-1,2i\}$, $1\le i\le 4$ by $\{1,5\}$, $\{3,7\}$, $\{2,6\}$ and $\{4,8\}$ as shown in Figure~\ref{f:I}.

The inverse of the above operation is called an inverse type I switching. Hence, to perform an inverse switching, we will choose a light 2-star and label the points in the two pairs by $\{5,1\}$ and $\{3,7\}$ so that 1 and 3 are in the same light vertex; choose another 2-star (not necessarily light) and label the points in the two pairs by $\{6,2\}$ and $\{4,8\}$ so that 2 and 4 are in the same vertex. Label all involved vertice as shown in Figure~\ref{f:I}. If (a) all the six vertices are distinct; (b) both 2-stars chosen are simple; (c) there is no edge between $u_i$ and $v_i$ for all $1\le i\le 3$, then replace all four pairs $\{1,5\}$, $\{3,7\}$, $\{2,6\}$ and $\{4,8\}$ by  $\{2i-1,2i\}$, $1\le i\le 4$.
}

A type I switching creates only a pair of 2-stars in $\Z_0(P)$ for some $P$ and no other extra pairs. Hence, given any $P\in\state_i$ ($i\le \imax-1$) and any $\pairstars\in\Z_0(P)$, there is only one pre-state of $P$ corresponding to type $I$ and $\pairstars$. In other words,
for every $0\le i\le \imax-1$, $\b_{I}(P,\pairstars)=1$ for any $P\in\state_i$ and for any $\pairstars\in \Z_{t(I)}(P)=\Z_0(P)$. This confirms~\eqn{LBS1}, i.e.
\[
\LBS_{I}(i)=1, \quad \mbox{for all}\ 0\le i\le \imax-1.
\]

\begin{figure}[htb]

 \hbox{\centerline{\includegraphics[width=8cm]{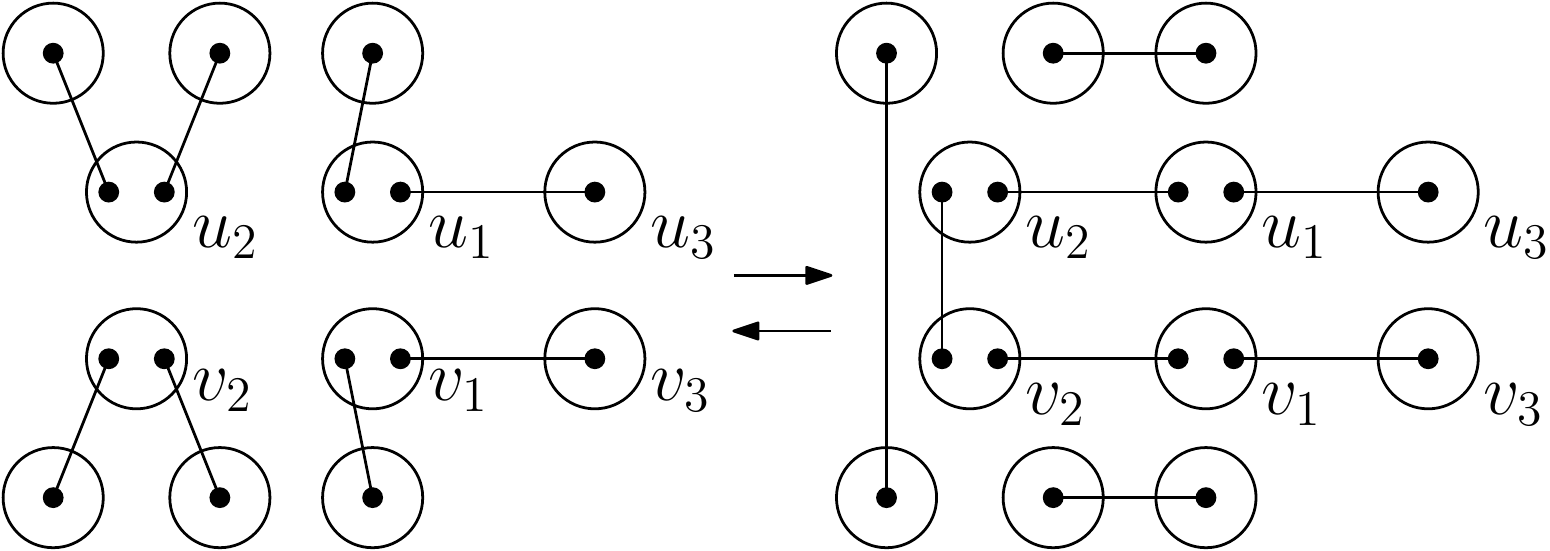}}}
\caption{type IV}
\lab{f:IV}

\end{figure}

\begin{figure}[htb]

 \hbox{\centerline{\includegraphics[width=8cm]{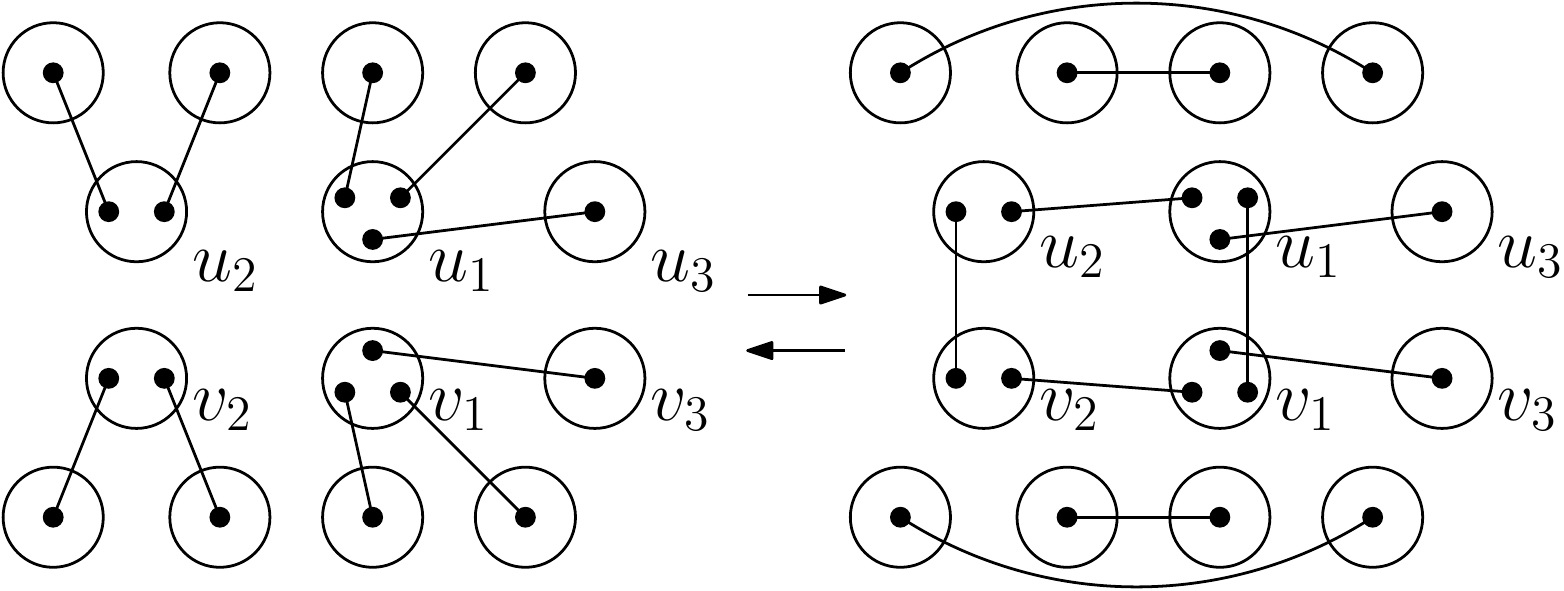}}}
\caption{type V}
\lab{f:V}

\end{figure}

\begin{figure}[htb]

 \hbox{\centerline{\includegraphics[width=10cm]{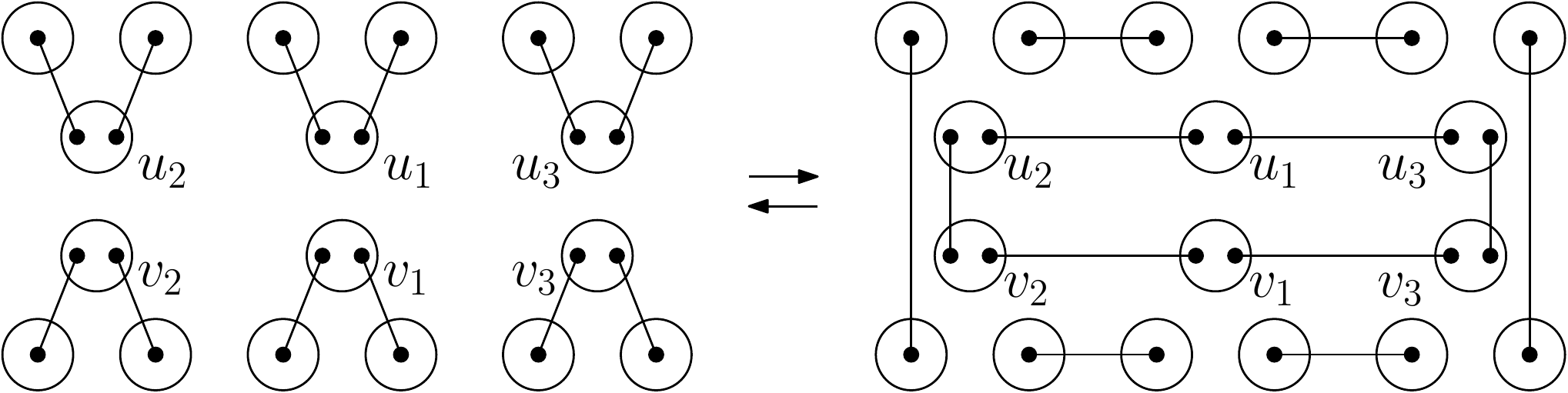}}}
\caption{type VI}
\lab{f:VI}

\end{figure}
\begin{figure}[htb]

 \hbox{\centerline{\includegraphics[width=10cm]{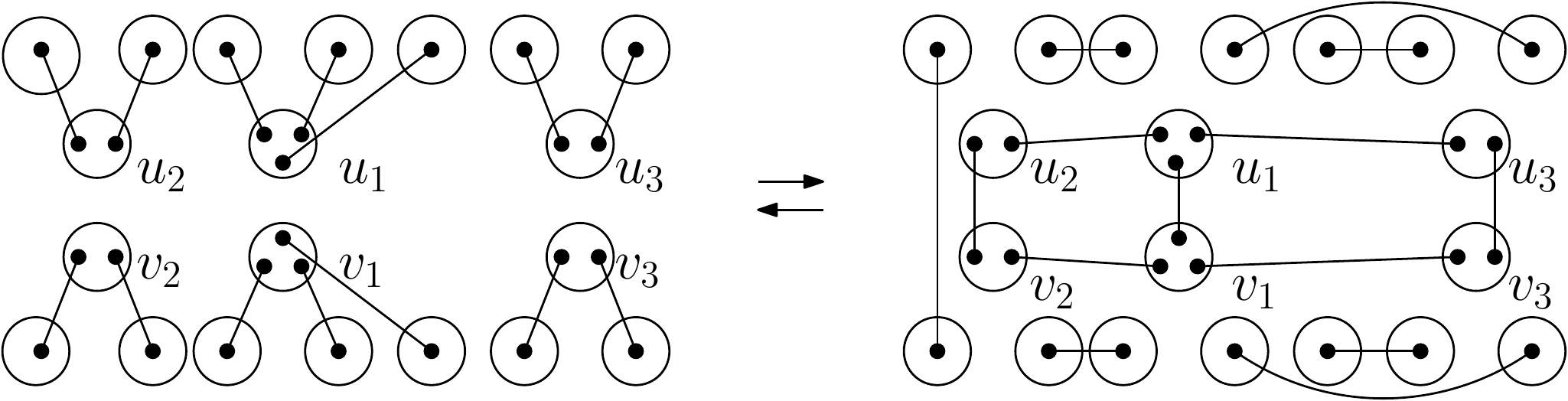}}}
\caption{type VII}
\lab{f:VII}

\end{figure}


\subsubsection{Specifying $\UB_{\tau}(i)$, $\LB(i)$ and $\LBS_{\tau}(i)$}

Define
\bea
&&\UB_I(i)= 4iM_1^2;\quad \UB_{III}(i)=M_3 L_3;\quad \UB_{IV}(i)=2M_2^3L_2;\lab{UB1} \\ 
&&\UB_{V}(i)=2M_2^2M_3L_3;\quad \UB_{VI}(i)=M_2^5L_2;\quad \UB_{VII}(i)=M_2^4 M_3L_3.\lab{UB2}
\eea
It is easy to verify that for every $P\in\state_i$ and $\tau\in \Lambda$, $f_{\tau}(P)\le \UB_{\tau}(i)$ for each $0\le i\le \imax$. The factor 2 in $\UB_{IV}(i)$ and $\UB_{V}(i)$ accounts for the two cases that (a) $u_2$ is adjacent to $v_2$; (b) $u_3$ is adjacent to $v_3$.

  Define
\bea
U_1&=&(Kn)^{1/(\gamma-1)}(Kn)^{(\gamma-2)/(\gamma-1)^2}=  (Kn)^{(2\gamma-3)/(\gamma-1)^2};\lab{0U1}\\
U_k&=& \frac{\gamma K^{k/(\gamma-1)}}{k+1-\gamma} n^{1-\delta(\gamma-k-1)}\quad \mbox{for}\ k\ge 2.\lab{0Uk}
\eea
The $U_k$ ($k\ge 1$) are upper bounds for the number of various types of structures in $P\in\a_0$. For instance, $U_1$ is an upper bound for the number of $2$-paths with one end at a given vertex (see the Appendix for details).  These $U_k$ will be used to define parameters such as $\LB(i)$ and to bound the probabilities of f-rejections, pre-b-rejections and b-rejections in Phase 5. 
        
Next we specify $\LB(i)$ and $\LBS_{\tau}(i)$. 
Define
\bel{LB0}
\LB(i)=M_2L_2-8i(d_h M_2+d_1 L_2)-(2id_1^2d_h^2+4iU_1^2+8M_2 U_1+L_4).
\ee
and
\bea
&&\LBS_I(i)=1; \ \LBS_{III}(i)=M_1-2U_1-4i;\ \LBS_{IV}(i)=(M_1-2U_1-4i)^3;\lab{LB1}\\
&&\LBS_{V}(i)=(M_1-2U_1-4i)^4;\ \LBS_{VI}(i)=(M_1-2U_1-4i)^6;\lab{LB2}\\ 
&&\LBS_{VII}(i)=(M_1-2U_1-4i)^7.\lab{LB3}
\eea

Simple counting arguments yield the following lemma, whose proof is presented in the Appendix.
\begin{lemma}\lab{lem:lowerbounds}
\begin{enumerate}
\item[(a)] For every $0\le i\le \imax$, $P\in\state_i$, $b(P)\ge \LB(i)$. 
\item[(b)] For every $\tau\in\Lambda$, $0\le i\le \imax$, $P\in\state_i$ and $\pairstars\in \Z_{t(\tau)}$, $\b_{\tau}(P,\pairstars)\ge \LBS_{\tau}(i)$.
\end{enumerate}
\end{lemma}


\subsection{Specifying $\rho_{\tau}(i)$}
\lab{sec:rho}

Let $\xi=\xi_n>0$ be a function of $n$ to be specified later. Consider the system~\eqn{initial}--\eqn{atmost1}. We will show that there is a unique solution $(x^*_i,\rho^*_{\tau}(i))$ to the system such that $\rho^*_I(i)+\rho^*_{III}(i)=1-\xi$  and $\sum_{\tau\neq I, III}\rho^*_{\tau}(i)<\xi$ for all $i$. 

\remove{
\[
\UB_{III}(i)\approx M_3 L_3;
\UB_{IV}(i)\approx M_2^3 L_2M_1; \UB_{V}(i)\approx M_3 L_2.
\]
Then we have
\bean
\rho_{III}(i)&=&O(\Delta n^{\delta}/n)\\
\rho_{IV}(i)&=&O(\Delta^4/n^3)\\
\rho_{V}(i)&=&O(\Delta/n).
\eean
It is easy to see that $\rho(IV) \imax=o(1)$ and $\rho(V)\imax=o(1)$. We only need to care about type III.

We can specify a function $\xi=o(1/\imax)$ such that $\sum_{\tau\notin\{I,III\}}\rho_{\tau}(i)\le \xi$. 
}

We first specify a recursive algorithm to compute $\rho^*_I(i)$ and $\rho^*_{III}(i)$ and $x^*_i$ satisfying~\eqn{initial}--\eqn{rhorec} so that
\[
\rho^*_I(i)+\rho^*_{III}(i)+\xi=1,\quad \mbox{for all $i$}.
\]

Base case: $\rho^*_I(\imax)=1-\xi$ and $\rho^*_{III}(\imax)=0$. From this we can compute $x^*_{\imax}$.

Inductive step: assume that we have computed $\rho^*_{I}(i+1)$, $\rho^*_{III}(i+1)$ and $x^*_{i+1}$.
For $i$: we can compute $\rho^*_{I}(i)$, $\rho^*_{III}(i)$ and $x^*_i$ by solving the following system.
\bean
x^*_i&=&x^*_{i+1}\rho^*_I(i+1)\frac{\LB(i)}{\UB_I(i+1)}+1;\\
\rho^*_{III}(i)&=&\rho^*_I(i+1)\frac{x^*_{i+1}}{x^*_i}\frac{\UB_{III}(i)}{\UB_I(i+1) \LBS_{III}(i)};\\
\rho^*_I(i)+\rho^*_{III}(i)&=&1-\xi.
\eean
Putting 
\[
C_1=x^*_{i+1}\rho^*_I(i+1)\frac{\LB(i)}{\UB_I(i+1)}+1;\ \ C_2= \rho^*_I(i+1)x^*_{i+1}\frac{\UB_{III}(i)}{\UB_I(i+1) \LBS_{III}(i)},
\]
we have 
\bean
x^*_i&=&C_1;\\
\rho^*_I(i)&=&1-\xi-\frac{C_2}{C_1};\\
 \rho_{III}(i)&=&\frac{C_2}{C_1}.
\eean
Hence, all $\rho^*_I(i)$, $\rho^*_{III}(i)$ and $x^*_i$ can be uniquely determined, once we fix $\xi$. We only have to verify that the solution satisfies~\eqn{atmost1} by choosing $\xi$ appropriately.
Define
\bel{xi}
\xi=\frac{32M_2^2}{M_1^3}.
\ee
\begin{lemma}\lab{lem:xi}
Assume $\delta+1/(\gamma-1)<1$ and $\xi=o(1)$.  Then,
\[
\max_{0\le i\le \imax}\sum_{\tau\in\Lambda\setminus\{I,III\}} \frac{\UB_{\tau}(i)}{\LBS_{\tau}(i)\LB(i)}\le \xi.
\]
\end{lemma}
\proof Recall~\eqn{UB1}~--~\eqn{LB3} and it is easy to see that $\LB(i)\ge M_2L_2/2$ and each $\LBS_{\tau}(i)$ is at least a half of the first term in their expressions. Hence, for every $0\le i\le \imax$,
\[
\sum_{\tau\in\Lambda\setminus\{I,III\}} \frac{\UB_{\tau}(i)}{\LBS_{\tau}(i)\LB(i)}\le \frac{2}{M_2L_2}\left(\frac{4M_2^3L_2}{M_1^3}+\frac{4M_2^2M_3L_3}{M_1^4}+\frac{2M_2^5L_2}{M_1^6}+\frac{2M_2^4M_3L_3}{M_1^7}\right).
\]
Using the inequality that $L_3\le L_2d_h=O(L_2n^{\delta})$, and the fact that $M_2=O(n^{2/(\gamma-1)})$ by~\eqn{Delta}, and our assumptions that $\delta+1/(\gamma-1)<1$ and $\xi=o(1)$, the last three terms in the above formula are dominated by the first term. Hence,
\[
\sum_{\tau\in\Lambda\setminus\{I,III\}} \frac{\UB_{\tau}(i)}{\LBS_{\tau}(i)\LB(i)}\le \frac{32}{M_2L_2}\frac{M_2^3L_2}{M_1^3}=\xi.
\]
This holds uniformly for all $0\le i\le \imax$ and this completes the proof of this lemma.\qed\ss

Then we have the following lemma.
\begin{lemma}\lab{lem:solution} Assume $\xi=o(1)$ and $M_3L_3/M_2L_2M_1=o(1)$. System~\eqn{initial}--~\eqn{atmost1} has unique solution $(x^*_i,\rho^*_{\tau}(i))$ with $\rho^*_{I}(i)+\rho^*_{III}(i)=1-\xi$ for every $i$. Moreover,
\bean
\rho^*_I(i)&=&1-O\left(\frac{M_3L_3}{M_2L_2 M_1}\right);\quad \rho^*_{III}(i)=O\left(\frac{M_3L_3}{M_2L_2 M_1}\right);\\
\rho^*_{\tau}(i)&=&O\left(\frac{\UB_{\tau}(i)}{\LB(i)\LBS_{\tau}(i)}\right)\quad \mbox{for all}\  \tau\in\Lambda\setminus\{I,III\}.
\eean
\end{lemma}
\begin{proof}
We have proven the uniqueness.
We only need to verify that the solution given above satisfies the inequalities in~\eqn{atmost1}. It is obvious that both $x_i>0$ for each $0\le i\le \imax$. 
This implies that $\rho_{III}(i)>0$. Next, we bound $\rho_{III}(i)$ and $\rho_{\tau}(i)$ for $\tau\in \Lambda\setminus\{I,III\}$. 
Since $C_1\ge x^*_{i+1}\rho^*_I(i+1)\LB(i)/\UB_I(i+1)$ we have
\bean
\rho_{III}(i)&=&\frac{C_2}{C_1}\le \frac{\UB_{III}(i)}{\LBS_{III}(i)\LB(i)}.
\eean
By~\eqn{UB1},~\eqn{LB0} and~\eqn{LB1},
\[
\rho_{III}(i)=O\left(\frac{M_3L_3}{M_2L_2 M_1}\right)=o(1).
\] 
This verifies that $0<\rho_{III}(i)<1$ and  $0<\rho_I(i)<1$ for every $i$.

Moreover, by~\eqn{rhorec} and~\eqn{xratio}, for each $\tau\in\Lambda\setminus\{I,III\}$,
\[
\rho_{\tau}(i)\le \frac{x^*_{i+1}\rho^*_{I}(i+1)}{x^*_i}\frac{\UB_{\tau}(i)}{\UB_I(i+1)\LBS_{\tau}(i)}\le \frac{\UB_{\tau}(i)}{\LBS_{\tau}(i)\LB(i)}.
\]
Therefore,
\[
\sum_{\tau\in \Lambda}\rho_{\tau}(i)=1-\xi+\sum_{\tau\in \Lambda\setminus\{I,III\}}\rho_{\tau}(i)\le 1-\xi+\xi=1,
\]
by~\eqn{xi} and so~\eqn{atmost1} is satisfied. Hence, $(x^*_i,\rho^*_{\tau}(i))$ is the unique solution of~\eqn{initial}--~\eqn{atmost1} satisfying $\rho_{I}(i)+\rho_{III}(i)=1-\xi$ for every $0\le i\le \imax$. \qed
\end{proof}

In the rest of the paper, we fix $\xi$ as in~\eqn{xi} and we set the type distribution $\rho_{\tau}(i)$ to be $\rho^*_{\tau}(i)$, the unique solution specified in Lemma~\ref{lem:solution}.

\subsubsection{Probability of a t-rejection}

\begin{lemma}\lab{lem:steps}
The expected number of iterations during the phase of double edge reduction is at most $10B_D$. Moreover, with probability $1-o(1)$, the phase of double edge reduction terminates within $8B_D$ iterations.
\end{lemma}
\begin{proof}
Let $t\ge 8B_D$ be an integer. Assume the phase does not stop before step $t$. Let $x$ denote the number of steps among the first $t$ steps that the number of double edges does not decrease. Since the number of double edges in $P_0$ is at most $\imax$, if Phase \jt{5} lasts at least $t$ steps then $i_1+x-(t-x)\ge 0$, since the number of double edges in each of the $x$ steps will increase by at most one whereas in the other $t-x$ steps it decreases by exactly one. It follows then that $x\ge (t-\imax)/2$. The probability that the number of double edges (if there is any) does not decrease in any step is bounded by $1-\min_{1\le i\le \imax}\{\rho_{I}(i)\}<1/10$. Hence, the probability that $x\ge (t-\imax)/2$ is at most
\[
\binom{t}{(t-\imax)/2} (1/10)^{(t-\imax)/2}\le 2^t (1/10)^{7t/16}<0.8^t,
 \] 
as $(t-\imax)/2\le (t-t/8)/2=7t/16$. Putting $t=8\imax$, the above probability is clearly $o(1)$. Hence, with probability $1-o(1)$, this phase  lasts for at most $8\imax$ steps. The expected number of steps in this phase  is at most
\[
8\imax+\sum_{t\ge 8\imax} 0.8^t \le 10\imax.\qed
\]

\end{proof}

Recall that $B_D$ from~\eqn{imax} and recall that $\imax$ is set \jt{to} $B_D$ in Phase 5.
\begin{lemma}\lab{lem:typerej}
Assume $B_D\xi=o(1)$. Then with probability $1-o(1)$, \jt{no} $t$-rejection happens in Phase 5.
\end{lemma}
\begin{proof}
By the definition of $\rho_{\tau}(i)$, the probability of a $t$-rejection in each iteration is at most $\xi$. By Lemma~\ref{lem:steps}, a.a.s.\ the number of iterations in this phase is at most $8\imax$. Hence, the probability of a $t$-rejection during this phase is at most 
$8\imax \xi+o(1)=o(1)$.
\end{proof}

\subsubsection{Probability of an f-rejection, a pre-b-rejection and a b-rejection.}

By bounding the difference between $f_{\tau}(P)$ from $\UB_{\tau}(i)$ we can bound the probability of an f-rejection; similarly, by bounding the difference between $\b(P,\pairstars)$ and $\LBS_{\tau}(i)$, and bounding the difference between $b(P)$ and $\LB(i)$, we can bound the probability of a pre-b-rejection and 
a b-rejection in Phase 5. The arguments are similar to that in~\cite{GWreg}. Thus we just state the assumptions needed to have a small probability of a rejection and leave the proofs in the Appendix.
   
      \begin{lemma} \lab{lem:drejf}
Assume that $B_D\xi=o(1)$, $B_D(M_3  d_h^2+L_3 d_1^2)+L_6+M_3U_3+M_2U_1U_2=o(M_1^3)$ and $B_D n^{(2\gamma-3)/(\gamma-1)^2-1}=o(1)$. The probability of an f-rejection during Phase 5 is $o(1)$.
  \end{lemma}

\begin{lemma}\lab{lem:pre-b-rej}
Assume $U_1M_2L_2/M_1^2+(M_2L_2/M_1^2)^2=o(M_1)$. Then the probability of a pre-b-rejection during Phase 5 is $o(1)$.
\end{lemma}

 \begin{lemma} \lab{lem:drejb}Assume $(d_hM_2+d_1L_2+d_1^2d_h^2+U_1^2)M_2L_2/M_1^2+M_2U_1+L_4=o(M_1^2)$.
 The probability of a b-rejection during Phase 5 is $o(1)$.
  \end{lemma}


\subsection{Phases 3 and 4: reduction of light loops and light triple edges}\lab{sec:loop}

The switchings used in these two phases are the same as in the corresponding phases in~\cite{GWreg}, with almost the same analysis. Thus, we briefly describe the switchings, define the switching steps, and specify key parameters used in these two phases. Then we give the lemmas which give bounds on $f(P)$ and $b(P)$, and bounds on the probabilities of f-rejections and b-rejections. The proofs are almost the same as in~\cite{GWreg} and thus will be presented in the Appendix.

When Phase 3 starts, the input is a pairing uniformly distributed in $\a_0$: the set of pairings $P$ such that $G_{[\heavy]}(P)$ is simple. Set $\imax=B_L$ in this phase. Then by~\eqn{imax}, $\imax=O(L_2/M_1)$. \jt{In Phase 3, $\state_i$} is the set of pairings in $\a_0$ containing exactly $i$ loops.

 We use the switching shown in Figure~\ref{f:loop} to reduce the number of loops.


    Define
    \bean
    \UB_{\ell}(i)&=&2i M_1^2;\\
    \LB_{\ell}(i)&=&L_2M_1-2d_hM_1(2i+4B_D+6B_T+id_h/2)-L_2(2i+4B_D+6B_T+6d_1+2U_1).
      \eean 
\begin{lemma}\lab{lem:loopbounds}
For each $P\in \state_i$, 
\[
\UB_{\ell}(i)-O(iM_1(i+B_D+B_T+d_1+U_1))\le f(P)\le \UB_{\ell}(i);\quad \LB_{\ell}(i)\le b(P)\le L_2M_2.
\]
\end{lemma}

Let $P_t$ denote the pairing obtained after step $t$ of Phase 3.
At step $t\ge 1$, choose uniformly at random a switching $S$ applicable on $P_{t-1}$. Assume that $P_{t-1}\in\state_i$ and $P$ is the pairing \jt{obtained} if $S$ is performed. Perform an f-rejection with probability $1-f(P_{t-1})/\UB_{\ell}(i)$ and a b-rejection with probability $1-\LB_{\ell}(i-1)/b(P)$. If no rejection happened, then set $P_t=P$. Repeat until $P_t\in\state_0$. 

 As proved in~\cite{MWgen,GWreg}, by the end of Phase 3, the output is uniformly distributed in \jt{$\a_3=\{P\in\a_0:\ L(P)=0\}$. }


\begin{lemma}\lab{lem:lrej} Assume $d_hM_1(B_D+B_T+B_Ld_h)+L_2(d_1+U_1)=o(M_1^2)$.
The probability of an f-rejection or a b-rejection during Phase 3 is $o(1)$.
\end{lemma}

 Phase 4 starts with a pairing $P_0$ uniformly distributed in \jt{$\a_3$}. We will use the switching shown in Figure~\ref{f:triple} to reduce the triple edges.




In this phase, $\imax$ is set \jt{to} $B_T$ defined in~\eqn{imax} and hence $\a_3=\cup_{\jt{0}\le i\le \imax}\state_i$ where $\state_i$ is the set of pairings in $\a_3$ containing exactly $i$ triple edges. 

Recalling $U_k$ from~\eqn{0U1} and~\eqn{0Uk},
define 
\bea
\UB_{t}(i)&=&12 i M_1^3\non\\
\LB_{t}(i)&=&M_3L_3-3M_3(4B_Dd_h^2+6id_h^2)-3L_3(4B_Dd_1^2+6id_1^2)-L_6\non\\
&&-16M_3U_3-3M_2U_1U_2.\lab{triplebound}
\eea

\begin{lemma}\lab{lem:triplebounds}
For each $P\in \state_i$, 
\[
\UB_{t}(i)-O(iM_1^2(i+B_D+d_1+U_1))\le f(P)\le \UB_{t}(i);\quad \LB_{t}(i)\le b(P)\le M_3L_3.
\]
\end{lemma}
 The algorithm in Phase 4 is similar to that in Phase 3. At step $t$, the algorithm chooses uniformly at random a switching $S$ applicable on $P_{t-1}$. Assume that $P_{t-1}\in\state_i$ and $P$ is the pairing \jt{obtained} if $S$ is performed. Perform an f-rejection with probability $1-f(P_{t-1})/\UB_{t}(i)$ and a b-rejection with probability $1-\LB_{t}(i-1)/b(P)$. If no rejection happened, then set $P_t=P$. Repeat until $P_t\in\state_0$. The output then is uniformly distributed in \jt{$\a_4=\{P\in\a_0:\ L(P)=T(P)=0\}$.}

\begin{lemma}\lab{lem:trej} Assume $(d_1+U_1)L_3M_3=o(M_1^4)$ and $M_3d_h^2(B_D+B_T)+L_3(B_Dd_1^2+B_Td_1^2)+L_6+M_3U_3+M_2U_1U_2=o(M_1^3)$.
The overall probability that either an f-rejection or a b-rejection occurs during Phase 4 is $o(1)$.
\end{lemma}


\subsection{Fixing $\delta$}
~\lab{sec:delta}
We show that there exists $\delta$ such that \jt{the} various conditions on $\gamma$ and $\delta$ together with the assumptions of Lemmas~\ref{lem:Phi0},~\ref{lem:heavy-f-rej}, \ref{lem:heavy-b-rej}, \ref{lem:heavyloopcond}, \ref{lem:A0rej}, \ref{lem:xi}, \ref{lem:typerej}, \ref{lem:drejf}, \ref{lem:pre-b-rej}, \ref{lem:drejb}, \ref{lem:lrej}, and~\ref{lem:trej} \jt{are simultaneously satisfied} when $\gamma>21/10+\sqrt{61}/10$. After removing redundant constraints (see the full details in Appendix), these conditions finally reduce to
\bel{delta}
f_1(\gamma)<\delta<f_2(\gamma), \quad 5/2<\gamma<3
\ee
where
\[
f_1(\gamma)=\max\left\{\frac{4/(\gamma-1)-2}{\gamma-2},\frac{3-\gamma}{\gamma-2},\frac{1}{2\gamma-3}\right\},
\]
\bean
f_2(\gamma)&=&\min\Big\{\frac12,1-\frac{1}{\gamma-1},\frac{2}{7-\gamma}, \frac{2-2/(\gamma-1)-(2\gamma-3)/(\gamma-1)^2}{3-\gamma},\frac{2-3/(\gamma-1)}{4-\gamma} \Big\}
\eean 
It is easy to see that \eqn{delta} is feasible whenever
\[
21/10+\sqrt{61}/10<\gamma<3,
\]
where $21/10+\sqrt{61}/10\approx 2.881024968$.

For $\gamma$ in that range, we may choose any
\[
\frac{1}{2\gamma-3}<\delta<\frac{2-3/(\gamma-1)}{4-\gamma}.
\]

\subsection{Uniformity}

\begin{lemma}\lab{lem:uniformity}
Assume that ${\bf d}$ is a plib sequence with parameter $\gamma>21/10+\sqrt{61}/10\approx 2.881024968$ with even sum of the components. Then, the output of \PL\ is uniform over the set of graphs with degree sequence ${\bf d}$.  Moreover,  the probability of any rejection occurring in Stage 2 (Phases 3--5) is $o(1)$.
\end{lemma}

\proof The uniformity follows by Lemma~\ref{lem:afterheavy} and~\eqn{eq:uniform}. By choosing $\delta$ satisfying~\eqn{deltaRange}, the probability of any rejection occurring in Stage 2 is $o(1)$ by Lemmas~\ref{lem:Phi0},~\ref{lem:heavy-f-rej}, \ref{lem:heavy-b-rej}, \ref{lem:heavyloopcond}, \ref{lem:A0rej}, \ref{lem:xi}, \ref{lem:typerej}, \ref{lem:drejf}, \ref{lem:pre-b-rej}, \ref{lem:drejb}, \ref{lem:lrej}, and~\ref{lem:trej}.\qed

\newpage

\no {\bf \Large Appendix}\bigskip

\no{\bf Bounds on $H_k$ and $L_k$}\ss

By~\eqn{di}, 
  \be\lab{dheavy}
  H_1= \sum_{i=1}^{h} O((n/i)^{1/(\gamma-1)})
   =O(n^{1/(\gamma-1)}h^{1-1/(\gamma-1)}) 
  =O(n^{1-\delta\gamma+2\delta}).
  \ee
 For each $k\ge 2$,~\eqn{Delta} gives 
\bel{Hk}
H_k\le M_k=O(n^{k/(\gamma-1)}).
\ee

 Since  $\gamma>5/2$,~\eqn{dheavy} implies  $H_1=o(n)$ and so
$L_1=M_1-H_1= \Theta(n)$.
By~\eqn{di} and~\eqn{h},
\bel{Lk}
L_k\le\sum_{i=n^{1-\delta(\gamma-1)}}^n((Kn/i)^{1/(\gamma-1)})^k=O(n^{1-\delta(\gamma-1-k)})\quad \mbox{for each fixed $k\ge 2$}.
\ee

Several parts of our analysis will require the following upper bound. 
\begin{lemma}\lab{lem:min}
Assume 
\bel{anotherCondition}
1-\delta>1/(\gamma-1).
\ee
Then
\bel{min}
\sum_{i<j\le h}\min\left\{\frac{(d_id_j)^k}{n},\frac{(d_id_j)^{k+1}}{n^2}\right\}= 
\begin{cases}
O\left(n^{2-\gamma}\log n+n^{-2\delta(\gamma-2)}\right) & k=0\\
O\left(n^{3-\gamma}\log n\right)& k=1\\
O\left(n^{2k/(\gamma-1)-1}+n^{k+2-\gamma}\log n\right)& k\ge 2.
\end{cases}
\ee

\end{lemma}

\proof We wish to obtain an upper bound for
\[
\sum_{i<j\le h}\min\left\{(d_id_j)^k/n,(d_id_j)^{k+1}/n^2\right\},
\]
for fixed integer $k\ge 0$.
This function takes the first term if $d_id_j\ge n$ and the second term otherwise. We use $f\lea g$ to denote $f=O(g)$ for simplicity. By~\eqn{di}, we know that $d_id_j=O(n)$ if $ij> n^{3-\gamma}$. Thus, 
we obtain an upper bound for the minimum function by taking the first term if $ij\le n^{3-\gamma}$ and taking the second term otherwise. Hence, 
\bea
\sum_{i<j\le h}\min\left\{\frac{(d_id_j)^k}{n},\frac{(d_id_j)^{k+1}}{n^2}\right\}&\lea& \sum_{i=1}^h\sum_{j\le \min\{h,n^{3-\gamma}/i\}}\frac{(d_id_j)^k}{n}+  \sum_{i=1}^h\sum_{n^{3-\gamma}/i<j\le h}\frac{(d_id_j)^{k+1}}{n^2}\non\\
&\lea& n^{2k/(\gamma-1)-1}\sum_{i=1}^h\sum_{j\le \min\{h,n^{3-\gamma}/i\}} (ij)^{-k/(\gamma-1)}\non\\
&&+n^{2(k+1)/(\gamma-1)-2}\sum_{i=1}^h\sum_{n^{3-\gamma}/i<j\le h}(ij)^{-(k+1)/(\gamma-1)}.\lab{minf}
\eea

By~\eqn{h} and~\eqn{anotherCondition},
we  have $n^{3-\gamma} <h$. Hence, using~\eqn{h} and~\eqn{minf}, we obtain the desired bound as in Lemma~\ref{lem:min}.\qed\ss

We bound $W_{i,j}$ and $W_i$ in the following lemma, which
has been used to specify $\Phi_0$.

\begin{lemma}\lab{lem:typical}
Let $P$ be a pairing chosen uniformly at random from $\Phi$.  Then   
\bean
\sum_{i<j\le h} \ex\left(m_{i,j}I_{m_{i,j}\ge 2}\left(\frac{W_{i,j}}{d_i}+\frac{W_{j,i}}{d_j}\right)\right)&=& O\left(\frac{M_2^2H_1}{M_1^3}\right),\\
\sum_{i\le h} \ex\left(m_{i,i}\frac{W_{i}}{d_i}\right)&=& O\left(\frac{M_2^2H_1}{M_1^3}\right),\\
\sum_{i<j\le h} \ex\left(m_{i,j}I_{m_{i,j}\ge 2}\right)&=& (1/2+o(1)) \frac{M_2^2}{M_1^2},\\
\sum_{i\le h} \ex\left(m_{i,i}\right)&\le & (1/2+o(1)) \frac{M_2}{M_1},
\eean
where $m_{i,j}$ denotes the the number of pairs between $i$ and $j$ and $m_{i,i}$ denotes the number of loops at $i$ in $P$ and $W_{i,j}=W_{i,j}(P)$, $W_i=W_i(P)$.
\end{lemma}

\proof We will prove the first equation; the other equations follow with a similar argument. 

Note that
\bean
\sum_{1\le i<j\le h} \ex\left(m_{i,j}I_{m_{i,j}\ge 2}\frac{W_{i,j}}{d_i}\right) &\le& \sum_{1\le i<j\le h} \sum_{1\le w\le h} \ex\left(\frac{m_{i,j}m_{i,w} I_{m_{i,j}\ge 2,m_{i,w}\ge 2}}{d_i}\right)\\
&&+\sum_{1\le i<j\le h} \ex\left(\frac{m_{i,j}m_{i,i} I_{m_{i,j}\ge 2,m_{i,i}\ge 1}}{d_i}\right)
\eean
Note also that $\ex(m_{i,j}I_{m_{i,j}\ge 2})\le \ex(m_{i,j}) \sim d_id_j/M_1$. Conditional on any possible value of $m_{i,j}$,  which is $o(M_1)$ always, the expectation of 
$m_{i,w}I_{m_{i,w}\ge 2}$ is at most the conditional expectation of $[m_{i,w}]_2$ which is $O([d_i]_2[d_w]_2/M_1^2)$ and the conditional expectation of $m_{i,i}I_{m_{i,i}\ge 1}$ is at most $O(d_i^2/M_1)$. 
Hence, 
\bea
\sum_{1\le i<j\le h} \ex\left(m_{i,j}I_{m_{i,j}\ge 2}\frac{W_{i,j}}{d_i}\right) &=&\sum_{1\le i<j\le h} \sum_{w\le h} O\left(\frac{d_i^2d_w^2d_j}{M_1^3}\right)+\sum_{1\le i<j\le h}  O\left(\frac{d_i^2d_j}{M_1^2}\right)\non\\ &=&O(H_2H_1M_2/M_1^3+H_2H_1/M_1^2)=O(M_2^2H_1/M_1^3),\lab{Ws}
\eea
where the last equation above holds by the fact $H_2\le M_2$ and by our assumption~\eqn{M2}.

By symmetry we have
\[
\sum_{1\le i<j\le h} \ex\left(m_{i,j}I_{m_{i,j}\ge 2}\frac{W_{j,i}}{d_j}\right)=O(M_2^2H_1/M_1^3),
\]
which together with~\eqn{Ws} yields the first equation in the lemma.\qed\ss

\no {\bf Proof of Lemma~\ref{lem:Phi0}.\ } It follows as a direct corollary of Lemma~\ref{lem:typical}.\qed
\ss

\begin{lemma}\lab{lem:bounds} For every $k\ge 1$, every $m\ge 1$,  and for every $P\in  \C(\MM_k,i,j,m) $, 
\[
\LB^f_{i,j}(\MM_k,m)\le f_{i,j}(P)\le \UB^f_{i,j}(\MM_k,m);
\]
for every $P\in  \C(\MM_k,i,j,0)$,
\[
\LB^b_{i,j}(\MM_k,m)\le b_{i,j}(P,m)\le \UB^b_{i,j}(\MM_k,m).
\]
where $(i,j)$ is the $k$-th element in $\heavyindicespair$.
\end{lemma}

\proof
The upper bound for $f_{i,j}(P)$ is obvious: there are $m!$ ways to ordering the $m$ pairs between $i$ and $j$ and there are at most $M_1^m$ ways to pick another $m$ pairs and label their end points. For the lower bound, note that the total number of pairs incident with a heavy vertex is at most $H_1$ and thus, the number of choices to pick the $k$-th pair ($k\le m$) that has not been picked by the first $k-1$ pairs, and label its end points is at least $M_1-H_1-2(k-1)\ge M_1-H_1-2m$. This proves the lower bound for $f_{i,j}(P)$. Now, consider $b_{i,j}(P',m)$. Recall that $W_{i,j}=W_{i,j}(P)$ is the number of pairs that are in a multiple edge and with one end in $i$ and the other wind in some $w\neq j$ in $P$. To perform an inverse switching on $P$, pick $m$ ordered points in $i$ and their mates and $m$ ordered points in $j$ and  their mates. Switch these $2m$ pairs as indicated in Figure~\ref{f:heavy1}. One forbidden case is that one of these $2m$ pair is contained in a heavy multiple edge. Hence, an upper bound for $b_{i,j}(P,m)$ is $[d_i-W_{i,j}]_m[d_j-W_{j,i}]_m$. The only other forbidden case is that the $k$-th pair $\{2k-1,2m+2k-1\}$  incident with $i$ and the $k$-th pair $\{2k,2m+2k\}$ incident with $j$ are such that both $2m+2k-1$ and $2m+2k$ are contained in heavy vertices for some $1\le k\le m$. The number of ways to choose such two pairs is at most $mh^2[d_i-W_{ij}]_{m-1}[d_j-W_{j,i}]_{m-1}$, as $m$ is the number of choices for $k$;  $[d_i-W_{ij}]_{m-1}[d_j-W_{j,i}]_{m-1}$ is an upper bound for the number of choices for the other $2(m-1)$ pairs and $h^2$ is an upper bound for the number of choices for $\{2k-1,2m+2k\}$ and $\{2k,2m+2k-1\}$; both are contained in a single edge, as there are only $h$ heavy vertices. Note that we do not need to consider the case that these pairs are contained in a multiple edge as this case has already been excluded in the upper bound $[d_i-W_{i,j}]_m[d_j-W_{j,i}]_m$. The lower bound follows.\qed
\ss

Note that Lemma~\ref{lem:heavy-rej} follows from the following two lemma.

\begin{lemma}\lab{lem:heavy-f-rej}
The probability of an $f$-rejection during Phase 1 is $o(1)$, if $\delta>(3-\gamma)/(\gamma-2)$, $\delta>1/(2\gamma-3)$, $\delta<1/2$ and $1-\delta>1/(\gamma-1)$.
\end{lemma}

\begin{lemma}\lab{lem:heavy-b-rej}
The probability of a $b$-rejection during Phase 1 is $o(1)$, if $\delta>(3-\gamma)/(\gamma-2)$, $\delta>1/(2\gamma-3)$, $\delta<1/2$ and $1-\delta>1/(\gamma-1)$.
\end{lemma}

\no {\em Proof for Lemma~\ref{lem:heavy-f-rej}.\ }
Let $(i,j)$ be the k-th element in $\heavyindicespair$ and $P$ be a random pairing in $\C(\MM_{k-1})$. Let $Y_{i,j}$ be the random variable denoting the multiplicity of $\{i,j\}$ in $P$. Note that $Y_{i,j}$ is the $ij$ entry of $\MM_{k-1}$ and thus the $ij$ entry of $\MM_{0}=\MM(P_0)$ by the definition of $\MM_{k-1}$. 

An f-rejection can happen in sub-step (ii) or (v).

For (ii), the probability of an f-rejection is 
\[
\sum_{m\ge 2}\pr(\multiplicity(P,i,j)=m)\left(1-\frac{f_{i,j}(P)}{\UB_{i,j}^f(\MM_k,m)}\right)I_{\multiplicity(P,i,j)=m}.
\]
By~\eqn{boundsmultiplefu} and~\eqn{boundsmultiplefl}, if $\multiplicity(P,i,j)=m$ then $\UB_{i,j}^f(\MM_k,m)-f_{i,j}(P)\le \UB_{i,j}^f(\MM_k,m)-\LB_{i,j}^f(\MM_k,m)=O(m!mM_1^{m-1}n^{1-\delta\gamma+2\delta})$. Hence,
the above probability is at most
\[
\sum_{m\ge 2}\pr(\multiplicity(P,i,j)=m)\cdot O\left(mn^{-\delta\gamma+2\delta}\right)=O(n^{-\delta\gamma+2\delta})\ex (Y_{i,j} I_{Y_{i,j}\ge 2}).
\]
We will bound $\ex (Y_{i,j} I_{Y_{i,j}\ge 2})$ by 
\[
O\left(\min\left\{\frac{d_id_j}{n},\frac{d_i^2d_j^2}{n^2}\right\}\right).
\]
Then the above probability is bounded by
\[
O(n^{-\delta\gamma+2\delta})\min\left\{d_id_j/n,d_i^2d_j^2/n^2\right\}.
\]

For (v), the probability that $P_k$ is not set as $P_k^*$ is
at most one.
The probability of an f-rejection at sub-step (v) is
\[
\pr(Y_{i,j}\ge 2)(1-b_{i,j}(P')/\UB_{i,j}^b(\MM_k,1))=\min\left\{\frac{d_id_j}{n},\frac{d_i^2d_j^2}{n^2}\right\}\cdot O(h^2/d_id_j),
\]
by~\eqn{boundsmultiplebu} and~\eqn{boundsmultiplebl}.

Hence, the overall probability of an f-rejection is at most
\bean
&&\sum_{i<j\le h}O(n^{-\delta\gamma+2\delta}+h^2/d_id_j)\min\left\{d_id_j/n,d_i^2d_j^2/n^2\right\}\\
&&=O(n^{-\delta\gamma+2\delta+3-\gamma}\log n)+O(h^2n^{-2\delta(\gamma-2)})=O(n^{3-\gamma-\delta(\gamma-2)}\log n+n^{2-2\delta(2\gamma-3)}),
\eean
by~\eqn{h} and~\eqn{min} and using $\delta<1/2$.

To ensure that the overall probability of an f-rejection is $o(1)$, we need
$\delta>(3-\gamma)/(\gamma-2)$, $\delta>1/(2\gamma-3)$ besides the conditions $\delta<1/2$ and $1-\delta>1/(\gamma-1)$.\qed

\no {\em Proof for Lemma~\ref{lem:heavy-b-rej}.\ }
A b-rejection can happen in sub-step (ii) or (v) in a switching step. For (ii), the probability that a b-rejection happens is
\[
\sum_{m\ge 2}\pr(\multiplicity(P,i,j)=m)\left(1-\frac{\LB^b_{i,j}(\MM_k,m)}{b_{i,j}(P',m)}\right).
\]
By~\eqn{boundsmultiplebu},~\eqn{boundsmultiplebl} and~\eqn{Wij},
\[
1-\frac{\LB^b_{i,j}(\MM_k,m)}{b_{i,j}(P',m)}=O\left(\frac{mh^2}{d_id_j}\right).
\]
Hence the above probability is at most a constant times
\[
\sum_{m\ge 2}\pr(Y_{i,j}=m)\frac{mh^2}{d_id_j}=\frac{h^2}{d_id_j}\ex(Y_{i,j}I_{Y_{i,j}\ge 2}).
\]
The overall probability of a b-rejection in sub-step (ii) is at most a constant times
\bean
\sum_{i<j\le h}\frac{h^2}{d_id_j}\ex(Y_{i,j}I_{Y_{i,j}\ge 2})&\le& O\left(\sum_{i<j\le h}\frac{h^2}{d_id_j} \min\left\{\frac{d_id_j}{n},\frac{d_i^2d_j^2}{n^2}\right\}\right)\\
&\le& O\left(h^2 n^{-2\delta(\gamma-2)} \right)\quad \mbox{by~\eqn{min}}\\
&\le& O\left(n^{2-2\delta(2\gamma-3)}\right) \quad\mbox{by~\eqn{h}}.
\eean

For (v), the probability that a b-rejection happens is at most
\[
\pr(m(P,i,j)\ge 2) \left(1-\frac{\LB_{i,j}^f(\MM_k,1)}{f_{i,j}(P'')}\right).
\]
By~\eqn{boundsmultiplefu} and~\eqn{boundsmultiplefl},
\[
1-\frac{\LB_{i,j}^f(\MM_k,1)}{f_{i,j}(P'')}=O\left(\frac{H_1}{M_1}\right)=O(n^{-\delta\gamma+2\delta}),
\]
by~\eqn{dheavy}.
Hence, the above probability is at most a constant times
\[
\pr(Y_{i,j}\ge 2)\cdot n^{-\delta\gamma+2\delta}\le O\left( n^{-\delta\gamma+2\delta}\min\left\{\frac{d_id_j}{n},\frac{d_i^2d_j^2}{n^2}\right\}\right).
\]
Now the overall probability of a b-rejection for (v) is at most a constant multiple of
\[
n^{-\delta\gamma+2\delta}\sum_{i<j\le h}\min\left\{\frac{d_id_j}{n},\frac{d_i^2d_j^2}{n^2}\right\}\le O\left( n^{-\delta\gamma+2\delta+3-\gamma}\log n\right),
\]
which is $o(1)$ for $\gamma$ and $\delta$ that satisfy the conditions of the lemma.\qed

\no{\em Proof for Lemma~\ref{lem:heavyloopcond}.\ } The argument is analogous to, but simpler, than those in Lemmas~\ref{lem:heavy-f-rej} and~\ref{lem:heavy-b-rej}, so we only give a simple sketch. The probability of an $f$-rejection in Phase 2 is
\[
\sum_{i\le h}\ex(Y_{i,i})\cdot O(n^{-\delta\gamma+2\delta})=\sum_{i\le h}O([d_i]_2/n)\cdot O(n^{-\delta\gamma+2\delta}),
\]
where $Y_{i,i}$ is the number of loops at $i$.  Noting that $\sum_{i\le h} [d_i]_2\le M_2$, the  above probability is $O(n^{2/(\gamma-1)-\delta\gamma+2\delta})$, which is $o(1)$ if $2/(\gamma-1)<1+\delta(\gamma-2)$. 

The probability of a $b$-rejection in Phase 2 is at most a constant multiple of 
\[
\sum_{i\le h} \ex(Y_{i,i}) \frac{h^2}{d_i^2} =\sum_{i\le h} O(d_i^2/n) \frac{h^2}{d_i^2}= O\left(\frac{h^3}{n}\right),
\]
which is $o(1)$ by~\eqn{h} if $\delta(\gamma-1)>2/3$.\qed

\no {\em Proof for Lemma~\ref{lem:afterheavy}.\
}
Recall that $G(P)$ denotes the multigraph obtained from $P$. For any vertex set $S\subseteq [n]$, let $G_{[S]}(P)$ denote the subgraph of $G(P)$ induced by $S$.
We first confirm that the output of Phase 1, if no rejection occurs, is uniformly distributed in $\Phi_1$, defined by 
\bel{Phi1}
\Phi_1=\{P\in\heavyaccept:\ G_{[\heavy]}(P) \ \mbox{contains no multiple edges $ij$ with $i\neq j$} \}.
\ee

Given a pairing $P\in  \C(\MM_k,i,j,0)$, let $\state^+_{i,j,m}(P)$ denote the  set  of inverse heavy $m$-way switchings that convert  $P$ to some pairing $P'$ in $ \C(\MM_k,i,j,m) $. Then  $|\state^+_{i,j,m}(P)|=b_{i,j}(P,m)$.
Conversely, let $\state^-_{i,j}(P)$ denote the set of heavy  switchings that convert $P \in  \C(\MM_k,i,j,m)$ to some $P'$ in $ \C(\MM_k,i,j,m) $; then $|\state^-_{i,j}(P)|=f_{i,j}(P)$.

For each $k\ge 0$, let $(i,j)$ be the $k$-th element in $\heavyindicespair$. 
 We prove by induction on $k$ that for every $k\ge 0$, the pairing obtained after the $k$-th step of Phase 1 is uniformly distributed in $\C(\MM_k)$.

The base case $k=0$ is trivially true as $P_0$ is uniformly distributed in $\C(\MM_0)$, conditional on the value of $\MM_0$. Assume $k\ge 1$ and that the statement is true for $k-1$.  Note that $\C(\MM_{k-1})$ denote the set of pairings $P$ such that the multiplicity of $(i,j)$ in $P$ is 0 or 1 for all $(i',j')\preceq (i,j)$.

  Consider the $k$-th step in Phase 1.
Let $P_k^*$ be as in (iii). Then, for any pairing $P\in  \C(\MM_k,i,j,0) $,
\bea
\pr(P_k^*=P\mid \multiplicity(P_{k-1},i,j)\ge 2)&=&\sum_{m\ge 2}\sum_{P'\in \state^+_{i,j,m}(P)}\frac{1}{\UB^f_{i,j}(\MM_k,m)} \frac{\LB^b_{i,j}(\MM_k,m)}{b_{i,j}(P,m)}\non\\
&=&\sum_{m\ge 2}\frac{\LB^b_{i,j}(\MM_k,m)}{\UB^f_{i,j}(\MM_k,m)}.\lab{Pstar}
\eea
Note that the final expression above does not depend on $P$.

Now, for any pairing $P\in \C(\MM_k,i,j,0) $:
\bea
\pr(P_k=P)&=&\pr(P_{k-1}=P)+\frac{\pr(P_k^*=P\mid \multiplicity(P_{k-1},i,j)\ge 2)}{1+\UB^b_{i,j}(\MM_k,1)/\LB^f_{i,j}(\MM_k,1)}\non\\
&=&\pr(P_{k-1}=P)+\frac{1}{1+\UB^b_{i,j}(\MM_k,1)/\LB^f_{i,j}(\MM_k,1)}\sum_{m\ge 2}\frac{\LB^b_{i,j}(\MM_k,m)}{\UB^f_{i,j}(\MM_k,m)};\lab{P0}
\eea
for any pairing $P\in \C(\MM_k,i,j,1) $:
\bea
&&\pr(P_k=P)=\pr(P_{k-1}=P)\non\\
&&+\sum_{P'\in\state^-(i,j)(P)}\pr(P_k^*=P'\mid \multiplicity(P_{k-1},i,j)\ge 2)\frac{\UB^b_{i,j}(\MM_k,1)/\LB^f_{i,j}(\MM_k,1)}{1+\UB^b_{i,j}(\MM_k,1)/\LB^f_{i,j}(\MM_k,1)}\frac{b_{i,j}(P')}{\UB^b_{i,j}(\MM_k,1)}\frac{\LB^{f}_{i,j}(\MM_k,1)}{f_{i,j}(P)}\non\\
&&=\pr(P_{k-1}=P)+\sum_{P'\in\state^-(i,j)(P)}\sum_{m\ge 2}\frac{\LB^b_{i,j}(\MM_k,m)}{\UB^f_{i,j}(\MM_k,m)}\frac{\UB^b_{i,j}(\MM_k,1)/\LB^f_{i,j}(\MM_k,1)}{1+\UB^b_{i,j}(\MM_k,1)/\LB^f_{i,j}(\MM_k,1)}\frac{1}{\UB^b_{i,j}(\MM_k,1)}\frac{\LB^{f}_{i,j}(\MM_k,1)}{f_{i,j}(P)}\non\\
&&=\pr(P_{k-1}=P)+\frac{1}{1+\UB^b_{i,j}(\MM_k,1)/\LB^f_{i,j}(\MM_k,1)}\sum_{m\ge 2}\frac{\LB^b_{i,j}(\MM_k,m)}{\UB^f_{i,j}(\MM_k,m)},\lab{P1}
\eea
since $|\state^-(i,j)(P)|=f_{i,j}(P)$.

By~\eqn{P0} and~\eqn{P1} and noting that $\C(\MM_k)= \C(\MM_k,i,j,0) \cup  \C(\MM_k,i,j,1) $, for any $P\in\C(\MM_k)$,
\bean
\pr(P_k=P)&=&\pr(P_{k-1}=P)+\frac{1}{1+\LB_{i,j}(\MM_k,1)/\UB_{i,j}(\MM_k,0)}\sum_{m\ge 2}\frac{\LB^b_{i,j}(\MM_k,m)}{\UB^f_{i,j}(\MM_k,m)}\\
&=&\frac{1}{|\C(\MM_{k-1})|}+\frac{1}{1+\LB_{i,j}(\MM_k,1)/\UB_{i,j}(\MM_k,0)}\sum_{m\ge 2}\frac{\LB^b_{i,j}(\MM_k,m)}{\UB^f_{i,j}(\MM_k,m)},
\eean
which is independent of $P$. Hence, $P_k$ is uniformly distributed in $\C(\MM_k)$. 

With an analogous but simpler analysis as in Section~\ref{sec:heavy-uniform}, we can show that the output of Phase 2, if no rejection happens, is uniformly distributed on $\Phi_2$.
This completes the proof. \qed\ss

\no{\em Proof of Lemma~\ref{lem:A0rej}. }
\remove{
Define 
\bel{k0}
k_0=\left\lfloor \frac{1+\delta-\delta\gamma}{1-\delta-1/(\gamma-1)}\right\rfloor.
\ee

We need another two conditions here
\[
\jt{1+\delta-\delta\gamma>0,\ 1-\delta-1/(\gamma-1).} 
\]

\jcom{The stopping point is again very close to the mysterious position: for $\gamma\approx 2.708$, we get $k_0=478$.}

}
Given a pairing $P$, let $P_{[\heavy]}$ be the set of pairs in $P$ whose end vertices are both in $\heavy$. Conditional on $(P_0)_{[\heavy]}$ be a given simple pairing $P'$, $P_0$ can be generated by pairing up all remaining points in $\heavy$ to points in $\light$ and then take a uniform pairing over all remaining points in $\light$. Let $\Phi(P')$ be this conditional probability space. It suffice to prove that for every $P'$, with probability at least $1/2$, a random pairing in $\Phi(P')$ contains at most $B_L$ loops, $B_D$ double edges, $m_{k,0}$ multiple edges with multiplicity $k$, for $3\le k\le k_0$ and no other types of multiple edges.

Let $d'_i=d_i- d_i(P')$ where $d_i(P')$ is the degree of $i$ in $P'$. In other words, $d'_i$ is the number of remaining points in vertex $i$ after removing all points used in $P'$. Let $H'_1=\sum_{i\in\heavy} d'_i$. Clearly, $d'_i\le d_i$ and so 
$H'_1\le H_1=O(n^{1-\delta\gamma+2\delta})=o(n)$
by~\eqn{dheavy} and by our assumption that $\gamma>5/2$.

Recall that the generation of $\Phi(P')$ has two stages: in the first stage, each remaining point in $\heavy$ is paired randomly to a point in $\light$; in the second stage, a uniform pairing is taken over all remaining $L'_1:=L_1-H'_1$ points in $\light$.

Let $X$ denote the number of double loops in $\Phi(P')$: two loops at the same vertex. Since $P'$ is simple, the only possible way a double loop can be created is that for some vertex $i\in \light$, there are four points $p_j$, $1\le j\le 4$, in $i$ such that $p_1$ is paired to $p_2$ and $p_3$ is paired to $p_4$. The probability for this to happen is
\[
\frac{1}{(L'_1-1)(L'_1-3)}\sim L_1^{-2}=O(n^{-2}),
\]
as $L_1=M_1-o(n)=\Theta(n)$ and $H'_1\le H_1=o(n)$.
The total number of ways to choose $\{p_j:\ 1\le j\le 4\}$ is at most
  \[
  \sum_{i\in\light} 3\binom{d_i}{4}.
  \]
  Hence,
  \[
  \ex X=O\left(\frac{L_4}{L_1^2}\right)=O(L_4/M_1^2)=o(1).
  \]
by~\eqn{Lk} and the assumption that $\delta(5-\gamma)<1$.

Let $Y$ denote the number of multiple edges with multiplicity at least $4$. Such multiple edges can be created either in the first stage or in the second stage; one of the end vertices must be light.
\[
\ex Y=O(M_{4}L_{4}/L_1^{4})=O(n^{4/(\gamma-1)+1-\delta(\gamma-5)-4})=o(1).
\]
by~\eqn{Hk} and~\eqn{Lk}. 

Let $L=L(P)$ denote the number of loops in $P$ for $P\in\Phi(P')$. Then,
\[
\ex L= \sum_{i\in\light}\frac{\binom{d_i}{2}}{L'_1-1} \sim \frac{L_2}{2M_1}.
\]
Hence, with probability at most $1/8+o(1)$ will $L$ exceeds $B_L$.

Let $Y_k$ denote the number of multiple edges with multiplicity $k$. 
It only remains to verify that with probability at least $1/4+o(1)$, $Y_2\le B_D$ and $Y_3\le B_T$.
\bean
\ex Y_k&\le& (1+o(1))\left(\frac{\sum_{i\in[n]}\sum_{j\in\light}\binom{d_i}{k} \binom{d_j}{k} k!}{L_1^k}\right)\\
&\le& (1+o(1))\frac{L_kM_k}{k! M_1^k}.
\eean
Hence, with probability $1=o(1)$, there are no multiple edges with multiplicity greater than $3$ by our assumptions in the lemma. Moreover, with probability at most $1/8+o(1)$, the number of double edges is more than $B_D$, and with probability at most $1/12+o(1)$, the number of triple edges is more than $B_T$. By taking the union bound, the probability that $P$ does not satisfy conditions in (b) is less than $1/2+o(1)$. \qed
\ss

\no {\em Proof of Lemma~\ref{lem:lowerbounds}.} We first prove part (b). It is easy to see that, for any fixed $\pairstars\in \Z_{t(\tau)}$, the only condition for the choices of the additional pairs is that (a) each pair is contained in a single edge; (b) for each pair, their end vertices are not allowed to be adjacent to two certain vertices. There are at most $4i$ choices to violate (a) and at most $2U_1$ choices to violate (b) since the number of 2-paths starting from a given vertex is bounded by $U_1$ by~\eqn{U1}. Therefore, the number of choices for each pair is at least $M_1-4i-2U_1$. The lower bounds in (b) follows by the illustrations of the inverse switchings of each type shown in Figures~\ref{f:typeIII}--\ref{f:VII}.

For (a), it is sufficient to bound $Z^*(P)$ uniformly by $8i(d_h M_2+d_1 L_2)+(2id_1^2d_h^2+4iU_1^2+8M_2 U_1+L_4)$ by~\eqn{Zstar}. Recall that $Z^*(P)$ is the size of $\Z^*(P)=\calf(P)\setminus \Z(P)$. Elements in $\Z^*(P)$ must fall into at least one of the following cases.
\begin{enumerate}
\item[(a1)] one of the two stars contains a pair that is contained in a double edge;
\item[(a2)] $u_j$ and $v_j$ are adjacent with a double edge, for some $j=1,2,3$;
\item[(a3)] the two stars share a vertex.
 \end{enumerate} 

For (a1), there are $4i$ choices to choose a point that is contained in a double edge. Label this point 1 or 3 or 2 or 4. If it is labelled 1 (or 3), there are at most $d_h$ ways to choose the point 3 (or 1) and there are at most $M_2$ ways to choose the other 2-star. If it is labelled 2 (or 4), there are at most $d_1$ ways to choose the point 4 (or 2) and then there are at most $L_2$ ways to choose the other light 2-star. Hence, the number of choices for (a1) is at most $4i \cdot 2\cdot d_h\cdot M_2+4i\cdot 2\cdot d_1\cdot L_2=8i(d_h M_2+d_1 L_2)$.

For (a2), we discuss two cases: $j=1$ and $j=2,3$. If $j=1$, there are at most $2i$ ways to choose vertices $u_1$ and $v_1$. There are at most $d_h^2$ ways to choose points and label them by 1 and 3 in $u_1$ and at most $d_1^2$ ways to choose points and label them by 2 and 3  in $v_1$. So the number of choices in this case is at most $2id_1^2d_h^2$. If $j=2$, then there are $2i$ ways to choose vertices $u_2$ and $v_2$; then at most $U_1^2$ ways to fix a 2-path starting at $u_2$ and a 2-path starting at $v_2$. Therefore, the number of choices for $j=2$ is at most $2i U_1^2$. Similarly, this holds for $j=3$. Hence, the total number of choices for (a2) is at most $2id_1^2d_h^2+4iU_1^2$.

  For (a3), we discuss three cases: (i) $v_1\in \{u_2,u_3\}$, the number of such choices is at most $2 M_2 U_1$; (ii) $v_1=u_1$, then $u_1$ must be light and so the number of such choices is at most $L_4$; (iii) $v_j\in \{u_1,u_2,u_3\}$ for some $j=2,3$, then the number of such choices is at most $6 L_2 U_1$. Hence, the total number of choices for (a3) is at most $2 M_2 U_1+L_4+6L_2 U_1\le 8M_2 U_1+L_4$ as $L_2\le M_2$.
  
     This implies the lower bound  for $b(P)$ as desired. \qed  \ss

Before bounding the probabilities of various types of f- and b-rejections, we bound the number of certain structures in $P$, uniformly for all $P\in\a_0$. These bounds turn out to be $U_k$ defined in~\eqn{0U1} and~\eqn{0Uk}.\ss

\no {\bf 2-paths at $v$}\ss

For an integer $i\ge 1$, an $i$-path in a pairing $P$ is a sequence of $i$ pairs: $\{p_1,p_2\},\ldots,\{p_{2i-1},p_{2i}\}$ such that $p_{2j}$ and $p_{2j+1}$ are in the same vertex for each $1\le j\le i-1$. 

Let $A_i(P,v)$ denote the number of \jt{$i$-paths} starting at $v$ in $P$ and let
\[
A_i({\bf d})=\max_{P\in \Phi}\max_{v\in[n]} A_i(P,v).
\]
We will bound $A_2=A_2({\bf d})$.

For any $u\in[n]$, the degree of $u$ is at most $d_1=\Delta$. The number of 2-paths from $u$ is at most the sum of degrees of vertices incident with $u$, 
which is at most 
\[
\sum_{i=1}^{\Delta}(Kn/i)^{1/(\gamma-1)}=(Kn)^{1/(\gamma-1)}\sum_{i=1}^{d_1} i^{-1/(\gamma-1)}\le (Kn)^{1/(\gamma-1)} d_1^{(\gamma-2)/(\gamma-1)}.
\]
Using $d_1\le (Kn)^{1/(\gamma-1)})$ by~\eqn{Delta}, we have
\be
A_2\le  (Kn)^{1/(\gamma-1)}(Kn)^{(\gamma-2)/(\gamma-1)^2}= (Kn)^{(2\gamma-3)/(\gamma-1)^2}=U_1. \lab{U1}
\ee

\remove{
A 2-bloom at vertex $v$ consists of four pairs $\{p_{2i-1},p_{2i}\}$, $1\le i\le 4$ such that $p_1\in v$; $p_{2}$ and $p_3$ are in the same vertex and $p_4$, $p_5$ and $p_7$ are in the same vertex that is light. See Figure~\ref{f:bloom} as an example.
\begin{figure}[htb]

 \hbox{\centerline{\includegraphics[width=7cm]{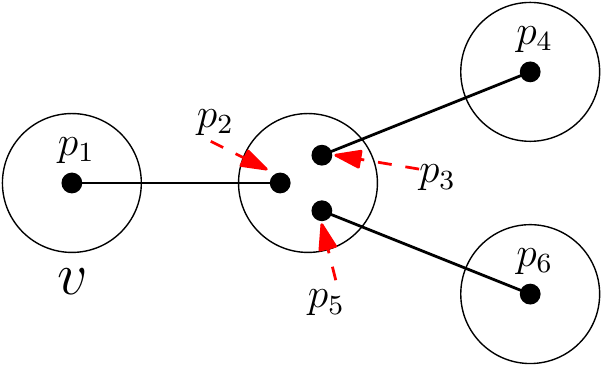}}}
\caption{2-bloom}
\lab{f:bloom}

\end{figure}

Let $B(P,v)$ denote the number of 2-blooms at $v$ in $P$ and define
\[
B({\bf d})=\max_{P\in\Phi}\max_{v\in [n]} B(P,v).
\]
Next we bound $B=B({\bf d})$.

We have estimated that the total degree of the first $\Delta$ vertices is at most $\phi$ for some $\phi=O(n^{(2\gamma-3)/(\gamma-1)^2})$. Hence, given $v$, the number of ways to choose the point $p_4$ is at most $\phi$. Note that the vertex containing $p_4$ is light and thus has degree at most $d_{h+1}$. Hence, the number of 2-blooms at $v$ is at most \jcom{not correct!!!}
   \bean
   \sum_{j=h+1}^{h+\phi}d_j^2&\lea& \sum_{j=h+1}^{h+\phi}(n/j)^{2/(\gamma-1)}\\
   &\lea& n^{2/(\gamma-1)} h^{(\gamma-3)/(\gamma-1)}=n^{a}
   \eean
   where 
   $\phi=n^{(2\gamma-3)/(\gamma-1)^2}$ and
   \[
   a=\frac{2}{\gamma-1}+(1-\delta(\gamma-1))\frac{\gamma-3}{\gamma-1}=1+\delta(3-\gamma).
   \]
}

\no {\bf Light k-blooms at $v$}\ss

A light $k$-bloom at vertex $v$ consists of $k+1$ pairs $\{p_{2i-1},p_{2i}\}$, $1\le i\le k+1$ such that $p_1\in v$; $p_{2}$, $p_{2i-1}$ for $2\le i\le k+1$ are in the same vertex that is light.
See Figure~\ref{f:bloom} for an example where $k=2$. Note that a light $k$-bloom at $v$ is also a light $(k+1)$-star at $u$ where $u$ is the vertex containing points $p_2$ and $p_{2i-1}$, $2\le i\le k+1$.
\begin{figure}[htb]

 \hbox{\centerline{\includegraphics[width=5cm]{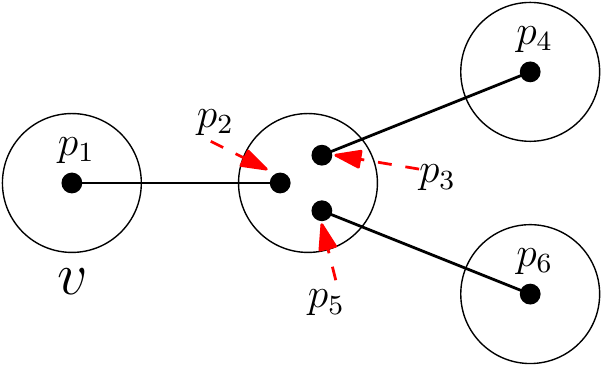}}}
\caption{2-bloom}
\lab{f:bloom}

\end{figure}
Let $B_k(P,v)$ denote the number of light $k$-blooms at $v$ in $P$; and let $B_k({\bf d})=\max_{P\in\Phi}\max_{v\in[n]}B_k(P,v)$. We will bound $B_k=B_k({\bf d})$ for $k\ge 2$. Similar to the calculations for $A_2$,
\bean
B_k&\le& \sum_{i=h+1}^{h+\Delta} d_i^k \le \sum_{i=h+1}^{h+\Delta} (Kn/i)^{k/(\gamma-1)}\le (Kn)^{k/(\gamma-1)} \sum_{i\ge h+1} i^{-k/(\gamma-1)}\\
&\le& \frac{\gamma}{k+1-\gamma}(Kn)^{k/(\gamma-1)} h^{(\gamma-k-1)/(\gamma-1)}.
\eean
By~\eqn{h}, it follows that
\bel{Uk}
B_k\le \frac{\gamma K^{k/(\gamma-1)}}{k+1-\gamma} n^{1-\delta(\gamma-k-1)}=U_{k}\quad \mbox{for}\ k\ge 2.
\ee

\no {\em Proof of Lemma~\ref{lem:drejf}.}
For a given pairing $P$, we first bound the probability that a random switching performed on $P$ is f-rejected. Assume $P\in\state_i$; the probability that a type $\tau$ switching is performed on $P$ and is f-rejected, is
\[
\rho_{\tau}(i)\left(1-\frac{f_{\tau}(P)}{\UB_{\tau}(i)}\right).
\]
Summing over all $\tau\in\Lambda$ we obtain that assuming $P$ arises in the algorithm, the probability that a random switching performed on $P$ is $f$-rejected is
\be\lab{fprob1}
\sum_{\tau}\rho_{\tau}(i)\left(1-\frac{f_{\tau}(P)}{\UB_I(i)}\right).
\ee
Now we bound the probability that an f-rejection happens during Phase 5. Recall that $P_t$ is the pairing obtained after step $t$ of Phase 3. Then, by~\eqn{fprob1}, the probability of an f-rejection during Phase 3 is
\[
\sum_{t\ge 0} \sum_{1\le i\le \imax} \sum_{P\in \state_i} \pr(P_t=P) \left(\sum_{\tau}\rho_{\tau}(i)\left(1-\frac{f_{\tau}(P)}{\UB_{\tau}(i)}\right)\right). 
\]
Note that $\sum_{t\ge 0}\pr(P_t=P)$ is the expected number of times that $P$ is reached, which is $\sigma(i)$ for every $P\in\state_i$. Hence, the above probability is
\be\lab{fprob}
\sum_{1\le i\le \imax}\sum_{P\in\state_i}  \sigma(i) \left(\sum_{\tau}\rho_{\tau}(i)\left(1-\frac{f_{\tau}(P)}{\UB_{\tau}(i)}\right)\right). 
\ee

Note that for each $\tau$,
   \bean
\sum_{1\le i\le \imax} \sum_{P\in\state_i}\sigma(i)\rho_{\tau}(i) \left(1-\frac{f_{\tau}(P)}{\UB_{\tau}(i)}\right)=\sum_{1\le i\le \imax} \sigma(i)\rho_{\tau}(i)|\state_i|\cdot \ex\left(1-\frac{f_{\tau}(P)}{\UB_{\tau}(i)}\right),
 \eean
where the above expectation is taken on a random pairing $P\in\state_i$. It will be sufficient to show that the product of this expectation and $\rho_{\tau}(i)$ is $o(1/\imax)$ for every $1\le i\le\imax$ and for every $\tau$, since then
\bean
\sum_{1\le i\le \imax}\sum_{P\in\state_i}  \sigma(i) \left(\sum_{\tau}\rho_{\tau}(i)\left(1-\frac{f_{\tau}(P)}{\UB_{\tau}(i)}\right)\right)&=&\sum_{\tau}\sum_{1\le i\le \imax} \sigma(i)\rho_{\tau}(i)|\state_i|\cdot \ex\left(1-\frac{f_{\tau}(P)}{\UB_{\tau}(i)}\right)\\
&=&o(1/\imax) \sum_{1\le i\le \imax}\sum_{P\in\state_i}  \sigma(i) |S_i| 
\eean
and the lemma follows since $\sum_{1\le i\le \imax}\sum_{P\in\state_i}  \sigma(i) |S_i| $ is the expected number of steps that the phase takes, which is $O(\imax)$ by Lemma~\ref{lem:steps}.

For $\tau\in \Lambda\setminus\{I,III\}$ we have $\rho_{\tau}\le \xi$ and so 
\be
\rho_{\tau}(i)\ex\left(1-\frac{f_{\tau}(P)}{\UB_{\tau}(i)}\right)=O(\rho_{\tau}(i))=o(1/\imax)\lab{err-cond}
\ee
by our assumption $B_D\xi=o(1)$.  

For $\tau=III$, by Lemma~\ref{lem:solution}, 
\[
\rho_{III}(i)=O\left(\frac{M_3L_3}{M_2L_2 M_1}\right).
\]
The basic operation of type III switching is as in Figure~\ref{f:typeIII}, repeated \jt{in Figure~\ref{f:typeIIIagain}} for convenience. 

\begin{figure}[htb]

 \hbox{\centerline{\includegraphics[width=10cm]{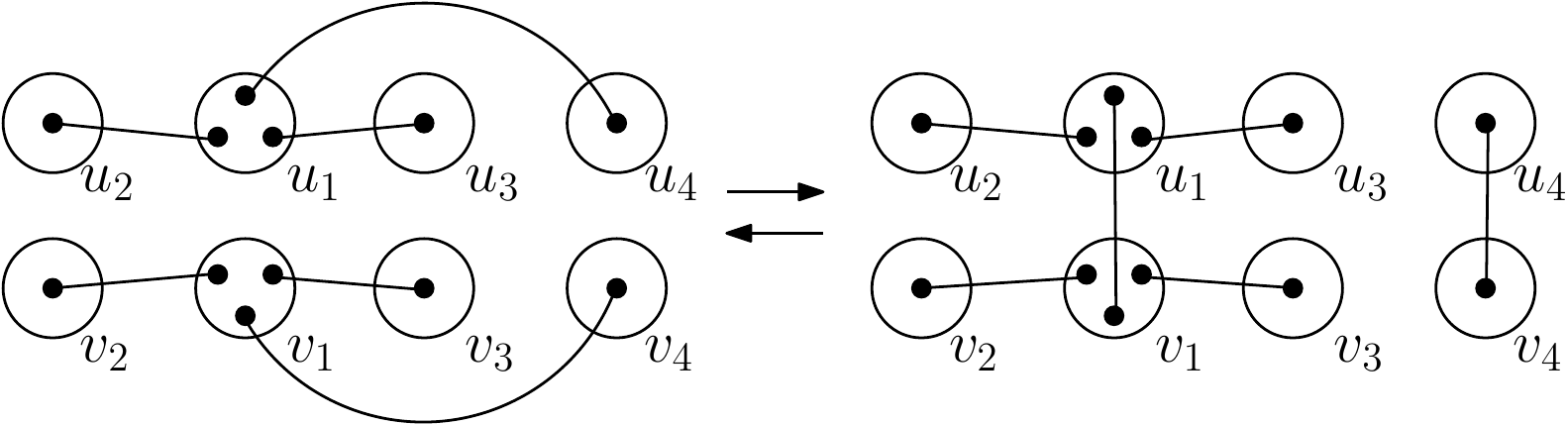}}}
 \caption{Type III again}
 \lab{f:typeIIIagain}

\end{figure}

To perform a type III switching on $P$, we randomly and independently choose a 3-star and a light 3-star. Label the points in the stars and the related vertices as in the figure (here $u_1$ is light). Perform an f-rejection if a type III switching cannot be performed with the given choice.  Now $\ex(1-f_{\tau}(P)/\UB_{\tau}(i))$ is the probability that an f-rejection is performed. 

An f-rejection happens if
 \begin{enumerate}
\item[(a)] one of the 3-stars contains a double edge;
\item[(b)] the two 3-stars share a vertex;
\item[(c)] there is an edge between $u_{j}$ and $v_{j}$ for some $j=1,2,3,4$.
\end{enumerate} 
  The number of choices for (a) is at most $3M_3\cdot 4\imax d_h^2+3L_3\cdot 4\imax d_1^2$.
  
  For (b), there are four cases: (b1), $u_1=v_1$. Since $u_1$ is light, the number of such choices is at most $L_6$; (b2), $v_1\in\{u_2,u_3,u_4\}$. Recall the light $2$-bloom defined  above Figure~\ref{f:bloom}. The number of choices for (b2) is at most $3M_3B_2\le 3M_3U_2$; (b3), $v_j=u_1$ for some $j=2,3,4$. The number of such choices is at most $3M_2B_3\le 3M_2U_3$; (b4), $v_j\in\{u_2,u_3,u_4\}$ for some $j=2,3,4$, the number of such choices is at most $9M_3B_2\le 9M_3U_2$. Hence, the total number of choices for (b) is at most $L_6+12M_3U_2+3M_2U_3$.
  
  For (c), there are two cases: (c1), $u_1$ is adjacent to $v_1$. The number of such choices is at most $M_3B_3\le M_3U_3$; (c2), $u_j$ is adjacent to $v_j$ for some $j=2,3,4$. The number of such choices is at most $3M_2A_2B_2\le 3M_2U_1U_2$. So the total number of choices for (c) is bounded by $M_3U_3+3M_2U_1U_2$.
  
Hence, the probability of an f-rejection in each step is bounded by
\[
\frac{M_3 \imax d_h^2+L_3\imax d_1^2+L_6+M_3U_3+M_2U_1U_2}{M_3L_3},
\]
  by noting that $M_3U_2=O(M_3U_3)$ and $M_2U_3=O(M_3U_3)$.
      Since $\rho_{III}(i)=O(M_3L_3/M_2L_2M_1)$ and $\imax=O(M_2L_2/M_1^2)$,~\eqn{err-cond} is satisfied for $\tau=III$ as long as 
       \[
       (M_3  d_h^2+L_3 d_1^2)\frac{M_2L_2}{M_1^2}+L_6+M_3U_3+M_2U_1U_2=o(M_1^3),
       \]
         which is guaranteed by the hypotheses of the lemma.
         
   Finally, for $\tau=I$, we use the bound $\rho_{I}(i)\le 1$. An f-rejection occurs if $u_2$ and $v_2$ is chosen such that they are not distinct from $u_1$ and $v_1$, or there is an existing edge between $u_1$ and $u_2$ or an existing edge between $v_1$ and $v_2$; same for $u_3$ and $v_3$. The total number of such invalid choices for $u_2$ and $v_2$ is $O(\Delta+A_2)$, where $A_2$  is bounded by $O(n^{(2\gamma-3)/(\gamma-1)^2})$ in~\eqn{U1}. Hence, the probability of an f-rejection in each step is bounded by $O(n^{(2\gamma-3)/(\gamma-1)^2}-1)$, which satisfies~\eqn{err-cond} by the assumption of the lemma.\qed
\ss 


\no {\em Proof of Lemma~\ref{lem:pre-b-rej}.}
  Note that a pre-b-rejection only happens when type $\tau\neq I$ switchings are performed. 
At a given step $t$, we estimate the probability that a given switching $S=(P',P,\pairstars)$ is performed and is pre-b-rejected at step $t$. Here $S$ converts $P'$ \jt{to} $P$ and $\pairstars$ is the ordered pair of 2-stars created by $S$. Note that there can be several switchings converting $P'$ to $P$, corresponding to different $\pairstars$, due to a different labelling of points in the involved pairs that are switched.  Let $\tau=\tau(P',P)$ denote the type of $S$. Assume that $P\in\state_i$ and so $P'\in\state_i$. 

The probability that $S$ is performed, conditional on $P_{t-1}=P'$, is
\[
\frac{\rho_{\tau}(i)}{\UB_{\tau}(i)}.
\]
Conditional on that, the probability that $S$ is pre-b-rejected is
$1-\LBS_{\tau}(i)/\b(P,\pairstars)$.

Let $\state^+_{\tau}(P)$ denote the set of pairings $P'$ that can be switched to $P$ via a type $\tau$ switching. The probability that a pre-b-rejection occurs during Phase 5 is at most
\bean
&&\sum_{t\ge 1}\sum_{\tau\neq I}\sum_{1\le i\le \imax} \sum_{P\in \state_i}\sum_{P'\in\state_i\cap \state^+_{\tau}(P)}\sum_{\pairstars} \pr(P_{t-1}=P') \frac{\rho_{\tau}(i)}{\UB_{\tau}(i)} \left(1-\frac{\LBS_{\tau}(i)}{\b(P,\pairstars)}\right)\\
 &&=\sum_{1\le i\le \imax}\sum_{P\in \state_i}\sum_{\tau\neq I}   \sum_{P'\in\state_i\cap \state^+_{\tau}(P)}\sum_{\pairstars}  \sigma_i \frac{\rho_{\tau}(i)}{\UB_{\tau}(i)} \left(1-\frac{\LBS_{\tau}(i)}{\b(P,\pairstars)}\right),
 \eean
as $\sum_{t\ge 1}\pr(P_{t-1}=P')=\sigma(P')$ which is $\sigma_j$ for every $P'\in\state_j$. The $\pairstars$ in the summation are determined by $(P',P)$, and there can be several of them, due to different labelling of points in the pairs that are switched. 
For each $\tau\neq I$, define $U_{\tau}$ to be the same as $\LBS_{\tau}(i)$ but without terms $-2U_1-4i$. Note that $U_{\tau}$ is a trivial upper bound for $\b(P,\pairstars)$.
Now,
by Lemma~\ref{lem:lowerbounds}, for every $\tau\neq I$,
 \[
 1-\frac{\LBS_{\tau}(i)}{\b(P,\pairstars)} = O\left(\frac{U_1+i}{M_1}\right).
 \]
  Noting that given $\tau$ and $\pairstars\in \Z_{t(\tau)}(P)$, $\sum_{P'\in\state_i\cap \state^+_{\tau}(P)}1 =\b(P,\pairstars)\le U_{\tau}$, and noting also that $|\Z_{t(\tau)}(P)|\le M_2L_2$ for any $\tau$, the above probability of a pre-b-rejection becomes
\bean
O\left(\sum_{1\le i\le \imax}\sum_{P\in \state_i} \sum_{\tau\neq I}   M_2L_2 U_{\tau} \sigma_i \frac{\rho_{\tau}(i)}{\UB_{\tau}(i)} \frac{U_1+i}{M_1} \right)=O\left( M_2L_2\frac{U_1+\imax}{M_1}\sum_{\tau\neq I} U_{\tau}\sum_{1\le i\le \imax}\sum_{P\in \state_i}   \sigma_i \frac{\rho_{\tau}(i)}{\UB_{\tau}(i)} \right).
\eean
Note that $\UB_{\tau}(i)$ does not depend on $i$ for every $\tau\neq I$; so we may write $\UB_{\tau}$ for $\UB_{\tau}(i)$ and rewrite the above probability as 
$O\left(\sum_{\tau\neq I} h(\tau)\right)$,
where 
\[
h(\tau)= M_2L_2\frac{(U_1+\imax)\imax}{M_1} \frac{U_{\tau}}{\UB_{\tau}} \rho^*_{\tau},\ \quad \rho^*_{\tau}=\max_{1\le i\le \imax} \rho_{\tau}(i),
\]
since by Lemma~\ref{lem:steps}, $\sum_{1\le i\le \imax}\sum_{P\in \state_i}   \sigma_i\le 10 \imax$ and so
\[
\sum_{1\le i\le \imax}\sum_{P\in \state_i}   \sigma_i \rho_{\tau}(i)\le \rho^*_{\tau}\sum_{1\le i\le \imax}\sum_{P\in \state_i}   \sigma_i =O(\imax \rho^*_{\tau}).
\]
Since $\LB(i)=\Omega(M_2L_2)$ and $\LBS_{\tau}(i)=\Omega(U_{\tau})$ for each $i$, 
by Lemma~\ref{lem:solution},
$\rho^*_{\tau}=O\left(\UB_{\tau}/M_2L_2U_{\tau}\right)$. Hence,
\[
h(\tau)= M_2L_2\frac{(U_1+\imax)\imax}{M_1} \frac{U_{\tau}}{\UB_{\tau}} O\left(\frac{\UB_{\tau}}{M_2L_2U_{\tau}}\right)=O\left(\frac{(U_1+\imax)\imax}{M_1}\right), \quad \mbox{for each $\tau\neq I$.}
\]
Noting that $\imax=O(M_2L_2/M_1^2)$, now $h(\tau)=o(1)$ for \jt{every} $\tau\neq I$ by the hypotheses of the lemma and so the probability of a pre-b-rejection in Phase 5 is $o(1)$.\qed 
 
\ss


 \no{\em Proof of Lemma~\ref{lem:drejb}.}
Now, we bound the probability of a b-rejection during Phase 5. 
At a given step $t$, we estimate the probability that a given switching $S=(P',P)$ is performed and is b-rejected at step $t$.  Let $\tau=\tau(P',P)$ denote the type of $S$. Assume that $P'\in \state_j$ and $P\in\state_i$. Hence, $j=i+1$ if $\tau=I$ and $j=i$ otherwise.

The probability that $S$ is performed, conditional on $P_{t-1}=P'$, is
\[
\frac{\rho_{\tau}(j)}{\UB_{\tau}(j)}.
\]
Conditional on that, the probability that $S$ is b-rejected is
$1-\LB(i)/b(P)$. Let $\state_{\tau}(P',P)$ denote the set of type $\tau$ switchings that \jt{converts} $P'$ to $P$.
So the probability that a b-rejection occurs during Phase 3 is at most
\bean
&&\sum_{t\ge 1}\sum_{\tau}\sum_{1\le i\le \imax}\sum_{1\le j\le\imax} \sum_{P\in \state_i}\sum_{P'\in\state_j\cap \state^+_{\tau}(P)} \sum_{S\in \state_{\tau}(P',P)} \pr(P_{t-1}=P') \frac{\rho_{\tau}(j)}{\UB_{\tau}(j)} \left(1-\frac{\LB(i)}{b(P)}\right)\\
 &&=\sum_{1\le i\le \imax}\sum_{P\in \state_i}\sum_{\tau} \sum_{1\le j\le\imax}  \sum_{P'\in\state_j\cap \state^+_{\tau}(P)} \sum_{S\in \state_{\tau}(P',P)} \sigma_j \frac{\rho_{\tau}(j)}{\UB_{\tau}(j)} \left(1-\frac{\LB(i)}{b(P)}\right),
 \eean
as $\sum_{t\ge 1}\pr(P_{t-1}=P')=\sigma(P')$ which is $\sigma_j$ for every $P'\in\state_j$.

Note that after fixing $i$, for every $\tau$ and $j$ such that 
 $\state_j\cap \state^+_{\tau}(P)$ is non-empty for $P\in\state_i$, $\rho_{\tau}(j)$ were chosen so that
 \[
 \frac{\sigma_j\rho_{\tau}(j)}{\UB_{\tau}(j)}=q(i).
 \]
We also have
\[
b(P)=\sum_{\tau} \sum_{1\le j\le\imax}  \sum_{P'\in\state_j\cap \state^+_{\tau}(P)}\sum_{S\in \state_{\tau}(P',P)}1.
\]
So the above expression is
\[
\sum_{1\le i\le \imax}\sum_{P\in \state_i} q(i) (b(P)-\LB(i))=\sum_{1\le i\le \imax}|S_i| q(i) \ex(b(P)-\LB(i)),
\]
where the expectation is taken over a random $P\in \state_i$.

The above is
\bean
\sum_{1\le i\le \imax}|S_i| q(i) \ex(b(P)-\LB(i))&\le& \sum_{1\le i\le \imax}|S_i| \frac{\sigma(i+1)}{\UB_{I}(i+1)} \ex(b(P)-\LB(i))\\
&=&O\left(\sum_{1\le i\le \imax}|S_i| \sigma(i)\frac{\ex b(P)-\LB(i)}{\LB(i)}\right)
\eean
by~\eqn{qi} and~\eqn{xratio}.
By Lemma~\ref{lem:steps}, $\sum_{1\le i\le \imax}|S_i| \sigma(i)=O(\imax)$ and so the above is $o(1)$ if we can show that for every $0\le i\le \imax$, 
\be
\frac{\ex b(P)-\LB(i)}{\LB(i)}=o(1/\imax).\lab{expb-rej}
\ee

 By~\eqn{Zstar},~\eqn{LB0} and Lemma~\ref{lem:lowerbounds} we have
 \[
 b(P)= M_2L_2-O(\imax(d_hM_2+d_1L_2+d_1^2d_h^2+U_1^2)+M_2U_1+L_4).
 \]
 We also have $\imax=O(M_2L_2/n^2)$.
 Hence,
 \[
 \frac{\ex b(P)-\LB(i)}{\LB(i)}\cdot \imax=O\left(\frac{\imax(d_hM_2+d_1L_2+d_1^2d_h^2+U_1^2)+M_2U_1+L_4}{M_1^2}\right).
 \]
 By the hypotheses of the lemma, the probability of a b-rejection during Phase 5 is $o(1)$. \qed
\ss


\no {\em Proof of Lemma~\ref{lem:loopbounds}.}
   The proof is a slight modification of~\cite[Lemma 15]{GWregJournal}.   For the upper bound of $f(P)$, the number of ways to choose a loop and label its end points in $P$ is $2i$. Then, the number of ways to choose another two pairs and label their end points is at most $M_1^2$. This shows that $f(P)\le \UB_{\ell}(i)$ for any $P\in\state_i$. 
For the lower bound, note that each pair of $y$ and $z$ cannot be contained in a loop or a multiple edge, the total number of such choices is bounded by $O(iM_1(i+B_D+B_T))$. Moreover, $u_1$ and $u_2$ (same for $u_3$ and $u_4$) cannot be the same vertex as $u_0$ or be adjacent to $u_0$. The total number of forbidden choices like this is bounded by $iM_1(d_1+U_1)$. This verifies the lower bound for $f(P)$.

Next we prove the bounds for $b(P)$. To perform an inverse switching on $P$, we pick a light 2-star and label the points in the pairs by \{3,1\} and \{2,5\}, with 1 and 2 being in the same light vertex $u_0$, together with another pair $x$ whose end points are labeled by $4$ and $6$. There are $L_2$ ways to pick the light 2-star and $M_1$ ways to pick and label $x$. This yields the upper bound for $b(P)$. We need to exclude choices so that 
\begin{enumerate}
\item[(a)] the 2-star contains a loop, or a double edge or a triple edge; or there is a loop at $u_0$;
\item[(b)] $x$ is a loop or is contained in a multiple edge;
\item[(c)] $x$ and the 2-star share a vertex;
\item[(d)] $u_1$ and $u_2$ are adjacent or $u_3$ and $u_4$ are adjacent. 
\end{enumerate}

We first show that the number of choices for (a) is at most $2(2id M_1+ 4B_DdM_1+ 6B_Td M_1)$. We first count the choices so that \{1,3\} is a loop. There are $2i$ ways to choose a loop and label their end points by 1 and 3. Then, there are at most $d_h$ ways to choose $2$ and $5$ as $u_0$, the vertex containing point $1$ has degree at most $d_h$. Therefore, the number of ways to choose the 2-star and $x$ so that \{1,3\} is a loop is at most $2id_hM_1$. Similarly, the number of such choices so that $\{1,3\}$ is contained in a double edge or a triple edge is at most $4B_D d_hM_1$ and $6B_Td_h M_1$ respectively as there are at most $B_D$ double edges and $B_T$ triple edges. The extra factor $2$ accounts for the case that \{2,5\} is a loop or is contained in a double edge or triple edge. The number of choices so that there is a loop at $u_0$ is at most $id_h^2 M_1$ as there are $i$ ways to choose $u_0$.
Hence, the number of choices for (a) is at most $2d_hM_1(2i+4B_D+6B_T+id_h/2)$.

The number of choices for (b) is at most $L_2(2i+4B_D+6B_T)$, for (c) is at most $6L_2 d_1$ and for (d) is at most $2L_2U_1$. Thus 
\bean
b(P)&\ge& L_2M_1-2d_hM_1(2i+4B_D+6B_T+id_h/2)-L_2(2i+4B_D+6B_T+6d_1+2U_1)=\LB(i).\qed
\eean

\ss

 \no {\em Proof of Lemma~\ref{lem:lrej}.}
  By Lemma~\ref{lem:loopbounds}, the probability of an f-rejection in each iteration is
\[
O\left(\frac{i+B_D+B_T+d_1+U_1}{M_1}\right),
\]
and the probability of a b-rejection in each iteration is
\[
O\left(\frac{d_hM_1(\imax+B_D+B_T+\imax d_h)+L_2(\imax+B_D+B_T+d_1+U_1)}{L_2M_1}\right)
\]
Since the number of loops decreases by one in each iteration, Phase 3 lasts for at most $\imax\le B_L=O(L_2/M_1)$ steps. 
 Now the probability of an f-rejection or a b-rejection in Phase 3 is bounded by
 \bean
 &&O\left(\frac{d_hM_1(B_L+B_D+B_T+B_Ld_h)+L_2(i_0+B_D+B_T+d_1+U_1)}{M_1^2}\right)\\
 &&= O\left(\frac{d_hM_1(B_D+B_T+B_Ld_h)+L_2(d_1+U_1)}{M_1^2}\right),
  \eean
  by noting that $L_2\le d_h M_1$. Hence, this probability is $o(1)$ by our hypotheses.\qed

\ss


\no {\em Proof of Lemma~\ref{lem:triplebounds}.}
For $f(P)$: there are $i$ ways to choose a triple edge and then $12$ ways to label their end points and then at most $M_1^3$ ways to choose the other $3$ pairs. This gives the upper bound. The lower bound follows by an analogous argument as in Lemma~\ref{lem:loopbounds}.

For $b(P)$: in order to perform an inverse switchings on $P$, we choose an ordered pair of 3-stars, the first being light, and label their end points as in Figure~\ref{f:triple} (here $u_1$ is required to be light). The number of such choices is $L_3M_3$. This gives the upper bound for $b(P)$. Compare this inverse switching with the type III switching in Phase 5. It is easy to see that $b(P)$ equals the number of type III switchings applicable on $P$. Hence, the same analysis in Lemma~\ref{lem:drejf} for the lower bound on the number of type III switchings on a pairing $P$ (the arguments below~\eqn{err-cond}) works here for the lower bound on $b(P)$, except that for case (a): one of the 3-stars contains a multiple edge, we need to consider also contribution from triple edges instead of just double edge. This gives  $3M_3(4B_Dd_h^2+6id_h^2)+3L_3(4B_Dd_1^2+6id_1^2)$. The discussions for cases (b,c) are exactly the same. Thus, we obtain that 
  %
  %
  %
  %
  \bean
  b(P)&\ge& M_3L_3-3M_3(4B_Dd_h^2+6id_h^2)-3L_3(4B_Dd_1^2+6id_1^2)-L_6-12M_3U_2-3M_2U_3\\
  &&-M_3U_3-3M_2U_1U_2\ge \LB_{t}(i),
  \eean
  by noting that $M_3U_3\ge M_2U_3$ and $M_3U_3\ge M_3U_2$. \qed
     \ss
     

 \no {\em Proof of Lemma~\ref{lem:triplebounds}. }
The number of iterations in Phase 4 is at most $\imax=O(L_3M_3/M_1^3)$. In iteration $t$, assuming $P_{t-1}\in\state_i$, the probability of an f-rejection is
\[
1-\frac{f(P_{t-1})}{\UB{t}(i)}=O\left(\frac{\imax+B_D+d_1+U_1}{M_1}\right)
\]
by Lemma~\ref{lem:triplebounds}.

Similarly, the probability of a b-rejection in iteration $t$ is at most
\[
1-\frac{\LB_{t}(i-1)}{\max_{P\in \state_{i-1}}b(P)}=O\left(\frac{M_3(B_Dd_h^2+id_h^2)+L_3(B_Dd_1^2+id_1^2)+L_6+M_3U_3+M_2U_1U_2}{M_3L_3}\right).
\]
Hence, the probability of an f-rejection or a b-rejection during Phase 4 is
\[
O\left(\frac{(B_T+B_D+d_1+U_1)L_3M_3}{M_4}+\frac{M_3d_h^2(B_D+B_T)+L_3(B_Dd_1^2+B_Td_1^2)+L_6+M_3U_3+M_2U_1U_2}{M_1^3}\right)
\]
which is $o(1)$ by our lemma hypothesis by noting that $(B_T+B_D)L_3M_3=O((B_T+B_D)M_3M_1d_h^2)$.\qed

\no{\em Proof for Lemma~\ref{lem:heavy-compute}.\ } We prove bounds for the computation time of $f_{i,j}(P)$ and $b_{i,j}(P,m)$ only; the arguments for $f_i(P)$ and $b_i(P,m)$ are similar. 

Given $P$, let $m$ denote the number of pairs between $i$ and $j$. Assume that $m\ge 1$ and we want to compute $f_{i,j}(P)$ exactly. We say that an ordered pair $(p_1,p_2)$ in $P$ is {\em nice} if either (i) both $p_1$ and $p_2$ are contained in light vertices, or (ii) $p_1$ is contained in a light vertex, and $p_2$ is contained in a heavy vertex $w\neq j$ such that there is no pair between $j$ and $w$, or (iii) $p_2$ is contained in a light vertex and $p_1$ is contained in a heavy vertex $w\neq i$ such that there is no pair between $i$ and $w$. To perform a heavy multiple edge switching, we need to order the $m$ pairs between $i$ and $j$ and switch them away together with another $m$ nice ordered pairs, as shown in Figure~\ref{f:heavy1}. There are $m!$ ways to order the $m$ pairs between $i$ and $j$. Then, $f_{i,j}(P)$ equals to $m!$ multiplied by the number of ways to sequentially choose $m$ distinct nice ordered pairs, $(p_{k_1},p_{k_2}),\ldots, (p_{k_{2m-1}},p_{k_{2m}})$ satisfying
\begin{enumerate}
 \item[(a)] no $p_{k_{2g-1}}$ and $p_{k_{2h-1}}$ are in the same heavy vertex for some $1\le g<h\le m$;\lab{conda}
 \item[(b)]  no $p_{k_{2g}}$ and $p_{k_{2h}}$ are in the same heavy vertex for some $1\le g<h\le m$.\lab{condb}
  \end{enumerate} 
  We describe a computing scheme for this number. Let $I$ be the set of heavy vertices that are not adjacent to $i$ and $J$ be the set of heavy vertices that are not adjacent to $j$ in $P$; note that $I\cap J$ is not necessarily empty. Let $0\le \ell\le m$. We first compute the number of choices such that there are exactly $\ell$ nice ordered pairs of cases (ii) or (iii), denoted by $(p_{k_1},p_{k_2}),\ldots, (p_{k_{2\ell-1}},p_{k_{2\ell}})$, and exactly $m-\ell$ nicely ordered pairs of case (i). Let $Z$ denote the total number of nice ordered pairs of case (i) in $P$.  For each vertex $w\in I$, it can contain at most one point from $p_{k_{2g-1}}$, for $1\le g\le \ell$, and for each vertex $v\in J$, it can contain at most one point from $p_{k_{2g}}$, for $1\le g\le \ell$. For each $w\in I$, we can compute $X_w$, the number of nice ordered pairs $(p_1,p_2)$ with $p_1\in w$. For each $v\in J$, we can compute $Y_v$, the number of nice ordered pairs $(p_1,p_2)$ with $p_2\in v$.  Now, the generating function for the choices of the nice ordered pairs satisfying conditions (a,b) is
  \[
  \prod_{w\in I}(1+X_wz) \prod_{v\in J} (1+Y_v z).
  \]
Thus, the number of such choices with exactly $\ell$ pairs of cases (ii) or (iii) is
\[
\binom{m}{\ell}[Z]_{m-\ell} [z^{\ell}]\prod_{w\in I}(1+X_wz) \prod_{v\in J} (1+Y_v z),
 \] 
 where $\binom{m}{\ell}$ is the number of ways to choose the $\ell$ out of $m$ pairs of nice ordered pairs to be of cases (ii) or (iii). For the $m-\ell$ pairs of case (i), there are $[Z]_{m-\ell}$ ways to sequentially take $m-\ell$ distinct pairs. There are then $[z^{\ell}]\prod_{w\in I}(1+X_wz) \prod_{v\in J} (1+Y_v z)$ ways to sequentially choose $\ell$ nice ordered pairs satisfying conditions (a,b).
 So,
 \[
 f_{i,j}(P)=\sum_{0\le \ell\le m}\binom{m}{\ell}[Z]_{m-\ell} \cdot [z^{\ell}]\prod_{w\in I}(1+X_wz) \prod_{v\in J} (1+Y_v z).
 \]
 Thus, computing all of these numbers take $O(H_1)$-time. Similarly, computing $Z$ takes $O(H_1)$-time as $Z=M_1-Z'$, where $Z'$ counts the number of pairs incident with at least one heavy vertex, and so computing $Z'$ takes only $O(H_1)$-time. Next, we show how to compute
 $[z^{\ell}]\prod_{w\in I}(1+X_wz) \prod_{v\in J} (1+Y_v z)$ for each $0\le \ell\le m$. 

Denote $I$ by $w_1,w_2,\ldots w_g$ and $J$ by $v_1,v_2,\ldots,v_{h}$. Let ${\bf a}=(a_0,\ldots, a_m)$ be a vector, initialised by $a_0=1$ and $a_k=0$ for all $1\le k\le m$. Note that ${\bf a}$ is the coefficient sequence of $P_0(z)=1$ (we only track the first $m+1$ coefficients).

Let $P_1(z)=(1+X_{w_1}z)$ and $P_2(z)=P_1(z)(1+X_{w_2}z)$ and repeat. Thus, $P_{g+h}(z)=\prod_{w\in I}(1+X_wz) \prod_{v\in J} (1+Y_v z)$. We will update ${\bf a}$ so that after iteration $k$, ${\bf a}$ is the coefficient sequence of $P_k(z)$. By the inductive definition of $P_k(z)$, it is easy to see that, for each $0\le \ell\le m$,
\be
[z^{\ell}] P_k(z) =
\left\{
\begin{array}{ll}
 {[z^{\ell}]} P_{k-1}(z)+X_{w_k}{[z^{\ell-1}]} P_{k-1}(z) & \mbox{for $1\le k\le g$}\\
 {[z^{\ell}]} P_{k-1}(z)+Y_{v_{k-g}}{[z^{\ell-1}]} P_{k-1}(z) & \mbox{for $g+1\le k\le g+h$}
 \end{array}
 \right. \lab{update}
\ee
\remove{
\[
a^{k}_{\ell} =
\left\{
\begin{array}{ll}
 a^{k-1}_{\ell}+X_{w_k}a^{k-1}_{\ell-1} & \mbox{for $1\le k\le g$}\\
 a^{k-1}_{\ell}+Y_{v_{k-g}}a^{k-1}_{\ell-1} & \mbox{for $g+1\le k\le g+h$}
 \end{array}
 \right.
\]
}

Note that~\eqn{update} gives the updating rule for ${\bf a}$ in each iteration. Since each iteration runs in $O(m)$-time, the total number of steps for computing 
$[z^{\ell}]\prod_{w\in I}(1+X_wz) \prod_{v\in J} (1+Y_v z)$ for every $0\le \ell\le m$   is $O(m(|I|+|J|))=O(m|\heavy|)$. 
Thus, the total number of steps required to compute $f_{i,j}(P)$ is $O(H_1+m|\heavy|+m)=O(H_1+m|\heavy|)$.

 \remove{

      Let $X$ denote the number of nicely ordered pairs $(p_1,p_2)$ in $P$.  It takes $O(\Delta^2)$-time to compute $X$. To perform a heavy multiple edge switching on $P$, we much choose $m$ distinct nicely ordered pairs. Hence, there are $[X]_m$ ways to do it. But we must exclude choices such that there are two nicely ordered pairs $(p_1,p_2)$ and $(p_3,p_4)$ such that either $p_1$ and $p_3$ are in the same heavy vertex, or $p_2$ and $p_3$ are in the same heavy vertex. Let $B$ be all such choices. Then,
\[
f_{i,j}(P)=m![X]_m-|B|.
\]
\jcom{I will fill the formula for |B|. Should be easy with generating functions. But good if you can see whether I miss anything in the above formula.}
               It takes $O(H_1)$-running time to compute $H$, the number of heavy pairs in $P$. So, there are $2^m[M_1/2-H]_m$ ways to choose $m$ distinct light pairs and label their end vertices. We call them ordered pairs as their end vertices are labelled. For each $1\le k\le m$, let $B_k$ be the set of choices of these $m$ distinct ordered pairs such that switching the $k$-th pair between $i$ and $j$ and the $k$-th of the $m$ chosen pairs as described in Definition~\ref{def:heavymultiple} would create a heavy multiple edge or a heavy loop. Then 
\[
f_{i,j}(P)=m!(2^m[M_1/2-H]_m-|\cup_{1\le k\le m} B_k|).
 \] 
 \jcom{There is another type of bad choices. If two pairs are chosen that share a heavy vertex $w$. Then after switching, there can be a double edge created between $i$ and $w$. Are there other bad edges?}

By the inclusion-exclusion principle,
\bean
|\cup_{1\le k\le m} B_k|&=&\sum_{k=1}^m|B_k|-\sum_{1\le k_1<k_2\le m}|B_{k_1}\cap B_{k_2}|+\sum_{1\le k_1<k_2<k_3\le m}|B_{k_1}\cap B_{k_2}\cap B_{k_3}|\\
&&-\cdots+(-1)^{m-1}|B_1\cap B_2\cap\cdots\cap B_k|.
\eean
Hence, it sufficient to show that we can compute the size of each intersection above efficiently. For each $B_k$, the end points of the $k$-th pair are labelled $2m+2k-1$ and $2m+2k$. Let $w_1$ denote the vertex containing $2m+2k-1$ and $w_2$ denote the vertex containing $2m+2k$. Since it is a light pair and switching $\{2k-1,2k\}$ and $\{2m+2k-1,2m+2k\}$ to $\{2k-1,2m+2k-1\}$ and $\{2k,2m+2k\}$ would create a new heavy multiple edge or heavy loop, it must be that either (i) $w_1=i$ and $w_2$ is light, or (ii) $w_1\neq i$ is heavy and there is one (or more) pair between $w_1$ and $i$ and $w_2$ is light, or (iii) $w_1$ is light and $w_2=j$, or (iv) $w_1$ is light, $w_2\neq j$ is heavy and there is one (or more) pair between $w_2$ and $j$.  Let $X_1$, $X_2$, $X_3$ and $X_4$ denote the number of choices for the $k$-th pair in each of the above four cases. Clearly, these four numbers can be computed in $O(\Delta^2)$-running time. Given the choice of the $k$-th pair, there are exactly $2^{m-1}[M_1/2-H-1]_{m-1}$ ways to choose the other $m-1$ distinct ordered light pairs. Hence, 
\[
\sum_{k=1}^m |B_k|=m(X_1+X_2+X_3+X_4)2^{m-1}[M_1/2-H-1]_{m-1},
\] 
which can be computed in $O(\Delta^2)$-running time.

Now we consider $|B_{k_1}\cap \cdots \cap B_{k_{\ell}}|$ for any $2\le \ell\le m$. The $k_{h}$-th pair must fall into the above four cases for every $1\le h\le \ell$. Let $\ell_h$, $1\le h\le 4$ denote the number of these $\ell$ pairs that are in case (i), (ii), (iii) and (iv) respectively. Then, given $(\ell_1,\ell_2,\ell_3,\ell_4)$, the number of ways to choose these $\ell$ distinct ordered light pairs is $\prod_{h=1}^4\binom{X_h}{\ell_h}$ and the number of ways to choose the other $m-\ell$ distinct ordered light pairs is
  $2^{m-\ell}[M_1/2-H-\ell]_{m-\ell}$. Hence, for every $2\le \ell\le m$,
  \[
  \sum_{1\le k_1<\cdots<k_{\ell}\le m} |B_{k_1}\cap \cdots \cap B_{k_{\ell}}|= \binom{m}{\ell} 2^{m-\ell}[M_1/2-H-\ell]_{m-\ell} \sum_{\ell_1,\ell_2,\ell_3,\ell_4}\prod_{h=1}^4\binom{X_h}{\ell_h},
  \]
        where each $X_h$ has already been computed and the summation is over all $(\ell_1,\ell_2,\ell_3,\ell_4)$ such that $\sum_{h=1}^4 \ell_h=\ell$. Since
        \[
        \sum_{\ell_1,\ell_2,\ell_3,\ell_4}\prod_{h=1}^4\binom{X_h}{\ell_h}=[z^{\ell}] \prod_{h=1}^4 \sum_{n\ge 0}\binom{X_h}{n} z^n=[z^{\ell}] (1+z)^{X_1+X_2+X_3+X_4}=\binom{X_1+X_2+X_3+X_4}{\ell},
        \]
 the computation of  $\sum_{1\le k_1<\cdots<k_{\ell}\le m} |B_{k_1}\cap \cdots \cap B_{k_{\ell}}|$ takes $O(1)$-time as $X_h$ has been computed already. Hence, the total time complexity for computing $f_{i,j}(P)$ is $O(H_1+\Delta^2)+O(m)=O(\Delta^2)$.      
         }

Next, we bound the computation time for $b_{i,j}(P,m)$, where $P$ is a pairing with no pairs between $i$ and $j$ and $m\ge 1$ is an integer. Recall that $W_{i,j}(P)$ denotes the number of points in $i$ that belong to heavy loops or heavy multiple edges with one end in $i$ and the other end not in $j$ in $P$. Now for $1\le k\le m$, let $B_k$ denote the set of choices of pairs $\{2g-1,2m+2g-1\}$, $\{2g,2m+2g\}$ for $1\le g\le m$, as in Definition~\ref{def:heavymultiple} for the inverse switching, such that (a) point $2g-1$ is not counted by $W_{i,j}(P)$ and point $2g$ is not counted by $W_{j,i}(P)$, and (b) switching $\{2k-1,2m+2k-1\}$ and $\{2k,2m+2k\}$ to $\{2k-1,2k\}$, $\{2m+2k-1,2m+2k\}$ would create new heavy loops or heavy multiple edges.  Then,
\[
b_{i,j}(P,m)=[d_i-W_{i,j}]_m[d_j-W_{j,i}]_m-|\cup_{1\le k\le m} B_k|.
\]

It is easy to see that $W_{i,j}$ and $W_{j,i}$ can be computed in $O(\Delta)$-time. By the inclusion-exclusion principle, we can express $|\cup_{1\le k\le m} B_k|$ by
\[
\sum_{k=1}^m|B_k|-\sum_{1\le k_1<k_2\le m}|B_{k_1}\cap B_{k_2}|+\sum_{1\le k_1<k_2<k_3\le m}|B_{k_1}\cap B_{k_2}\cap B_{k_3}|-\cdots+(-1)^{m-1}|B_1\cap B_2\cap\cdots\cap B_k|.
\]
If a set of $2m$ pairs are in $B_k$, then we must have that the vertex containing point $2m+2g-1$, denoted by $w_1$, and the vertex containing point $2m+2g$, denoted by $w_2$, are both heavy. Let $Y_1$ and $Y_2$ denote the number of points in $i$ and $j$ respectively that are paired with a point in a heavy vertex via a single edge in $P$. Then, $|B_k|=Y_1Y_2[d_i-W_{i,j}-1]_{m-1}[d_j-W_{j,i}-1]_{m-1}$ for every $1\le k\le m$. Hence,
\[
\sum_{k=1}^m|B_k|=mY_1Y_2[d_i-W_{i,j}-1]_{m-1}[d_j-W_{j,i}-1]_{m-1}.
\]
Clearly, both $Y_1$ and $Y_2$ can be computed in $O(\Delta)$-time.

Now for any $2\le \ell\le m$, consider any $|B_{k_1}\cap B_{k_2}\cap\cdots\cap B_{k_{\ell}}|$. It is easy to see that given $k_1,\ldots,k_{\ell}$,
\[
|B_{k_1}\cap B_{k_2}\cap\cdots\cap B_{k_{\ell}}|=[Y_1]_{\ell}[Y_2]_{\ell}[d_i-W_{i,j}-\ell]_{m-\ell}[d_j-W_{j,i}-\ell]_{m-\ell}.
\]
Thus,
\[
\sum_{1\le k_1<\cdots<k_{\ell}\le m}|B_{k_1}\cap B_{k_2}\cap\cdots\cap B_{k_{\ell}}|=\binom{m}{\ell}[Y_1]_{\ell}[Y_2]_{\ell}[d_i-W_{i,j}-\ell]_{m-\ell}[d_j-W_{j,i}-\ell]_{m-\ell}.
\]
Since $Y_1$ and $Y_2$ have been computed, computing each $\sum_{1\le k_1<\cdots<k_{\ell}\le m}|B_{k_1}\cap B_{k_2}\cap\cdots\cap B_{k_{\ell}}|$ takes only $O(1)$-time. Hence, the computation of $b_{i,j}(P,m)$ takes $O(\Delta+m)$ steps.
\qed
\ss

\no {\bf Running time in Phases 3--5}

We first discuss the time complexity of Phase 5, which significantly dominates that of the other two phases.

 Again, each step of the algorithm involves information on $f_{\tau}(P)$, $\b_{\tau}(P,\pairstars)$  and $b(P)$. However, it is easy to see that it is never necessary to compute $f_{\tau}(P)$, due to the way the $\UB_{\tau}(i)$ is defined. For instance, $\UB_{I}(i)=4iM_1^2$ which is the total number of ways to pick a double edge and label its end points, and pick two pairs,  with repetition allowed, and label their end points. If such a choice does not yield a valid switching then  an f-rejection is performed. It is easy to see that the f-rejection is performed with the correct probability, as defined in~\eqn{frejDef}. Thus, it suffices to compute $\b_{\tau}(P,\pairstars)$ and $b(P)$ only. 

Consider the computation of $b(P)$. By~\eqn{Zstar}, $b(P)=M_2L_2-Z^*(P)$. Note that $Z^*(P)$ is the number of pairs of $2$-paths involving at most $5$ vertices, or pairs of $2$-paths looking like structures in Figure~\ref{f:H} but involving double edges. It has been shown in the proof of Lemma~\ref{lem:lowerbounds} that $Z^*(P)\le 8i(d_h M_2+d_1 L_2)+(2id_1^2d_h^2+4iU_1^2+8M_2 U_1+L_4)$. By the assumption in Lemma~\ref{lem:drejb}, we have $Z^*(P)=o(M_1^2)$. Hence, computing $Z^*(P)$ takes no more than $M_1^2$ units of time by brute-force search. So the total computation time for $b(P)$ is at most $M_1^2$.


Next consider the computation of  $\b_{\tau}(P,\pairstars)$ for $\tau\in\{III,IV,V,VI,VII\}$. Each such function counts the number of ways to choose a structure corresponding to one of those in Figure~\ref{f:H}, plus several other pairs with restrictions that between certain pairs of vertices
 there should be no edge. 
An important quantity in these bounds will be  an upper bound $I_j$ on the number  of $j$-stars which have a given vertex $v$ as a leaf. These are similar to  light $(j-1)$-blooms at $v$, but without the lightness condition. For $j=2$, it also corresponds to a 2-path starting at $v$. Thus, 
\be
I_2\le U_1=O(n^{(2\gamma-3)/(\gamma-1)^2}).\lab{I2}
\ee
 A  $j$-star is a set of $j$ pairs all of which include a point in some given vertex $u$. We clearly have $I_j\le \sum_{i=1}^{\Delta} d_i^{j-1}$ where the degrees of the graph are $d_1\ge d_2\ge \cdots$; here $d_i$ represents the possible degree of vertex $u$. Hence for any $j\ge 3$,
\bea
I_j&\le& \sum_{i=1}^{\Delta} d_i^{j-1} \le \sum_{i=1}^{\Delta} (Kn/i)^{(j-1)/(\gamma-1)}\le (Kn)^{(j-1)/(\gamma-1)} \sum_{i\ge 1} i^{-(j-1)/(\gamma-1)}.\lab{Ik}
\eea

 We first focus on 
$\tau=III$, as a simple illustration of the method. Here, we have $\b_{III}(P,\pairstars)$ countings the choices  for pairs representing $H_1$ in Figure~\ref{f:H} and an extra pair that is not too close to the $H_1$. 
Using  brute force, we can find all copies of  $H_1$ in time $O(L_2I_3)$ by breadth-first search starting at light vertices $v$ (recall $L_k$ below~\eqn{dh}) and investigating all 3-stars with $v$ as a leaf. For each copy of $H_1$, we can count the valid choices of the extra pair $u_4v_4$ in Figure~\ref{f:typeIII} by inclusion-exclusion: make a list of each possible forbidden event (vertex coincidences, and edges present that are not allowed) and then, for each subset $S$ of the forbidden events, count that choices for which all events in $S$ hold. For $S=\emptyset$, the number is $2M_1$, without any computation. For all other $S$, the counting is done by breadth-first search. For the $S$ involving a vertex coincidence, there is a  trivial bound of $O(\Delta)$; the pair $u_4v_4$ can be chosen in this many ways starting at the specified vertex. Once the subgraph with the vertex coincidence is constructed, any other events specified in $S$ can be checked to see if this particular subgraph satisfies them all. Similarly, it is easy to see that all with an edge coincidence can be bounded above by $O(I_2)$.   Hence $O(L_2I_3I_2)$ is the complexity for this case.

Now consider the other types $\tau$. Each of these involves choosing a 2-path with a light vertex in the middle, as well as two other pairs $u_1u_2$ and $u_1u_3$,  together with a set of up to seven edges other edges with various forbidden events. We use an inclusion-exclusion method as for type III.

 We treat type VII in detail, as the most complicated case. We can check every occurrence of $H_1$ as in the type III case, in time $O(L_2I_3)$, and check each for the existence of the extra pairs $u_2v_2$ and $u_3v_3$ in constant time. Then take any subset $S$ of all the single events that must be avoided, and check for each copy of $H_1$ how many choices of the remaining seven pairs satisfy all events in $S$.

 Note that   for $S$ to be satisfied, there must   exist a   set of edges inducing a graph of a given isomorphism type, and containing $H$ at a certain location. Pick such an isomorphism type, $W$ say. The first thing to notice is that if we delete the vertices in $H_1$, $W$ falls into components that can be treated separately: for each component we can find the number of ways to place that component, plus the specified edges in $W$ that join in to $H$, and multiply together the counts for the various components. From this it is easy to see that the worst case is when there is only one such component. (If that is not clear now, it is at the end of this proof.) Consider the graph $W_0$  consisting of $W-H_1$  together with {\em one} edge  of $W$ from $H_1$ to $W-H_1$. (If there is no such edge, then  the number of possibilities for $W$ can be computed in the following manner independently of $H_1$, and it will be easy to see by the time we reach the end of this proof that this case is insignificant.)   Take a spanning tree $T$ of $W_0$. We just need to find all possible attachments of $T$ to the copy of $H$, as each can be investigated in constant time for whether  the required extra edges are present. Since there was only one component in $W$, the seven edges involve at most eight vertices, and so $T$ has at most eight edges. One leaf, call it   $w$, of $T$, is a vertex of $H_0$. We can decompose $T$, as an abstract tree, into stars, as follows. If we pull off all leaves of $T$ other than $W$, the new leaves are centres of stars. We remove the edges of these stars from $T$ and repeat with the tree formed from the remaining edges. In this way, $T$ is decomposed into stars that all contain at least two edges each, except perhaps there might be one star containing  a single  edge   incident with $w$. Moreover, each star contains the edge that goes from its centre towards $w$.  These stars can be placed in the graph, starting with the one containing $w$, and then successively building outwards. The number of possibilities for a single edge is at most $\Delta$, and for a $k$-star it is at most $I_k$. The time taken to do this is thus the maximum of $R_8$ and $\Delta R_7$, where $R_k$ is the maximum, over all partitions $\pi_1+\cdots +\pi_\ell$ of $k$, with each $\pi_i\ge 2$, of $\prod_i I_{\pi_i}$. Using the bounds for $I_j$ in~\eqn{I2} and~\eqn{Ik} it is easy to see that the maximum is achieved by $n^{7/(\gamma-1)}$, the upper bound for $I_8$. 

The argument is easy to adapt to the other types. For each type, we obtain the following bounds on the running time in a single switching step. For types IV and VI,  $M_2L_2$ is an upper bound for $u_i$ and $v_i$, $1\le i\le 3$, and the rest of the arguments are the same as above, adapted to the smaller number of extra edges involved. For convenience, we repeat the bound obtained above for type III.
 
\[
\mbox{III: } L_2 I_3 I_2\,,\quad \mbox{IV: } M_2 L_2 I_4\,,\quad \mbox{V: } L_2I_3I_5\,,\quad \mbox{VI: }  M_2L_2I_7\,,\quad \mbox{VII: }  L_2I_3I_8.
\]
Let $\rho^*_{\tau}$ denote $\max_{0\le i\le \imax}\rho_{\tau}(i)$. Then the expected total cost for computing $\b_{\tau}(P,\pairstars)$ during the algorithm, for a given $\tau$, is bounded  by the above quantity times $B_D\rho^*_{\tau}$.  With the value of $\rho^*_{\tau}$ in Lemma~\ref{lem:solution}, it is easy to see after some computations that 
  the type VII  term, $B_DM_2^3M_3L_3I_3I_8/M_1^7$   dominates. 
  Recall that $B_D=4L_2M_2/M_1^2$ from~\eqn{imax}. For the range of $\gamma$ in Theorem~\ref{thm:main} and for sufficiently small $\delta$ satisfying~\eqn{deltaRange}, this shows the total running time   is bounded above by 
  
 \[
  O( M_2^4M_3L_2L_3I_3I_8/M_1^9)=O(n^{4.081})
  \]
  
 for the range of $\gamma$ in Theorem~\ref{thm:main} and for sufficiently small $\delta$ satisfying~\eqn{deltaRange}.
 

Moreover, since $B_D\rho^*_{\tau}=o(1)$ for $\tau\in\{IV,V,VI,VII\}$ and   a.a.s.\ Phase 5 lasts only $O(B_D)$ iterations by Lemma~\ref{lem:steps}, the probability that at least one switching of one of  these types  is   performed in one run of \PL\ is $o(1)$. Therefore, a.a.s.\ the    computation time of  $\b_{\tau}(P,\pairstars)$ for $\tau\in\{III,IV,V,VI,VII\}$ is bounded by
\[
 B_D L_2 I_3 I_2\rho^*_{III}=O(  M_3 L_2 L_3  I_2 I_3  /M_1^3) = O(n^{2.89}).
\]
This can be improved by the judicious use of data structures.  Before running through all possible $H_1$, we can compute and  record the number of single edges incident with every vertex, the number of 2-paths starting at any vertex, and the numbers of 2-paths and of 3-paths connecting any two vertices. The time taken for this pre-computation is at most the time taken to generate all 3-paths, which is $O(M_2\Delta)$. Multiplying this by   $B_D$ still gives easily $O(n^2)$. Using this information, it is easy to compute, for a given copy of $H_1$, the number of choices for the extra edge satisfying a given set of constraints corresponding to $S$, in constant time. Then the expected complexity for this case is reduced to $O(B_D L_2I_3\rho^*_{III})=O(n^{2.107})$. 

It is easy to bound the  expected  computation time in Phase  3 (reduction of loops) using   brute-force searching, by $O(n^{2.087})$. For Phase 4, we can consider  for each pair of 2-stars ($L_2M_2$ possibilities), exactly how many ways to extend it to a pair of 3-stars (ordered of course), using the same data structures as introduced just above for type III switchings and a similar inclusion-exclusion scheme. The resulting complexity is  $O(B_TL_2M_2) =O(n^{2.107})$. 

 Hence, we obtain the time complexities claimed in Theorem~\ref{thm:main}. 

We remark that slight improvements in the running time bounds can be gained by making more use of data structures as we did for type III.  In~\cite{MWgen}, very significant gains were possible in the case of regular graphs by this means, but in the present case no such large further gains are easily apparent. So, to keep the argument simple, we don't pursue this here.

\no {\bf Constraints on $(\gamma,\delta)$}

Various constraints on $M_k$, $L_k$ and $U_k$ were placed to enforce  a small probability of a rejection during the algorithm. In this section, we list all these constraints and we deduce a condition on $\gamma$ under which we may find an appropriate $\delta$ so that all these constraints are satisfied. That completes the definition of $\light$ and $\heavy$ and thus completes the definition of the algorithm.

We always assumed $5/2<\gamma<3$; we also assumed $1-\delta>1/(\gamma-1)$ in~\eqn{anotherCondition}. Collecting constraints from Lemma~\ref{lem:Phi0}, Lemmas~\ref{lem:heavy-f-rej},~\ref{lem:heavy-b-rej} and~\ref{lem:heavyloopcond}, we have the following constraints for $\delta$ and $\gamma$.
\bea
5/2&<&\gamma<3\lab{0}\\
\frac{4}{\gamma-1}-2&<&\delta(\gamma-2)\lab{20}\\
\frac{3-\gamma}{\gamma-2}&<&\delta\\
\frac{1}{2\gamma-3}&<&\delta<\frac12 \lab{21}\\
\frac{1}{\gamma-1}&<&1-\delta \lab{1}\\
\frac{2}{\gamma-1}&<&1+\delta(\gamma-2)\lab{18}\\
\delta(\gamma-1)&>&\frac23.\lab{19}
\eea

Lemma~\ref{lem:A0rej} requires $L_4=o(M_1^2)$ and $L_4M_4=o(M_1^4)$.
Using~\eqn{Lk} and~\eqn{Delta},   these are implied by 
\bea
\delta(5-\gamma)&<&1 \lab{4}\\
\delta(5-\gamma)+\frac{4}{\gamma-1}&<&3. \lab{9}
\eea

Recall from~\eqn{xi} that
\[
\xi=\frac{32M_2^2}{M_1^3}.
\]
Lemma~\ref{lem:xi} requires that $\delta+1/(\gamma-1)<1$ and $\xi=o(1)$. The first condition is taken into consideration in~\eqn{1}. The constraint $\xi=o(1)$ requires $4/(\gamma-1)<3$ and thus $\gamma>7/3$ which is guaranteed by~\eqn{0}.

Recall from~\eqn{imax} that
\[
B_L=\frac{4L_2}{M_1}; \quad B_D=\frac{4L_2M_2}{M_1^2} \quad B_T=\frac{2L_3M_3}{M_1^3},
\]
Lemma~\ref{lem:typerej} requires $B_D\xi=o(1)$  which is guaranteed by 
\bel{2}
\delta(3-\gamma)+\frac{6}{\gamma-1}<4.
\ee

Lemma~\ref{lem:drejf} requires $B_D\xi=o(1)$, $B_D(M_3  d_h^2+L_3 d_1^2)+L_6+M_3U_3+M_2U_1U_2=o(M_1^3)$ and $B_Dn^{(2\gamma-3)/(\gamma-1)^2-1}=o(1)$,
which are implied by~\eqn{2} and
\bea
\delta(5-\gamma)+\frac{5}{\gamma-1}&<&4, \lab{5}\\
\delta(7-2\gamma)+\frac{4}{\gamma-1}&<&3,\lab{10}\\
\delta(7-\gamma)&<&2, \lab{6}\\
\delta(4-\gamma)+\frac{3}{\gamma-1}&<&2, \lab{8}\\ 
\delta(3-\gamma)+\frac{2}{\gamma-1}+\frac{2\gamma-3}{(\gamma-1)^2}&<&2. \lab{3}
\eea

Lemma~\ref{lem:pre-b-rej} requires $U_1M_2L_2/M_1^2+(M_2L_2/M_1^2)^2=o(M_1)$ which is implied by~\eqn{3} and
\bea
2\delta(3-\gamma)+\frac{4}{\gamma-1}&<&3.\lab{11}
\eea

Lemma~\ref{lem:drejb} requires $(d_hM_2+d_1L_2+d_1^2d_h^2+U_1^2)M_2L_2/M_1^2+M_2U_1+L_4=o(M_1^2)$ which is implied by~\eqn{4},~\eqn{9},
\[
\frac{2}{\gamma-1}+\frac{2\gamma-3}{(\gamma-1)^2}<2\quad \mbox{(already implied by~\eqn{3})}
\]
 and
\bea
2\delta(3-\gamma)+\frac{3}{\gamma-1}&<&2\lab{12}\\
\delta(3-\gamma)+\frac{2}{\gamma-1}+\frac{2(2\gamma-3)}{(\gamma-1)^2}&<&3. \lab{7}
\eea

Lemma~\ref{lem:lrej} requires $d_hM_1(B_D+B_T+B_Ld_h)+L_2(d_1+U_1)=o(M_1^2)$ which is implied by~\eqn{4},
\bean
\delta(4-\gamma)+\frac{2}{\gamma-1}&<&2 \quad \mbox{(implied by~\eqn{8})}\\
\delta(5-\gamma)+\frac{3}{\gamma-1}&<&3 \quad \mbox{(implied by~\eqn{9})}\\
\eean
and
\bea
\delta(3-\gamma)+\frac{1}{\gamma-1}&<&1\lab{13}\\
\delta(3-\gamma)+\frac{2\gamma-3}{(\gamma-1)^2}&<&1.\lab{14}
\eea

Lemma~\ref{lem:trej} requires
$(d_1+U_1)L_3M_3=o(M_1^4)$ and $M_3d_h^2(B_D+B_T)+L_3(B_Dd_1^2+B_Td_1^2)+L_6+M_3U_3+M_2U_1U_2=o(M_1^3)$, which is implied by~\eqn{5},~\eqn{6},~\eqn{3},~\eqn{8},~\eqn{10}  and
\bea
\delta(4-\gamma)+\frac{4}{\gamma-1}&<&3\lab{15}\\
\delta(4-\gamma)+\frac{3}{\gamma-1}+\frac{2\gamma-3}{(\gamma-1)^2}&<&3 \quad \remove{\mbox{\jt{come from $U_1L_3M_3=o(M_1^4)$}}}\lab{16}\\
\delta(6-\gamma)+\frac{6}{\gamma-1}&<&5\lab{17}\\
2\delta(4-\gamma)+\frac{5}{\gamma-1}&<&4. \lab{last}
\eea
\remove{
By~\eqn{0}, $1/2<1/(\gamma-1)<1$, $(2\gamma-3)/(\gamma-1)^2>1/(\gamma-1)$  and $(2\gamma-3)/(\gamma-1)^2<1$, and so~\eqn{9} implies~\eqn{4} and~\eqn{5},~\eqn{12} implies~\eqn{11},~\eqn{3} implies~\eqn{7},~\eqn{14} implies~\eqn{13} and is implied by~\eqn{3},~\eqn{8} implies~\eqn{15},~\eqn{16} implies~\eqn{10} and~\eqn{12}. Moreover, $7-2\gamma>4-\gamma$ and so~\eqn{16} implies~\eqn{8}. Using~\eqn{1}, we have that~\eqn{9} implies~\eqn{17} and~\eqn{16} implies~\eqn{last}. Using $1/(\gamma-1)<2/3$ by~\eqn{0}, we have that~\eqn{16} implies~\eqn{2}. Using $\delta(\gamma-2)+1/(\gamma-1)<\delta+1/(\gamma-1)<1$ by~\eqn{0} and~\eqn{1}, we have that~\eqn{16} implies~\eqn{9}. After removing redundant constraints~\eqn{2},~\eqn{5},~\eqn{7}--~\eqn{15},~\eqn{17},~\eqn{last}, system~\eqn{0}--\eqn{last} reduces to
}

By~\eqn{0}, $1/2<1/(\gamma-1)<1$, $(2\gamma-3)/(\gamma-1)^2>1/(\gamma-1)$  and $(2\gamma-3)/(\gamma-1)^2<1$, and so~\eqn{9} implies~\eqn{4} and~\eqn{5},~\eqn{12} implies~\eqn{11},~\eqn{3} implies~\eqn{7},~\eqn{14} implies~\eqn{13} and is implied by~\eqn{3},~\eqn{8} implies~\eqn{15}. Using~\eqn{1}, we have that~\eqn{9} implies~\eqn{17}. Next we show that~\eqn{8} implies~\eqn{9},~\eqn{2},~\eqn{10},~\eqn{12} and~\eqn{last}. Using $\gamma>2$ we have that~\eqn{8} implies~\eqn{12} and~\eqn{9} implies~\eqn{10}. By~\eqn{0}, $1/(\gamma-1)<2/3$ and so $3/(\gamma-1)<2$. Using this together with $\gamma<3$ we have that~\eqn{12} implies~\eqn{2}. Using~\eqn{1} that $\delta+1/(\gamma-1)<1$ we have that~\eqn{8} implies~\eqn{9}, and~\eqn{12} implies~\eqn{last}. Thus, we have proved that, using~\eqn{0} and~\eqn{1},~\eqn{8} implies~\eqn{9},~\eqn{2},~\eqn{10},~\eqn{12} and~\eqn{last}.  Moreover, since $5/2<\gamma<3$, it is easy to see that $(2\gamma-3)/(\gamma-1)^2<1$ and so~\eqn{8} implies~\eqn{16}. Since $\gamma<3$, we have $2/(\gamma-1)>1$ and so~\eqn{20} implies~\eqn{18}. Finally, it is easy to see that for all $5/2\le \gamma<3$, $1/(2\gamma-3)>2/3(\gamma-1)$ and so~\eqn{21} implies~\eqn{19}. 
 After removing the redundant constraints~\eqn{18}~--~\eqn{10} and~\eqn{11}~--~\eqn{last}, system~\eqn{0}--\eqn{last} reduces to the following

\bean
5/2&<&\gamma<3\\
\frac{4}{\gamma-1}-2&<&\delta(\gamma-2)\\
\frac{3-\gamma}{\gamma-2}&<&\delta\\
\frac{1}{2\gamma-3}&<&\delta<\frac12\\
\frac{1}{\gamma-1}&<&1-\delta\\
\delta(4-\gamma)+\frac{3}{\gamma-1}&<&2 \\
\delta(7-\gamma)&<&2\\
\delta(3-\gamma)+\frac{2}{\gamma-1}+\frac{2\gamma-3}{(\gamma-1)^2}&<&2,
\eean
 which leads to~\eqn{delta}.

\end{document}